\documentclass[a4paper,12pt]{article} 
\usepackage[english]{babel} 
\usepackage{amsmath,amssymb,amsthm,amsfonts} 
\usepackage[all]{xy}
\usepackage{xcolor}
\usepackage{graphicx}

\usepackage{hyperref}
\hypersetup{
  pdftitle={The equivariant local epsilon-constant conjecture for unramified twists of Zp(1)},
  pdfauthor={Werner Bley and Alessandro Cobbe}}

\newcommand{\Q}{\mathbb Q}
\newcommand{\N}{\mathbb N}
\newcommand{\Z}{\mathbb Z}
\newcommand{\C}{\mathbb C}

\newcommand{\p}{\mathfrak p}
\newcommand{\Gal}{\mathrm{Gal}}
\renewcommand{\epsilon}{\varepsilon}
\newcommand{\Kur}{K_{\mathrm{nr}}}
\newcommand{\Nur}{N_{\mathrm{nr}}}
\newcommand{\Lur}{L_{\mathrm{nr}}}
\newcommand{\Zpnr}{\mathbb Z_p^\mathrm{nr}}
\newcommand{\Qpab}{{\mathbb Q_p^\mathrm{ab}}}
\newcommand{\Qpram}{{\mathbb Q_p^\mathrm{ram}}}
\newcommand{\Qpnr}{{\mathbb Q_p^\mathrm{nr}}}
\newcommand{\tur}{t^\mathrm{nr}}
\newcommand{\omeganr}{\omega^\mathrm{nr}}
\newcommand{\omegaram}{\omega^\mathrm{ram}}
\newcommand{\calP}{\mathcal{P}}
\newcommand{\calM}{\mathcal M}
\newcommand{\calN}{\mathcal N}
\newcommand{\mm}{\mathfrak M}
\newcommand{\QGamma}{\Q[\Gamma]}
\newcommand{\T}{\mathcal T}
\newcommand{\chinr}{\chi^{\mathrm{nr}}}
\newcommand{\chiQpnr}{\chi^{\mathrm{nr}}_\Qp}
\newcommand{\BdR}{B_\mathrm{dR}}
\newcommand{\Bst}{B_\mathrm{st}}
\newcommand{\Bcris}{B_\mathrm{cris}}
\newcommand{\DdR}{D_\mathrm{dR}}
\newcommand{\Ucris}{U_\mathrm{cris}}
\newcommand{\ucris}{{u_\mathrm{cris}}}
\newcommand{\Dcris}{D_\mathrm{cris}}
\newcommand{\Ind}{\mathrm{Ind}}
\newcommand{\Irr}{\mathrm{Irr}}
\newcommand{\coker}{\mathrm{coker}}
\newcommand{\comp}{\mathrm{comp}}
\newcommand{\norm}{\mathcal{N}}
\newcommand{\trace}{\mathcal{T}}
\newcommand{\id}{\mathrm{id}}
\newcommand{\inv}{\mathrm{inv}}
\newcommand{\rec}{\mathrm{rec}}
\newcommand{\sseq}{\subseteq}
\newcommand{\Qp}{{\Q_p}}
\newcommand{\Zp}{{\Z_p}}
\newcommand{\ZpG}{{\Z_p[G]}}
\newcommand{\QpG}{{\Q_p[G]}}
\newcommand{\tensor}{\otimes}
\newcommand{\CEP}{C_{EP}^{na}(N/K, V)}
\newcommand{\DEPV}{\Delta_{EP}^{na}(N/K, V)}
\newcommand{\DEPT}{\Delta_{EP}^{na}(N/K, T)}
\newcommand{\lra}{\longrightarrow}
\newcommand{\lla}{\longleftarrow}
\newcommand{\calF}{\mathcal F}
\newcommand{\calL}{\mathcal{L}}
\newcommand{\calI}{\mathcal{I}}
\newcommand{\frd}{\mathfrak d}
\newcommand{\frD}{\mathfrak D}
\newcommand{\frf}{\mathfrak f}
\newcommand{\OK}{\mathcal{O}_K}
\newcommand{\ON}{\mathcal{O}_N}
\newcommand{\OL}{\mathcal{O}_L}
\newcommand{\DdRN}{\DdR^N}
\newcommand{\OKG}{{\OK[G]}}
\newcommand{\BdRG}{{\BdR[G]}}
\newcommand{\ra}{\rightarrow}

\newcommand{\calO}{\mathcal O}
\newcommand{\Opt}{\calO_p^t}
\newcommand{\Qpc}{\Q_p^c}
\newcommand{\QpcG}{\Q_p^c[G]}

\newcommand{\Hom}{\mathrm{Hom}}
\newcommand{\Aut}{\mathrm{Aut}}
\newcommand{\Map}{\mathrm{Map}}
\newcommand{\picprod}{\tensor}
\newcommand{\nr}{\mathrm{Nrd}}
\newcommand{\Nnr}{\calI_{N/K}(\chinr)}
\newcommand{\Lnr}{\calI_{L/K}(\chinr)}
\newcommand{\Mnr}{\calI_{M/K}(\chinr)}
\newcommand{\rk}{\mathrm{rk}}

\newcommand{\tfW}{\tilde f_{3,W}}
\newcommand{\fW}{f_{3,W}}
\newcommand{\homcont}{\Hom_\mathrm{cont}}
\newcommand{\Breu}{\mathrm{Breu}}
\newcommand{\PMod}{{\rm PMod}}
\newcommand{\unit}{{\bf{1}}}
\newcommand{\chiold}{\chi^\mathrm{old}}
\newcommand{\fFL}{f_{\calF,L}}
\newcommand{\fFN}{f_{\calF,N}}
\newcommand{\tfFL}{\tilde{f}_{\calF,L}}
\newcommand{\Qpx}{\Q_p^*}
\newcommand{\Qptimes}{\Q_p^\times}
\newcommand{\od}{\mathrm{od}}
\newcommand{\ev}{\mathrm{ev}}
\newcommand{\all}{\mathrm{all}}
\newcommand{\calG}{\mathcal G}

\newtheorem{thm_intro}{Theorem}

\newtheorem{thm}{Theorem}[subsection]
\newtheorem{lemma}[thm]{Lemma}
\newtheorem{coroll}[thm]{Corollary}
\newtheorem{prop}[thm]{Proposition}

\theoremstyle{remark}
\newtheorem{remark_intro}{Remark}

\newtheorem{remark}[thm]{Remark}

\newtheorem{conjecture}[thm]{Conjecture}

\theoremstyle{definition}

\title{The equivariant local $\epsilon$-constant conjecture \\ for unramified twists of $\Z_p(1)$}
\author{Werner Bley and Alessandro Cobbe}
\date{}
\begin{document}
\maketitle

\begin{abstract}
We study the equivariant local epsilon constant conjecture, denoted by $\CEP$, as formulated in various forms by
Kato, Benois and Berger, Fukaya and Kato and others, for certain $1$-dimensional twists $T = \Zp(\chinr)(1)$ of $\Zp(1)$.
Following ideas of recent work of Izychev and Venjakob we prove that for $T = \Zp(1)$ a conjecture of Breuning
is equivalent to $\CEP$. As our main result we show the validity of $\CEP$ for certain wildly and weakly ramified
abelian extensions $N/K$. A crucial step in the proof is the construction of an explicit  representative of
$R\Gamma(N, T)$.  

\end{abstract}

\begin{section}{Introduction}
Let $F/E$ denote a Galois extension of number fields and put $\Gamma := \Gal(F/E)$. We fix a motive $M$ which is defined over $E$ and admits an action of the group algebra $\QGamma$. Let $L(M, s)$ denote the equivariant complex $L$-function attached to $M$ and $F/E$ (see \cite[Sec.~4]{BuFl2001}). We shall always assume that this $L$-function satisfies a functional equation of the following form relating the $L$-functions attached to $M$ and its Kummer dual $M^*(1)$:
\[L(M,s) = \epsilon(M,s) \frac{L_\infty(M^*(1), -s)}{L_\infty(M,s)} L(M^*(1), -s).\]
Here $L_\infty$ is an Euler factor at infinity and the equivariant $\epsilon$-factor $\epsilon(M,s)$ decomposes into local factors 
\[\epsilon(M,s) = \prod_v \epsilon_v(M,s)\]
with $v$ running through the set of finite places of $E$. For some more details see \cite[Conj.~7]{BuFl2001}. 

In this manuscript we are interested in a local equivariant $\epsilon$-constant conjecture which relates $\epsilon_v(M,s)$ in a specific way to certain local cohomology groups of $R\Gamma(G_{E_v}, M_p)$ associated to the $p$-adic representation $V = M_p$. The validity of this $\epsilon$-constant conjecture for all finite places $v$ of $E$ is closely related to the compatibility of the equivariant Tamagawa number conjectures for $M$ and its Kummer dual $M^*(1)$ with the functional equation. Indeed, Conjecture \cite[Conj.~8]{BuFl2001} is a semi-local analogue of the local $\epsilon$-constant  conjecture and is directly related to the functional equation compatibility of the equivariant Tamagawa number conjecture, see \cite[Coroll.~1]{BuFl2001}.

So henceforth we fix a prime number $p$ and let $N/K$ denote a finite Galois extension of $p$-adic number fields with group $G := \Gal(N/K)$. We write $G_K$ (resp. $G_N$) for the absolute Galois group of $K$ (resp. $N$) and let $V$ denote a $p$-adic representation of $G_K$. Let $T \sseq V$ be a $G_K$-stable $\Zp$-sublattice such that $V = T \tensor_\Zp \Qp$.

Following seminal work of Fontaine and Perrin-Riou \cite{F-PR} several authors have formulated $\epsilon$-constant conjectures in this general context, see, e.g., Benois and Berger in \cite{BenBer} or Fukaya and Kato in \cite{FukKato}. For the relations between the different formulations of the conjecture we refer the reader to \cite{Iz}. We will use the notation of \cite[Conj.~5.1]{IV} and write $\CEP$ for the equivariant local $\epsilon$-constant conjecture. We will recall and formulate the conjecture in Section \ref{statement of conj}.

In a more specialized situation Breuning in \cite{Breuning04} has formulated an $\epsilon$-constant conjecture for $T = \Zp(1)$ in terms of relative algebraic $K$-groups. Extending ideas of \cite{BlBu03} he proved his conjecture, amongst other cases, for tamely ramified extensions $N/K$. Breuning's conjecture clearly has to be equivalent to $\CEP$ for $V = \Qp(1)$, see e.g. \cite[Appendix]{IV}. As a by-product of our work we will obtain a rigorous proof of this assertion.

Izychev and Venjakob consider in \cite{IV} the case where $T = \Zp(\chinr)(1)$ is a one-dimensional unramified twist of $\Zp(1)$. More precisely, if $\chinr$ is the restriction of an unramified character $\chiQpnr \colon G_\Qp \lra \Z_p^\times$ with $\chinr |_{G_N} \ne 1$, they reformulate $\CEP$ in the spirit of Breuning and adapt his proof of the tame case to show the validity of $\CEP$ for tamely ramified extensions $N/K$ and $V = \Qp(\chinr)(1)$. In the appendix of \cite{IV} they also indicate how to adapt Breuning's arguments to prove the tame case for unramified characters as above with $\chinr |_{G_N} = 1$.

We also recall that Benois and Berger in \cite[Th.~4.22]{BenBer} have proved the conjecture $\CEP$ for arbitrary crystalline representations $V$ of $G_K$ where $K/\Qp$ is unramified and $N$ is a finite subfield of $K(\mu_{p^\infty}) / K$. Here $\mu_{p^\infty}$ denotes the group of all $p$-power roots of unity. 

In this manuscript we will focus on weakly and wildly ramified extensions $N/K$ of an unramified extension $K/\Qp$. We recall that $N/K$ is weakly ramified if the second ramification group in lower numbering is trivial. As above we let $\chinr$ be an unramified character of $G_K$ which is the restriction of an unramified character $\chiQpnr \colon G_\Qp \lra \Z_p^\times$. We require no assumption on the restriction of $\chinr$ on $G_N$. The main result of our work is as follows.

\begin{thm_intro}\label{main theorem}
Let $p$ be an odd prime and let $K/\Qp$ be the unramified extension of degree $m$.
Let $N/K$ be a weakly and wildly ramified finite abelian extension with cyclic ramification group. Let $d$ denote the inertia degree of $N/K$ and assume that $m$ and $d$ are relatively prime. Then Conjecture $\CEP$ is true for $N/K$ and $V = \Qp(\chinr)(1)$.
\end{thm_intro}

\begin{remark_intro}
The assumptions of the theorem imply that the ramification group is cyclic of order $p$ (cf. \cite[Coroll.~3.4]{PickettVinatier}). 
More precisely, $|G| = pd, |G_0| = |G_1| = p$ and $|G_i| = 1$ for $i \ge 2$. Here $G_i$ for $i \ge 0$ denotes a higher ramification subgroup.
\end{remark_intro}
\begin{remark_intro}
The proof of Theorem \ref{main theorem} is an adaptation of the methods in \cite{BC} where we prove Breuning's conjecture for $T = \Zp(1)$ and extensions $N/K$ as in the theorem. If $\chinr |_{G_N} = 1$, it is almost straightforward to extend this result for $T = \Zp(\chinr)(1)$ because then twisting commutes with taking $G_N$-cohomology.
\end{remark_intro}
\begin{remark_intro}\label{remarkintro4}
If $\chinr |_{G_N}\ne 1$, the proof of the above weakly and wildly ramified case is much more involved than the tame case of Izychev and Venjakob because of the following fact. In the tame case the cohomology groups of $R\Gamma(N, T)$ turn out  to be perfect (see \cite[Prop.~2.1]{IV}) whereas this is in general no longer true in our weakly and wildly ramified case. This makes the computation of a certain refined Euler characteristic much more technically difficult.  
\end{remark_intro}

We recall that by \cite[Example~5.20]{Iz} the representation $T = \Zp(\chinr)(1)$ is naturally isomorphic to the $p$-adic Tate module $T_p\calF$ of a one-dimensional Lubin-Tate formal group $\calF$ defined over $\Zp$. Conversely, every one-dimensional Lubin-Tate formal group $\calF$ defined over $\Zp$ gives rise to an unramified twist of $\Zp(1)$ as above.

We let $\Kur$ denote the maximal unramified extension of $K$ and put $K_1 := \Kur \cap N$. We write $\Nur$ for the maximal unramified extension of $N$. Let $N_0$ denote the completion of $\Nur = N\Kur$ and write $\widehat{N_0^\times}$ for the $p$-completion of $N_0^\times$. We set
\[
\Nnr := \Ind_{\Gal(\Nur/K_1)}^{\Gal(\Nur/K)} \left( \widehat{N_0^\times} (\chinr) \right)
\]
and usually identify $\Nnr$ with $\prod_{i=1}^{d_{N/K}} \widehat{N_0^\times}$, where $d_{N/K}$ denotes the inertia degree of $N/K$. We will endow $\Nnr$ with a natural action of $\Gal(\Kur/K) \times G$. If $\chinr|_{G_N} \ne 1$ we set $\omega_N := v_p( 1 - \chinr(F_N))$ where $F_N$ denotes the Frobenius of $N$. Let $F=F_K$ be the Frobenius of $K$.

In order to deal with the difficulties indicated in Remark \ref{remarkintro4} above we construct an exact sequence
\begin{equation}\label{serre analogue}
0 \lra \calF(\p_N) \lra \Nnr  \xrightarrow{(F-1) \times 1} \Nnr \lra \Z / p^{\omega_L}\Z (\chinr) \lra 0
\end{equation}
and show the following theorem which may also be of independent interest.

\begin{thm_intro}\label{main theorem 2}
Let $N/K$ be a finite Galois extension of $p$-adic fields. Let $\chinr \colon G_K \lra \Z_p^\times$ be an unramified character as above
and assume that $\chinr|_{G_N} \ne 1$. Then the complex
\[C^\bullet_{N, T} := \left[ \Nnr\xrightarrow{(F-1)\times 1}\Nnr \right]\]
with non-trivial modules in degree 1 and 2 represents $R\Gamma(N,T)$.
\end{thm_intro}

\begin{remark_intro}
The exact sequence (\ref{serre analogue}) should be seen as an analogue of the fundamental short exact sequence of \cite[Exercises XIII, \S 5]{SerreLocalFields} which by \cite[Th.~4.20]{BreuPhd} is essentially a representative of $R\Gamma(G, \Zp(1))$. Note that in the case  $\chinr|_{G_N} \ne 1$ it does not seem to be possible to derive Theorem \ref{main theorem 2} from Serre's result by a simple twisting argument.
\end{remark_intro}

In Section \ref{determinants etc} we recall some preliminary results on $K$-theory, determinant functors and refined Euler characteristics which we will need throughout the manuscript. After these preparations we formulate in Section \ref{statement of conj} the conjecture $\CEP$ following the exposition of \cite{IV} and then give a short description of the organization of the paper. 

{\bf Notations: } We fix some standard notations which will be used throughout the manuscript. Given a finite field extension $M/L$, we will denote the norm and the trace  by $\norm_{M/L}$ and $\trace_{M/L}$ respectively. If $L/\Qp$ is a finite extension we write $L^c$ for the algebraic closure and set $G_L := \Gal(L^c/L)$. We let $\Lur$ denote the maximal unramified extension of $L$ and write $F_L$ for the arithmetic Frobenius of $L$. Note that $F_L$ is a topological generator of $\Gal(\Lur/L)$. As usual, we let $\OL$ denote the valuation ring of $L$ and write $\p_L = \pi_L\OL$ for the maximal ideal. 

We write $d_L$ for the absolute inertia degree and $e_L$ for the absolute ramification degree. If $M/L$ is a finite extension we let $d_{M/L}$ and $e_{M/L}$ denote the relative inertia and ramification degree, respectively. Then we have $d_M = d_{M/L}d_L$, $e_M = e_{M/L}e_L$ and $F_M = F_L^{d_{M/L}}$. We will denote the inertia group by $I_{M/L}$.

We let $L_0$ be the completion of $\Lur$ with respect to the $p$-adic topology. If $A$ is an abelian group we denote $p$-completion by $\widehat{A}$, i.e.,
\[\widehat{A} = \varprojlim_n A/p^nA.\]
Note that $\widehat A$ is in a natural way a $\Zp$-module. We will denote by $\nu_L \colon \widehat{L^\times} \tensor_\Zp \Qp \lra \Qp$ the map induced by the normalized valuation of $L^\times$.

If $\Sigma$ is any ring we denote its centre by $Z(\Sigma)$. 
\end{section}

\begin{section}{\texorpdfstring{$K$-theory, determinants and refined Euler characteristics}{K-theory, determinants and refined Euler characteristics}}
\label{determinants etc}
Let $\Lambda$ be any unital ring. Let $\PMod(\Lambda)$ denote the category of finitely generated projective $\Lambda$-modules and write $\PMod(\Lambda)^\bullet$ for the category of bounded complexes of such modules. We also write $D(\Lambda)$ for the derived category of complexes of $\Lambda$-modules and $D^{p}(\Lambda)$ for the full triangulated subcategory of $D(\Lambda)$ of perfect complexes. We recall that a complex $C^\bullet$ of $\Lambda$-modules is perfect, if and only if there exists a complex $P^\bullet \in \PMod(\Lambda)^\bullet$ and a quasi-isomorphism $P^\bullet \lra C^\bullet$. We say that a $\Lambda$-module $N$ is perfect, if the complex $N[0]$ belongs to $D^{p}(\Lambda)$.

Our main references for the theory of Picard categories, determinant functors and virtual objects are \cite[Sec.~2]{BuFl2001} and \cite{BrBuAdditivity}. For a more explicit approach we refer the reader to \cite[Sec.~1]{FukKato}. By Remark~1.2.10 of loc.cit. both approaches should be equivalent.

Let $V(\Lambda)$ denote the Picard category of virtual objects associated to $\PMod(\Lambda)$ and write $[ \cdot ]_\Lambda$ for the universal determinant functor
\[[ \cdot ]_\Lambda \colon \left( \PMod(\Lambda), is \right) \lra V(\Lambda),\]
where  $\left( \PMod(\Lambda), is \right)$ denotes the subcategory of all isomorphisms in $\PMod(\Lambda)$. By \cite[Prop.~2.1]{BuFl2001} this functor extends to a functor
\[[ \cdot ]_\Lambda \colon \left( D^{p}(\Lambda), is \right) \lra V(\Lambda).\]
We recall that $V(\Lambda)$ is equipped with a canonical bifunctor $(L, M) \mapsto L \picprod M$. We fix a unit object $\unit_{V(\Lambda)}$ and for each object $L$ an inverse $L^{-1}$ with an isomorphism $L \picprod L^{-1} \cong \unit_{V(\Lambda)}$. Each element of $V(\Lambda)$ is of the form $[P]_\Lambda \picprod [Q]_\Lambda^{-1}$ for modules $P, Q \in \PMod(\Lambda)$. Furthermore, $[P]_\Lambda$ and $[Q]_\Lambda$ are isomorphic in $V(\Lambda)$ if and only if the classes of $P$ and $Q$ in $K_0(\Lambda)$ coincide.

For any Picard category $\calP$ we define $\pi_0(\calP)$ to be the group of isomorphism classes of objects of $\calP$ and set $\pi_1(\calP) := \Aut_\calP(\unit_\calP)$. The groups $\pi_0(V(\Lambda))$ and $\pi_1(V(\Lambda))$ are naturally isomorphic to $K_0(\Lambda)$ and $K_1(\Lambda)$, respectively.

In the sequel we fix rings $\Lambda$ and $\Sigma$ and a ring homomorphism $\Lambda \lra \Sigma$ such that $\Sigma$ is flat as a right $\Lambda$-module. We assume that $\Sigma$ is noetherian and regular, that is every finitely generated $\Sigma$-module has finite projective dimension. We write $P_\Sigma = \Sigma \tensor_\Lambda P$ for the scalar extension map if $P$ is a $\Lambda$-module or a complex of $\Lambda$-modules. By abuse of notation we will write $[P]_\Sigma$ in place of $[\Sigma\otimes_\Lambda P]_\Sigma$. We recall that there is a canonical exact sequence of algebraic $K$-groups
\[K_1(\Lambda) \lra K_1(\Sigma) \xrightarrow{\partial^1_{\Lambda, \Sigma}} K_0(\Lambda, \Sigma) 
\xrightarrow{\partial^0_{\Lambda, \Sigma}} K_0(\Lambda) \lra K_0(\Sigma),\]
where $K_0(\Lambda, \Sigma)$ is the relative algebraic $K$-group defined by \cite[page~215]{Swan}. We recall also that elements in $K_0(\Lambda, \Sigma)$ are represented by triples $[P, g, Q]$ where $P, Q \in \PMod(\Lambda)$ and $g \colon P_\Sigma \lra Q_\Sigma$ is an isomorphism of $\Sigma$-modules.

If $\Sigma = L[G]$ for a finite group $G$ and a finite field extension $L/\Qp$ the reduced norm map induces an isomorphism
\[\nr_\Sigma \colon K_1(\Sigma) \lra Z(\Sigma)^\times\]
by \cite[Th.~(45.3)]{CRII}. The same is true for algebraically closed fields $L$, in particular for $L = \Qpc$.
We set $\hat\partial^1_{\Lambda, \Sigma} := \partial^1_{\Lambda, \Sigma} \circ \nr_{\Sigma}^{-1}$.

The scalar extension functor $\PMod(\Lambda) \lra \PMod(\Sigma)$ extends to a monoidal functor $V(\Lambda) \lra V(\Sigma)$ and we again write $L \mapsto L_\Sigma$ for $L \in V(\Lambda)$. Still following closely the exposition of \cite{BrBuAdditivity} we define a tensor category $V(\Lambda, \Sigma)$ as follows. Objects in $V(\Lambda, \Sigma)$ are pairs $(L, \lambda)$ with $L \in V(\Lambda)$ and $\lambda$ an isomorphism $L_\Sigma \lra \unit_{V(\Sigma)}$ in $V(\Sigma)$. A morphism $(L, \lambda) \lra (M, \mu)$ is an isomorphism $\alpha \colon L \lra M$ in $V(\Lambda)$ such that $\mu \circ \alpha_\Sigma = \lambda$. The product of $(L, \lambda)$ and $(M, \mu)$ is $(L \picprod M, \nu)$ where $\nu \colon (L \picprod M)_\Sigma \lra \unit_{V(\Sigma)}$ is the isomorphism
\[(L \picprod M)_\Sigma \lra L_\Sigma \picprod M_\Sigma \xrightarrow{\lambda \picprod \mu} 
\unit_{V(\Sigma)} \picprod  \unit_{V(\Sigma)} \lra   \unit_{V(\Sigma)}. \] 

By \cite[Lemma~5.1]{BrBuAdditivity} the category $V(\Lambda, \Sigma)$ is a Picard category and there is a natural isomorphism
\begin{equation}\label{yar 2}
i_{\Lambda, \Sigma} \colon \pi_0 \left( V(\Lambda, \Sigma)\right) \cong K_0(\Lambda, \Sigma).
\end{equation}

For a complex $P \in \PMod(\Lambda)^\bullet$ we define objects $P^\ev$ and $P^\od$ in $\PMod(\Lambda)$ by
\[P^\ev := \bigoplus_{i\ \mathrm{ even}} P^i, \quad P^\od := \bigoplus_{i\ \mathrm{ odd}} P^i.\]
Similarly we define
\[H^\ev(P) := \bigoplus_{i\ \mathrm{ even}} H^i(P), \quad H^\od(P) := \bigoplus_{i\ \mathrm{ odd}} H^i(P) \quad
H^\all(P) := \bigoplus_{i \in \Z} H^i(P).\]
Let $P \in D^p(\Lambda)$ and let $t \colon H^\ev(P_\Sigma) \lra H^\od(P_\Sigma)$ be an isomorphism of $\Sigma$-modules. We refer to $t$ as a trivialization and to the pair $(P, t)$ as a trivialized complex.

By \cite[Def.~5.5]{BrBuAdditivity} there is a (refined) Euler characteristic $\chi_{\Lambda, \Sigma}(P, t)\in K_0(\Lambda, \Sigma)$. We briefly recall the construction. The element $\chi_{\Lambda, \Sigma}(P, t)$ is defined by 
\[\chi_{\Lambda, \Sigma}(P, t)=i_{\Lambda, \Sigma}(([U]_\Lambda, \kappa(P,t))),\] 
where $([U]_\Lambda, \kappa(P,t)) \in V(\Lambda, \Sigma)$ is defined as follows. Let $a \colon U \lra P$ be a quasi-isomorphism with $U \in \PMod(\Lambda)^\bullet$ and let $\kappa(P,t) \colon [U]_\Sigma \lra \unit_{V(\Sigma)}$ be the isomorphism in $V(\Sigma)$ given by the composite
\begin{equation}\label{triv def}
\begin{split}
  [U]_\Sigma \xrightarrow{\eta_{U_\Sigma}} [H(U_\Sigma)]_\Sigma & \xrightarrow{H(a_\Sigma)}
[H(P_\Sigma)]_\Sigma \xrightarrow{\pi_{H(P_\Sigma)}} [H^\ev(P_\Sigma)[0]]_\Sigma \picprod [H^\od(P_\Sigma)[1]]_\Sigma \\
& \xrightarrow{[t] \picprod \id} [H^\od(P_\Sigma)[0]]_\Sigma \picprod [H^\od(P_\Sigma)[1]]_\Sigma
 \xrightarrow{\mu_{H^\od(P_\Sigma)[0]}} \unit_{V(\Sigma)}. 
\end{split}
\end{equation}
Here $\eta_{U_\Sigma}$ is defined in \cite[Prop.~3.1]{BrBuAdditivity}, $\pi_{H(P_\Sigma)}$ by \cite[Prop.~4.4]{BrBuAdditivity} and ${\mu_{H^\od(P_\Sigma)[0]}}$ by \cite[Lemma 2.3]{BrBuAdditivity}.

For computational applications it is often convenient to use an explicit construction of a refined Euler characteristic due to Burns in \cite{BuWhitehead}. We still consider a ring homomorphism $\Lambda \lra \Sigma$ as above, but now assume in addition that $\Sigma$ is semi-simple. We fix an object $P \in D^p(\Lambda)$ and a trivialisation $t \colon H^\ev(P_\Sigma) \lra H^\od(P_\Sigma)$. Following \cite{BrBuAdditivity} we write $\chiold_{\Sigma, \Lambda}(P, t^{-1}) \in K_0(\Lambda, \Sigma)$ for the refined Euler characteristic defined by Burns in \cite{BuWhitehead}. If we write $B^i(P_\Sigma) \sseq P^i_\Sigma$ for the boundaries of $P_\Sigma$ and set $B^\od(P_\Sigma) := \bigoplus_{i\ \mathrm{odd}}B^i(P_\Sigma)$, then by \cite[Th.~6.2]{BrBuAdditivity} we have
\begin{equation}\label{old and new}
-\chiold_{\Lambda, \Sigma}(P, t^{-1}) = \chi_{\Lambda, \Sigma}(P,t) + \partial^1_{\Lambda, \Sigma}((B^\od(P_\Sigma), -\id)).
\end{equation}
\end{section}

\begin{section}{The twisted local epsilon constant conjecture}\label{statement of conj}
In this section we formulate $\CEP$ following the exposition of \cite{IV}. We recall the necessary notations and results to the extend that we can relate $\CEP$ to Breuning's conjecture. For further details we refer the reader to \cite{BenBer} and \cite{IV}. 

\begin{subsection}{\texorpdfstring{The conjecture $\CEP$}{The conjecture CEPna(N/K,V)}}\label{subsectCEP}
We will briefly recall the definition of local $\epsilon$-constants following the expositions of \cite{BenBer} and \cite{IV}.

Let $E$ denote a large enough field of characteristic $0$ such that $E$ contains the $p$-power roots of unity $\mu_{p^\infty}$ and the roots of unity of order $p-1$. Let $K/\Qp$ be a finite extension. We write $\mu_K$ for the unique Haar measure on $K$ normalized by $\mu_K(\OK) = 1$. Let $\psi \colon K \lra E^\times$ denote an additive continuous character.

We recall that the theory of Langlands and Deligne (cf. \cite{Del73}) associates to each $E$-linear representation $V$ of the Weil group of $K$ a local constant $\epsilon(V, \psi, \mu_K)$. The basic properties of these $\epsilon$-constants are listed in \cite[Sec.~2.3]{BenBer}.

We fix a compatible system $\xi = \left( \xi_{p^n} \right)_{n \ge 1}$ of $p$-power roots of unity and define an additive character $\psi_\xi \colon \Qp \lra E^\times$ by $\psi_\xi |_\Zp = 1$ and $\psi_\xi(p^{-n}) = \xi_{p^n}$.

Let now $N/K$ be a finite Galois extension with group $G$. We assume that $E$ is large enough such that every irreducible character $\chi$ of $G$ can be realized over $E$. For each character $\chi$ of $G$ we fix an $E$-linear representation $V_\chi$ with character $\chi$ and write $V_\chi^*$ for the contragredient representation.

Following \cite[Sec.~3]{IV} we define the element
\[\epsilon_D(N/K, V) := \left( \epsilon\left( D_{pst}\left(\Ind_{K/\Qp}(V \tensor_E V_\chi^*), \psi_\xi, \mu_\Qp \right)\right)\right)_{\chi \in \Irr(G)}\]
in $\prod_{\chi \in \Irr(G)} (\Qpc)^\times \cong Z(\Qpc[G])^\times$. For all unexplained notation we refer the reader to \cite[Sec.~3]{IV}.

We will apply the notation introduced and explained in \cite[Sec.~1.1]{BenBer}. In particular, $\Bcris, \Bst$ and $\BdR$ denote the $p$-adic period rings constructed by Fontaine. If $V$ is a $p$-adic representation of $G_K$, we put
\[\DdR^K(V) := \left( \BdR \tensor_\Qp V \right)^{G_K}, \quad \Dcris^K(V) := \left( \Bcris \tensor_\Qp V \right)^{G_K}.\]
The $K$-vector space $\DdR^K(V)$ is finite dimensional and filtered. The tangent space of $V$ over $K$ is defined by
\[t_V(K) := \DdR^K(V) / \mathrm{Fil}^0 \DdR^K(V).\]
Finally, we write $\exp_V \colon t_V(K) \lra H^1(K, V)$ for the exponential map of Bloch and Kato. 

For any $\Qp$-vector space $W$ we write $W^* = \Hom_\Qp(W, \Qp)$ for its $\Qp$-linear dual. For convenience we usually write $t^*_V(K)$ instead of $t_V(K)^*$. 

Let $N/K$ be a finite Galois extension of $p$-adic fields with group $G$. We set $\Lambda := \ZpG$, $\Omega := \QpG$, $\tilde\Lambda := \overline{\Zpnr}[G]$ and $\tilde\Omega := \overline{\Qpnr}[G]$, where $\overline{\Zpnr}$ denotes the ring of integers in the completion $\overline{\Qpnr}$ of $\Qpnr$.

We recall the fundamental $7$-term exact sequence of $\QpG$-modules of \cite[(2.2)]{BenBer}
\begin{equation}\label{7terms}\begin{split}
0\lra &H^0(N,V)\lra\Dcris^N(V)\xrightarrow{1-\phi}\Dcris^N(V)\oplus t_V(N)\xrightarrow{\exp_V} H^1(N,V) \xrightarrow{\exp_V^*}\\
&\quad\lra \Dcris^N(V^*(1))^*\oplus t_{V^*(1)}^*(N)\lra\Dcris^N(V^*(1))^*\lra H^2(N,V)\lra 0.
\end{split}\end{equation}
In the following we will freely use the properties of determinant functors as, for example, formulated in \cite[Prop.~2.1]{BuFl2001}. As in \cite[Sec.~4]{IV} we use

\begin{equation}\label{tangent ses}
0 \lra t_{V^*(1)}^*(N) \lra \DdRN(V) \lra t_V(N) \lra 0
\end{equation}
and by \cite[Prop.~2.1]{BuFl2001} we obtain an isomorphism
\begin{equation}
  \label{cep1}
 \resizebox{\hsize}{!}{$[\DdRN(V)]_{\QpG} \picprod [R\Gamma(N,V)]_\QpG \to [t_{V^*(1)}^*(N)]_\QpG \picprod [t_V(N)]_\QpG \picprod\left(\bigotimes_i 
[H^i(N,V)]_\QpG^{(-1)^i}\right).$} 
\end{equation}
Using the commutativity constraint, the right hand side of (\ref{cep1}) is canonically isomorphic to 
\begin{equation}\label{cep2}
\begin{split}
& [H^0(N, V)]_{\QpG} \picprod [\Dcris^N(V)]^{-1}_\QpG \picprod  [\Dcris^N(V)]_\QpG \picprod  [t_V(N)]_\QpG 
\picprod [H^1(N, V)]^{-1}_\QpG \picprod \\
& [t_{V^*(1)}^*(N)]_\QpG \picprod [\Dcris^N(V^*(1))^*]_\QpG \picprod [\Dcris^N(V^*(1))^*]_\QpG^{-1} 
\picprod [H^2(N, V)]_\QpG. 
\end{split}\end{equation}

Now the $7$-term exact sequence (\ref{7terms}) induces an isomorphism from the object in (\ref{cep2}) to $\unit_{V(\QpG)}$, so that we obtain an isomorphism

\begin{equation}
  \label{cep3}
\delta'(N/K, V) \colon [\DdRN(V)]_\QpG \picprod [R\Gamma(N, V)]_\QpG \lra \unit_{V(\QpG)}.
\end{equation}
We put 
\[\DEPV := [R\Gamma(N, V)]_\QpG \picprod [\Ind_{N/\Qp}(V)]_\QpG \in V(\QpG).\]
By \cite[Lemme~2.13]{BenBer} we have an isomorphism $\psi \colon \DdRN(V) \lra \DdR^\Qp(\Ind_{N/\Qp}(V))$ which induces a canonical comparison isomorphism
\[\comp_V \colon \BdR \tensor_\Qp \DdRN(V) \lra \BdR \tensor_\Qp \Ind_{N/\Qp}(V), \quad c\tensor x \mapsto c\psi(x).\]
The morphism $\comp_V$ gives the isomorphism
\[\tilde\alpha(N/K, V)  \colon [\DdRN (V)]_\BdRG^{-1} \picprod [\Ind_{N/\Qp}(V)]_\BdRG \lra \unit_{V(\BdRG)}.\] 
Following \cite[Sec.~1.1]{BenBer} we put
\[h_i(V) := \dim_\Qp\left( \mathrm{Fil}^i \DdR^\Qp(V) /  \mathrm{Fil}^{i+1} \DdR^\Qp(V) \right)\]
and
\[t_H(V) := \sum_{i \in \Z} ih_i(V).\]
Furthermore, we set
\[\Gamma^*(V) := \prod_{i \in \Z} \Gamma^*(-i)^{h_i(V)},\]
where
\[\Gamma^*(i) =
\begin{cases}
  (i-1)! & \text{if } i > 0 \\ \frac{(-1)^i}{(-i)!} &\text{if } i \le 0.
\end{cases}\]
Multiplying $\tilde\alpha(N/K, V)$ by $t^{t_H(V)}$ where $t := \log [\xi] \in \BdR$ is the uniformizer of $\BdR^+$ associated with $\xi = \left( \xi_{p^n} \right)_{n \ge 0}$, we get
\begin{equation}\label{alphaV}
\alpha(N/K, V)  \colon [\DdRN (V)]_\BdRG^{-1} \picprod [\Ind_{N/\Qp}(V)]_\BdRG \lra \unit_{V(\BdRG)}.
\end{equation} 
Composing $\alpha(N/K, V)$ with the automorphism of $\unit_{V(\BdRG)}$ induced by the element $\Gamma^*(V) \epsilon_D(N/K, V)$ we obtain
\begin{equation}\label{betaV}
\beta(N/K, V)  \colon  [\DdRN (V)]_\BdRG^{-1} \picprod  [\Ind_{N/\Qp}(V)]_\BdRG \lra \unit_{V(\BdRG)}.
\end{equation}
We also define $\beta'(N/K,V)$ as the following composite
\begin{equation}\label{defbeta'}
\begin{split}
&  \DEPV_\BdRG\\
&\quad= [R\Gamma(N, V)]_\BdRG \picprod [\Ind_{N/\Qp}(V)]_\BdRG \\
&\quad= [R\Gamma(N, V)]_\BdRG \!\picprod\!  [\DdRN(V)]_\BdRG \picprod [\DdRN(V)]_\BdRG^{-1}\picprod [\Ind_{N/\Qp}(V)]_\BdRG\\
&\quad\xrightarrow{\delta'(N/K, V)} [\DdRN(V)]_\BdRG^{-1} \picprod   [\Ind_{N/\Qp}(V)]_\BdRG \\
&\quad\xrightarrow{\alpha(N/K, V)} \unit_{V(\BdRG)}.
\end{split}\end{equation}

Note that in order to apply $\delta'(N/K,V)$ in the above definition it is necessary to apply the commutativity constraint first.

The isomorphism $\delta(N/K, V)$ is defined analogously, just writing $\beta(N/K,V)$ in place of $\alpha(N/K,V)$ in the above formula. Since according to \cite[Lemme~2.15]{BenBer} the isomorphism $\beta(N/K, V)$ is actually defined over $\tilde\Omega$, the same is true for $\delta(N/K, V)$, so that we actually have an isomorphism
\begin{equation}\label{yar 3}
\delta(N/K, V)\colon\DEPV_{\tilde\Omega}\lra\unit_{V({\tilde\Omega})}.
\end{equation}

We define
\[\DEPT :=  [R\Gamma(N, T)]_\ZpG \picprod  [\Ind_{N/\Qp}(T)]_\ZpG \in V(\ZpG).\]
The element $\DEPT$ actually does  not depend on the choice of $T$ and is therefore a canonical element in $V(\ZpG)$ whose base change gives $\DEPV$. We can now finally formulate the epsilon constant conjecture (compare to \cite[Conj.~5.1]{IV} or \cite[Conj.~2.19]{BenBer}).

\begin{conjecture}\label{Ca}
If $V = \Qp \tensor_\Zp T$ is potentially semistable and if $N/K$ is a finite Galois extension, then the class $[ (\DEPT, \delta(N/K, V)) ] \in \pi_0(V(\ZpG, \tilde\Omega))$ is trivial in $\pi_0(V(\tilde\Lambda, \tilde\Omega))$.
\end{conjecture}

\end{subsection}

\begin{subsection}{\texorpdfstring{$\CEP$ for unramified twists of $\Zp(1)$}{CEPna(N/K,V) for unramified twists of Zp(1)}}\label{sit}
We specialize now to the situation considered in this manuscript. 

Let $p$ be a prime and let $K$ be a finite extension of $\Qp$. Let $\chiQpnr \colon G_\Qp \lra \Z_p^\times$ be a continuous unramified character and write $\chinr := \chiQpnr |_{G_K}$ for its restriction to $G_K$. We write $\chi^{cyc} \colon G_K \lra \Z_p^\times$ for the cyclotomic character and consider the $p$-adic representation
\[T := \Zp(\chinr)(\chi^{cyc}) =  \Zp(\chinr)(1).\]
Explicitly, $T$ is $\Zp$ as an abelian group equipped with the $G_K$-action defined by $gx := \chinr(g)\chi^{cyc}(g)x$ for $g \in G_K$ and $x \in \Zp$. We put $V := \Qp \tensor_\Zp T$.

Let $\varphi = F_\Qp$ denote the absolute Frobenius automorphism and set $u := \chiQpnr(\varphi)$. Let $\calF$ denote the Lubin-Tate formal group associated with $\pi := up$. By \cite[Example~5.20]{Iz} the representation $T$ is the restriction to $G_K$ of the $p$-adic Tate module $T_p\calF$. 

Let $N/K$ be a finite Galois extension with group $G$. As usual we write $\calF(\p_N)$ for the formal $\Zp$-module $\p_N$ with the $\Zp$-module structure induced by the formal group law $\calF$.

We define the complex $R\Gamma(N, T)$ as the $G_N$-invariants of the continuous standard resolution of $T$ (cf. \cite[Ch.~I, \S 2]{NSW}). As usual, we write $H^i(N, T)$ for the cohomology groups of $R\Gamma(N, T)$. We recall some well-known facts about these cohomology groups. We write $v_p \colon \Qp \lra \Z \cup \{\infty\}$ for the normalized valuation map and set
\begin{equation}\label{omega def}
\omega = \omega_N := v_p(1 - \chinr(F_N)).
\end{equation}
Note that $\omega \in \Z$ if and only if $\chinr|_{G_N} \ne 1$.

\begin{lemma}\label{galois cohomology}
Assume the notation of Subsection \ref{sit}.
\begin{enumerate}

\item[(a)] If $\chinr |_{G_N} \ne 1$, then
\begin{enumerate}
\item[(i)] $H^1(N, T) \cong \left( \widehat{N_0^\times}(\chinr) \right)^{\Gal(N_0/N)} \cong \calF(\p_N)$,
\item[(ii)] $H^2(N, T) \cong \left( \Zp/p^\omega\Zp \right) (\chinr)$ , 
\item[(iii)] $H^i(N, T) = 0 \text{ for } i \ne 1,2$.
\end{enumerate}

\item[(b)] If $\chinr |_{G_N} = 1$, then
\begin{enumerate}
\item[(i)] $H^1(N, T) \cong \widehat{N^\times}(\chinr)$,
\item[(ii)] $H^2(N, T) \cong \Zp (\chinr)$,
\item[(iii)] $H^i(N, T) = 0 \text{ for } i \ne 1,2$.
\end{enumerate}

\end{enumerate}
\end{lemma}

\begin{proof}
Part (a) is shown in \cite[Prop.~2.1]{IV}.
If  $\chinr |_{G_N} = 1$, then taking cohomology commutes with twisting, i.e., $H^i(N, A(\chinr)) = H^i(N, A)(\chinr)$ for each $G_N$-module $A$. Hence part (b) is immediate from \cite[Prop.~4.11]{BreuPhd}.
\end{proof}

In Section \ref{RGamma(N,T)} we will give an alternative proof of part (a) by explicitly constructing a representative for $R\Gamma(N, T)$. For details see Remark \ref{our coh proof}.

\begin{remark}
In \cite{BC} and \cite{Chinburg85} the symbol $N_0$ denotes the maximal unramified extension of $N$ which in
the present manuscript is denoted by $\Nur$. In order to relate the results and computations of \cite{BC}
to the results and computations of this manuscript it is therefore worth noting that $\widehat{\Nur^\times}=\widehat{N_0^\times}$.
Indeed, we have for each  natural number $n$ the canonical short exact sequence 
\[
0 \lra \Nur^\times / \left(\Nur^\times \cap (N_0^\times)^{p^n}\right) \lra N_0^\times / ( N_0^\times)^{p^n} \lra  N_0^\times / \Nur^\times ( N_0^\times)^{p^n}
\lra 0.
\]
It is an easy exercise to show that $ \Nur^\times ( N_0^\times)^{p^n}= N_0^\times$ and $\Nur^\times \cap (N_0^\times)^{p^n} = (\Nur^\times)^{p^n}$.
Hence we have a natural isomorphism $\Nur^\times / (\Nur^\times)^{p^n} \cong N_0^\times / ( N_0^\times)^{p^n}$.
\end{remark}

\begin{remark}
In our situation we have 
\[h_i(V) =
\begin{cases}
  1 &\text{if } i = -1 \\ 0   &\text{if } i \ne -1. 
\end{cases}\]
Indeed, by \cite[page~148]{FO} or \cite[page~5]{IV}, $t_{V^*(1)}(\Qp)\hookrightarrow \C_p((\chinr)^{-1})(-1)^{G_N}=0$ and hence $t_{V^*(1)}(\Qp) = 0$. Then the short exact sequence (\ref{tangent ses}) implies that $\DdR^\Qp(V) = t_V(\Qp)$, i.e. $\mathrm{Fil}^{0}\DdR^\Qp(V) = 0$. Furthermore the element $e_{\chinr_{\Qp,1}}=\tur\cdot t^{-1}\otimes (v\otimes \xi)$ defined in \cite[page~11]{IV} is a nonzero element in $\mathrm{Fil}^{-1}\DdR^\Qp(V)$. Since $\dim_\Qp(\DdR^\Qp(V)) = 1$, we deduce that $\mathrm{Fil}^{-1}\DdR^\Qp(V) = \DdR^\Qp(V)$, so that there is only one jump in the filtration, namely at $i=-1$.

Hence we obtain $t_H(V) = -1$ and $\Gamma^*(V) = \Gamma^*(1) = 1$.
\end{remark}
In this case we shall now reformulate Conjecture \ref{Ca} in the language of relative algebraic $K$-groups and refined Euler characteristics.

Following \cite{IV} and \cite{Breuning04} we define an element  $\tilde R_{N/K} \in K_0(\ZpG, \tilde\Omega)$ by
\[\tilde R_{N/K}=C_{N/K}+\Ucris+\partial^1_{\ZpG, \BdR[G]}(t)+\partial^1_{\ZpG, \BdR[G]}(\epsilon_D(N/K,V)),\]
where each of the terms is an element in $K_0(\ZpG,\BdR[G])$. After briefly recalling the definition of $C_{N/K}$ we give the formulation of the conjecture due to Izychev and Venjakob, and then, following the approach of Breuning, modify $\tilde R_{N/K}$ by the so-called unramified term in order to obtain an element $R_{N/K} \in K_0(\ZpG, \QpG)$. In Section \ref{eps} we will write $\epsilon_D(N/K,V)$ in terms of Galois Gau\ss\ sums in order to tie up to the notion of $\epsilon$-constants as it is used in Breuning's conjecture \cite[Conj.~3.2]{Breuning04}.

The element $C_{N/K} \in K_0(\ZpG,\tilde\Omega)$ is defined by

\begin{equation}\label{CNK}
C_{N/K}=-\chi_{\ZpG,\BdRG}(M^\bullet,\lambda^{-1})
\end{equation}
where the trivialized complex $(M^\bullet, \lambda^{-1})$ is given by
\begin{equation}\label{Mbullet def}
M^\bullet=R\Gamma(N,T)\oplus \Ind_{N/\Qp}T[0]
\end{equation}
and 
\begin{eqnarray*}
\lambda &\colon& H^\od(M^\bullet)_{\BdR[G]} \lra H^\ev(M^\bullet)_{\BdR[G]}, \\
\lambda &=&
\begin{cases}
 \comp\circ\exp_V^{-1} & \text{ if } \chinr|_{G_N} \ne 1 \\
  \comp\circ\exp_V^{-1}+ \nu_N & \text{ if } \chinr|_{G_N} = 1.
\end{cases}
\end{eqnarray*}
Note that the minus sign in the definition of $C_{N/K}$ is due to a different convention about Euler characteristics in comparison to \cite{IV}.
 
The explicit computation of $C_{N/K}$ is the technical heart of this paper. The computation is based on the construction of an explicit representative for $R\Gamma(N,T)$ (see Section \ref{RGamma(N,T)}) which is of independent interest. Our representative should be considered as the analogue of Serre's explicit representative of the local fundamental class (see \cite[Prop.~6.1]{Chinburg85} or \cite[Th.~4.20]{BreuPhd}).

The term $\Ucris$ will be defined in Section \ref{Ucris}. In the case $\chinr|_{G_N} \ne 1$ our definition agrees with the one given in \cite[Sec.~5]{IV}. In Section \ref{Ucris} we will also make the comparison to the correction term $M_{N/K}$ which occurs in (the generalization of) Breuning's Conjecture 3.2 in \cite{Breuning04}.

\begin{conjecture}\label{Cb}  
Under the current assumptions $\tilde R_{N/K} = 0$ in $K_0(\tilde\Lambda, \tilde\Omega)$.
\end{conjecture}
In the case $\chinr|_{G_N} \ne 1$ this is precisely the reformulation of 
Conjecture \ref{Ca} as stated in \cite[(5.1)]{IV}.

We now follow the approach of \cite{Breuning04} to define a modified element $R_{N/K}$ in $K_0(\ZpG,\QpG)$. We write $\Opt$ for the ring of integers in the maximal tamely ramified
extension of $\Qp$ in $\Qpc$. Let $\iota \colon K_0(\ZpG, \Qpc[G]) \lra K_0(\Opt[G], \Qpc[G])$ be the natural scalar extension map. We recall that by Taylor's fixed point theorem the restriction of $\iota$ to the subgroup $K_0(\ZpG, \QpG)$ is injective. Let us define $U_{N/K}$ as in \cite{Breuning04}, then by \cite[Prop.~2.12]{Breuning04} we have $\iota(U_{N/K}) = 0$.  We define
\begin{equation}\label{defRNK}
R_{N/K}=C_{N/K} + \Ucris+\partial^1_{\ZpG, \BdR[G]}(t) - U_{N/K} + \partial^1_{\ZpG, \BdR[G]}(\epsilon_D(N/K,V)),
\end{equation}
so that $R_{N/K} = \tilde R_{N/K} -  U_{N/K}$. Using some of the properties of $U_{N/K}$ with respect to the action of $\Gal(\Qpc / \Qp)$ one can prove that $R_{N/K} \in K_0(\ZpG, \QpG)$. We will do this in Subsection \ref{CNK 1} and we will also prove that if $T = \Zp(1)$, then the following conjecture is precisely Breuning's Conjecture \cite[Conj.~3.2]{Breuning04}.

\begin{conjecture}\label{Cc}
Under the current assumptions $R_{N/K} = 0$ in $K_0(\ZpG,\QpG)$.
\end{conjecture}

\begin{prop}\label{Conj equiv}
In the present setting the Conjectures \ref{Ca}, \ref{Cb} and \ref{Cc} are equivalent.
\end{prop}
The proof will be given in Subsection \ref{CNK 1}. In the rest of the manuscript we write $\CEP$ for the equivalent conjectures \ref{Ca}, \ref{Cb} and \ref{Cc}.
\end{subsection}

\begin{subsection}{Organization of the manuscript}
To end this section we briefly describe the organization of the manuscript. In Section \ref{RGamma(N,T)} we construct an explicit representative of $R\Gamma(N, T)$. Section \ref{Ucris} is dedicated to the definition of the term $\Ucris$ and the comparison of $\Ucris$ to Breuning's correction term $M_{N/K}$. In the following Section \ref{eps} we discuss epsilon constants and express them in terms of Galois Gau\ss\  sums. This will allow us to compare the term $\partial^1_{\ZpG, \BdR[G]}(\epsilon_D(N/K, V))$ to Breuning's element $T_{N/K}$. In Section \ref{computationCNK} we express $C_{N/K}$ as a sum of a norm resolvent and the refined Euler characteristic of $R\Gamma(N,T)$. 
With these preparations at hand we study the Galois action of the individual terms in the definition of $R_{N/K}$, prove Proposition \ref{Conj equiv} and compute the refined Euler characteristic $C_{N/K}$. As a by-product we obtain the proof of the equivalence of Breuning's Conjecture \cite[Conj.~3.2]{Breuning04} and $\CEP$ if $T = \Zp(1)$.
Finally in Section \ref{proof of theorem} we provide the proof of Theorem \ref{main theorem}.
\end{subsection}
\end{section}

\begin{section}{\texorpdfstring{A representative for $R\Gamma(N,T)$}{A representative for R Gamma(N,T)}}\label{RGamma(N,T)}
In this section we construct an explicit representative for the complex $R\Gamma(N,T)$.
We assume the situation described at the beginning of Subsection \ref{sit}.

\begin{subsection}{Some preliminary results}\label{preliminary results}
\begin{lemma}\label{UL1hat}
Let $L/K$ denote a finite Galois extension. Then we have a canonical exact sequence
\[0 \lra \calF(\p_L) \lra \left( \widehat{L_0^\times}(\chinr) \right)^{\Gal(L_0/L)} 
\stackrel{\nu_L}\lra \left( \Zp(\chinr) \right)^{\Gal(L_0/L)} \]
of $I_{L/K}$-modules. In particular, if $\chinr|_{G_L} \ne 1$, then there is an isomorphism
\[\fFL\colon\calF(\p_L) \lra \left( \widehat{L_0^\times}(\chinr) \right)^{\Gal(L_0/L)}.\]
Explicitly $\fFL$ is given by a power series of the form $1+\epsilon^{-1} X+\deg\geq 2$, where $\epsilon\in \overline{\Zpnr}^\times$ is such that $u=\epsilon^{\varphi-1}$.
\end{lemma}

\begin{proof}
We have $L_0^\times \cong \pi_L^\Z \times \kappa^\times \times U_{L_0}^{(1)}$, where $\kappa$ denotes the residue class field of $L_0$. Since any element of $\kappa^\times$ has coprime to $p$ order and $U_{L_0}^{(1)}$ is $p$-complete we obtain $\widehat{L_0^\times} \cong \Zp \times U_{L_0}^{(1)}$, and hence  a canonical short exact sequence
\begin{equation}\label{val ses}
  0 \lra U_{L_0}^{(1)} \lra \widehat{L_0^\times} \lra \Zp \lra 0.
\end{equation}
From \cite[Lemma~page~237]{LubinRosen} we derive the existence of an isomorphism
\[\tilde f\colon\calF(\p_L)\lra \left( U_{L_0}^{(1)}(\chinr) \right)^{\Gal(L_0/L)}.\]
Up to the obvious isomorphism $U_{L_0}^{(1)}\cong G_m(\p_{L_0})$, this isomorphism is given by the uniquely defined power series $\theta(X)=\epsilon^{-1} X+\deg\geq 2$ of \cite[Koroll.~V.2.3]{Neukirch92}.

The result follows by twisting and taking $\Gal(L_0/L)$-fixed points in (\ref{val ses}).
\end{proof}

We set $\calG := \Gal(K_0/K) \times \Gal(L/K)$. In analogy to \cite[(6.2)]{Chinburg85} we set
\[\Lnr := \Ind_{\Gal(L_0/K)}^\calG  \left( \widehat{L_0^\times}(\chinr) \right)\]
and we usually identify $\Lnr$ with $\left( \widehat{L_0^\times} \right)^{d_{L/K}}$. Recalling that $F = F_K$ denotes the Frobenius automorphism of $K$, note that the $\calG$-module structure on $\Lnr$ is characterized by
\begin{equation}\label{rule 1}
(F\times 1)\cdot[x_1,x_2,\dots,x_{d_{L/K}}]=[x_{d_{L/K}}^{F_L\chinr(F_L)},x_1,x_2,\dots,x_{d_{L/K}-1}],
\end{equation}
and
\begin{equation}\label{rule 2}
(F^{-n}\times \sigma)\cdot[x_1,x_2,\dots,x_{d_{L/K}}]=[x_1^{\tilde\sigma\chinr(\tilde\sigma)},
x_2^{\tilde\sigma\chinr(\tilde\sigma)},\dots,x_{d_{L/K}}^{\tilde\sigma\chinr(\tilde\sigma)}],
\end{equation}
where $F^{-n}$ and $\sigma\in \Gal(L/K)$ have the same restriction to $L \cap K_0$ and $\tilde\sigma \in \Gal(L_0/K)$ is uniquely defined by $\tilde\sigma|_{K_0}=F^{-n}$ and $\tilde\sigma|_L=\sigma$. 

We also have a $G_K$- and a $\Gal(L/K)$-action on $\Lnr$, via the natural maps $G_K\to\Gal(L_0/K)\hookrightarrow\calG$ and $\Gal(L/K) \hookrightarrow \calG$. Note that in the case $\chinr|_{G_L} = 1$ we have
\begin{equation}\label{twist trivial}
 \Lnr=\Ind_{\Gal(L_0/K)}^{\calG}  \left( \widehat{L_0^\times}(\chinr) \right)
\cong
 \left( \Ind_{\Gal(L_0/K)}^{\calG}   \widehat{L_0^\times}\right)(\chinr)
\end{equation}
both as $\Gal(L/K)$- and $G_K$-modules. The following lemma clarifies the structure of $\Lnr$ as a $\Gal(L/K)$-module in the general situation.

\begin{lemma}\label{LurInd}
If we identify the inertia group $I_{L/K}$ with $\Gal(L_0/K_0)$, then there is a $\Gal(L/K)$-isomorphism
\[\tfFL \colon \Ind_{I_{L/K}}^{\Gal(L/K)}\widehat{L_0^\times} \lra \Lnr.\]
In particular, $\Lnr$ is $\Gal(L/K)$-cohomologically trivial.
\end{lemma}

\begin{proof}
We identify $\Ind_{I_{L/K}}^{\Gal(L/K)}\widehat{L_0^\times}$ with $\Z_p[\Gal(L/K)]\otimes_{\Z_p[I_{L/K}]}\widehat{L_0^\times}$ and set
\[\tfFL(\sigma\otimes x) := (1\times\sigma)\cdot[x,1,\dots,1],\]
for $x\in\widehat{L_0^\times}$ and $\sigma\in\Gal(L/K)$. To check that this is well-defined we take $\tau\in I_{L/K}$ and write $\tilde\tau$ for the corresponding element in $\Gal(L_0/K_0)$. Then
\[\begin{split}
\tfFL(\sigma\tau\otimes x)
&=(1\times\sigma\tau)\cdot [x,1,\dots,1]=(1\times\sigma)(1\times\tau)\cdot [x,1,\dots,1]\\
&=(1\times\sigma)\cdot [x^{\tilde\tau\chinr(\tilde\tau)},1,\dots,1]=(1\times\sigma)\cdot [x^{\tilde\tau},1,\dots,1]\\
&=\tfFL(\sigma\otimes x^{\tilde\tau}).
\end{split}\]
To prove surjectivity, we check that all elements of the form $[1,\dots,1,x,1,\dots,1]$ are in the image of $\tfFL$. Let $i+1$ be the index of $x$ and let us take $\sigma\in\Gal(L/K)$ such that $\sigma|_{L\cap K_0}=F^{-i}$. Then
\[\begin{split}
\tfFL(\sigma\otimes x^{\tilde\sigma^{-1}\chinr(\tilde\sigma)^{-1}})&
=(1\times\sigma)\cdot[x^{\tilde\sigma^{-1}\chinr(\tilde\sigma)^{-1}},1,\dots,1]\\
&=(F^i\times 1)(F^{-i}\times\sigma)\cdot[x^{\tilde\sigma^{-1}\chinr(\tilde\sigma)^{-1}},1,\dots,1]\\
&=(F^i\times 1)\cdot[x,1,\dots,1]=[1,\dots,1,x,1,\dots,1].
\end{split}\]
It remains to prove injectivity. Let us take $\prod_i \sigma_i\otimes x_i\in\ker \tfFL$. Clearly we can assume that $0 \le i \le d_{L/K}-1$ and $\sigma_i|_{L\cap K_0}=F^{-i}$. Then we calculate
\[\tfFL\left(\prod_{i=0}^{d_{L/K}-1} \sigma_i\otimes x_i\right)=
[x_0^{\tilde\sigma_0\chinr(\tilde\sigma_0)},x_1^{\tilde\sigma_1\chinr(\tilde\sigma_1)},\dots, x_{d_{L/K}-1}^{\tilde\sigma_{d_{L/K}-1}\chinr(\tilde\sigma_{d_{L/K}-1})}].\]
Since we are assuming that the element is in the kernel, we deduce that all of the $x_i$ are equal to $1$. Hence $\tfFL$ is injective. The isomorphism is $\Gal(L/K)$-invariant by construction.

By \cite[Prop.~XIII.5.14]{SerreLocalFields} the $I_{L/K}$-module $L_0^\times$ is $I_{L/K}$-cohomologically trivial.
Using \cite[Lemma~4.9]{BreuPhd} the same holds true for the $p$-completion $\widehat{L_0^\times}$. By Shapiro's Lemma
we finally derive that $\Ind_{I_{L/K}}^{\Gal(L/K)}\widehat{L_0^\times}$ is $\Gal(L/K)$-cohomologically trivial.
\end{proof}

We slightly generalize a well-known lemma, see e.g. \cite[Lemma~V.2.1]{Neukirch92}.

\begin{lemma}\label{Neukirch}
Given $c\in \widehat{U_{L_0}}$ there exists $x\in \widehat{U_{L_0}}$ such that $x^{\chinr(F_L)F_L-1}=c$. If $x_n$ is such a solution modulo $U_{L_0}^{p^n}$, then we can assume that $x\equiv x_n\pmod{U_{L_0}^{p^n}}$. In particular if $c\in U_{L_0}^{(1)}$, then we can assume $x\in U_{L_0}^{(1)}$.
\end{lemma}

\begin{proof}
Let $n\in\N$, let $\psi_n\in\Z$ be such that $\psi_n\equiv\chinr(F_L)\pmod{p^n}$. First of all we want to determine an $x_n\in U_{L_0}/U_{L_0}^{p^n}$ such that $x_n^{\chinr(F_L) F_L-1}\equiv x_n^{\psi_n F_L-1}\equiv c\pmod{U_{L_0}^{p^n}}$. Let $c_n\in U_{L_0}$ be congruent to $c$ modulo $U_{L_0}^{p^n}$, then it is enough to find $x_n\in U_{L_0}$ such that $x_n^{\psi_n F_L-1}=c_n$. The same proof as in \cite[Lemma~V.2.1]{Neukirch92} works, simply writing $F_L\psi$ instead of $\varphi$ everywhere. (Note that here we are using that $L_0$ is complete.) Moreover, if we assume that we already have an $x_{n-1}\in U_{L_0}/U_{L_0}^{p^{n-1}}$ such that $x_{n-1}^{\chinr(F_L) F_L-1}\equiv x_{n-1}^{\psi_n F_L-1}\equiv c\pmod{U_{L_0}^{p^{n-1}}}$, then we can find $x_n$ with the additional property that $x_n\equiv x_{n-1}\pmod{U_{L_0}^{p^{n-1}}}$.

By passing to the limit we obtain an element $x\in \widehat{U_{L_0}}$ such that $x^{\chinr(F_L)F_L-1}=c$. The remaining assertions are now straightforward.
\end{proof}
\end{subsection}

\begin{subsection}{The construction of the complex}
We continue to assume the situation of Subsection \ref{sit}.
\begin{thm}[Serre]\label{Serrethm}
If $\chinr(F_L)=1$, then we have an exact sequence
\[0\lra \widehat{L^\times}(\chinr)\lra\Lnr\lra\Lnr\lra\Z_p(\chinr)\lra0\]
of $\Zp[\Gal(L/K)]$-modules.
\end{thm}

\begin{proof}
Taking into account the isomorphism (\ref{twist trivial}) this follows directly 
from  \cite[Prop.~6.1]{Chinburg85} by $p$-completing and twisting.
\end{proof}

Now we want an analogous result for the case $\chinr(F_L)\neq 1$.

We set 
\[\omega_L := v_p(1 - \chinr(F_L)),\]
which is finite when $\chinr(F_L)\neq 1$.

For $n>\omega_L$ we put
\[ V_n=\{y\in L_0^\times/(L_0^\times)^{p^n}:\ \nu_L(y)\equiv 0\pmod{p^{\omega_L}}\}\]
and
\[W_n=\{y\in L_0^\times/(L_0^\times)^{p^n}:\ y^{\chinr(F_L)F_L-1}=1\}.\]
Note that when $\omega_L=0$ the congruence in the definition of $V_n$ is always trivially satisfied, so that $V_n=L_0^\times/(L_0^\times)^{p^n}$.

\begin{lemma}\label{Neu comm diag}
Assume $\chinr(F_L)\neq 1$. If $m \ge n \ge \omega_L$, then we have a commutative diagram of exact sequences:
\[\xymatrix{0\ar[rr]&&W_{m}\ar[rr]\ar[d]^{\tau_{m,n}}&&L_0^\times/(L_0^\times)^{p^{m}}
\ar[rr]^-{\chinr(F_L)F_L-1}\ar[d]^{\pi_{m,n}}   &&V_{m}\ar[rr]\ar[d]^{\sigma_{m,n}}&&0\\
0\ar[rr]&&W_{n}\ar[rr]&&L_0^\times/(L_0^\times)^{p^{n}}\ar[rr]^-{\chinr(F_L)F_L-1}&&V_{n}
\ar[rr]&&0.}\]
Here $\pi_{m,n}$ denotes the canonical projection and $\tau_{m,n}$ and $\sigma_{m,n}$ are induced by $\pi_{m,n}$. Moreover, $\pi_{m,n}$ and $\sigma_{m,n}$ are surjective and the projective system $\left( W_n \right)_n$ satisfies the Mittag-Leffler condition.
\end{lemma}

\begin{proof}
To prove the exactness of the rows we must show that
\[L_0^\times/(L_0^\times)^{p^{n}}\xrightarrow{\chinr(F_L)F_L -1}V_n\] 
is surjective. If $c \in V_n$, then $\nu_L(c)\in p^{\omega_L}\Z/p^n\Z$, so that there exists $k \in \Z$ such that 
\[\nu_L(c) \equiv k(\chinr(F_L)-1) \pmod{p^n\Z}.\] 
We set 
\[c' = \frac{c}{[\pi_L]^{\nu_L(c)}} \in U_{L_0}/(U_{L_0})^{p^n} 
\subseteq L_0^\times / \left(  L_0^\times \right)^{p^n},\] 
where $[\pi_L]$ denotes the class of $\pi_L$ modulo $(L_0^\times)^{p^n}$. By Lemma \ref{Neukirch} there exists $x' \in U_{L_0} / (U_{L_0})^{p^n}$ such that 
\[(x')^{\chinr(F_L)F_L-1}=c'. \]
Then we set $x=x'[\pi_L]^k$ and obtain
\[x^{\chinr(F_L)F_L-1}=(x'[\pi_L]^k)^{\chinr(F_L)F_L-1}
=c'[\pi_L]^{k(\chinr(F_L)-1)}=c'[\pi_L]^{\nu_L(c)}=c.\]
The commutativity of the diagram as well as the surjectivity of $\pi_{m,n}$ and $\sigma_{m,n}$ are evident. It remains to show that the projective system $(W_n)_n$ satisfies the Mittag-Leffler condition. To that end we prove that $\coker(\tau_{m,n})$ stabilizes for $n \ge \omega_L$ as $m$ tends to infinity.

By the Snake Lemma we obtain 
\[(L_0^\times)^{p^n}/(L_0^\times)^{p^{m}}\xrightarrow{\chinr(F_L)F_L-1}
\ker(\sigma_{m,n})\lra\coker(\tau_{m,n})\lra 0.\]
Since $n>\omega_L$, it is easy to see that
\[\ker(\sigma_{m,n})=(L_0^\times)^{p^n}/(L_0^\times)^{p^{m}},\]
so that we have to compute the cokernel of
\[(L_0^\times)^{p^n}/(L_0^\times)^{p^{m}} \stackrel{\psi}\lra (L_0^\times)^{p^n}/(L_0^\times)^{p^{m}},\]
where $\psi$ is induced by $\chinr(F_L)F_L-1$. Clearly,
\[\mathrm{im}(\psi) \sseq \left( \pi_L^{p^{\omega_L+n}\Z} \times \left( U_{L_0} \right)^{p^n} \right)  
\left( L_0^\times \right)^{p^m} / \left( L_0^\times \right)^{p^m}.\]
We claim that we have equality. Let $c^{p^n} = \pi_L^{lp^{n+\omega_L}} v^{p^n}$ with $l \in \Z$ and $v \in U_{L_0}$ represent an element of the right hand side and set $c := \pi_L^{lp^{\omega_L}} v$. By the first part of the proof we can construct $x \in L_0^\times / \left( L_0^\times \right)^{p^m}$ such that $x^{\chinr(F_L)F_L - 1} = [c]$. Then $\left( x^{p^n} \right)^{\chinr(F_L)F_L - 1} = [c^{p^n}]$, which proves our claim.

Finally we note that for $m \ge n + \omega_L$ one has $\left( L_0^\times\right)^{p^{m}} \sseq \pi_L^{p^{\omega_L+n}\Z} \times \left( U_{L_0} \right)^{p^n}$. Hence $\coker(\tau_{m,n})\cong (L_0^\times)^{p^n}/(\pi_L^{p^{\omega_L+n}\Z} \times \left( U_{L_0} \right)^{p^n})$ stabilizes.
\end{proof}

By the definition of $V_n$ we have a commutative diagram of short exact sequences
\[\xymatrix{
0\ar[rr] && V_{m}\ar[rr]\ar[d]^{\sigma_{m,n}} && L_0^\times/(L_0^\times)^{p^{m}} \ar[rr]^-{\nu_L}\ar[d]^{\pi_{m,n}}   
&& \Z/p^{\omega_L}\Z \ar[rr]\ar[d]^{=} && 0\\
0\ar[rr]&&V_{n}\ar[rr]&&L_0^\times/(L_0^\times)^{p^{n}}\ar[rr]^-{\nu_L}&&\Z/p^{\omega_L}\Z
\ar[rr]&&0.}\]
Note that when $\omega_L=0$ we simply interpret $\Z/p^{\omega_L}\Z$ as $0$. By Lemma \ref{Neu comm diag} the maps $\sigma_{m,n}$ are surjective and hence the projective system
$(V_n)_n$ satisfies the Mittag-Leffler condition. We define
\[W := \varprojlim_n W_n, \quad V := \varprojlim_n V_n,\]
and thus obtain by \cite[Prop.~3.5.7]{Weibel} two short exact sequences 
\[ 0 \lra W \lra \widehat{L_0^\times} \lra V \lra 0, \]
\[ 0 \lra V \lra \widehat{L_0^\times} \lra \Z / p^{\omega_L}\Z \lra 0.\]
Splicing together these two sequences we obtain
\[0 \lra W \lra \widehat{L_0^\times} \xrightarrow{\chinr(F_L)F_L - 1} \widehat{L_0^\times} \stackrel{\nu_L}\lra \Z / p^{\omega_L}\Z  \lra 0,\]
which we rewrite in the form
\begin{equation}\label{twisted ses 117}
0 \lra W(\chinr) \lra \widehat{L_0^\times}(\chinr) \xrightarrow{F_L - 1} \widehat{L_0^\times}(\chinr) \stackrel{\nu_L}\lra  \Z / p^{\omega_L}\Z (\chinr)  \lra 0.
\end{equation}

We clearly have $W(\chinr) \cong \widehat{L_0^\times}(\chinr)^{\Gal(L_0/L)}$ which by Lemma \ref{UL1hat} is isomorphic to $\calF(\p_L)$.

Note that we also have a natural $\Gal(L/K)$-action on $(\widehat{L_0^\times}(\chinr))^{\Gal(L_0/L)}$
and $\Z/p^{\omega_L}\Z(\chinr)$ defined by extending an element $\sigma\in \Gal(L/K)$ 
to an element $\tilde\sigma\in\Gal(L_0/K)$.

So we can formulate the following theorem.

\begin{thm}\label{FpnZpomega}
Assume $\chinr(F_L)\neq 1$. We have an exact sequence
\[0\lra(\widehat{L_0^\times}(\chinr))^{\Gal(L_0/L)}\lra \Lnr \xrightarrow{(F-1)\times 1}
\Lnr \stackrel{w_L}\lra\Z/p^{\omega_L}\Z(\chinr)\rightarrow 0\]
of $\Z_p[\Gal(L/K)]$-modules. Here $w_L$ denotes the sum of the valuations $\nu_L$.
\end{thm}

\begin{remark}
In the proof of Theorem  \ref{fundamental} we will use this result together with the isomorphism 
$(\widehat{L_0^\times}(\chinr))^{\Gal(L_0/L)} \cong \calF(\p_L)$ provided by Lemma \ref{UL1hat}.
\end{remark}

\begin{proof}
We start by showing that all the homomorphisms are $\Gal(L/K)$-invariant. Let $\sigma\in \Gal(L/K)$ and let $\tilde\sigma\in\Gal(L_0/K)$ be such that $\tilde\sigma|_L=\sigma$ and $\tilde\sigma|_{K_0}=F^{-n}$, where $n$ is an integer. Let $x\in (\widehat{L_0^\times}(\chinr))^{\Gal(L_0/L)}$. Then $x$ maps diagonally to the element $[x,x,\dots,x] \in \Lnr$ and we have
\[\begin{split}&(1\times\sigma)\cdot[x,x,\dots,x]=(F^{n}\times 1)(F^{-n}\times\sigma)\cdot[x,x,\dots,x]\\
&\qquad=(F^{n}\times 1)\cdot[x^{\tilde\sigma\chinr(\tilde\sigma)},x^{\tilde\sigma\chinr(\tilde\sigma)},\dots,x^{\tilde\sigma\chinr(\tilde\sigma)}]\\
&\qquad=[x^{\tilde\sigma \chinr(\tilde\sigma)F_L\chinr(F_L)},\dots,x^{\tilde\sigma \chinr(\tilde\sigma)F_L\chinr(F_L)},
   x^{\tilde\sigma\chinr(\tilde\sigma)},\dots,x^{\tilde\sigma\chinr(\tilde\sigma)}]\\
&\qquad=[x^{\tilde\sigma\chinr(\tilde\sigma)},x^{\tilde\sigma\chinr(\tilde\sigma)},\dots,x^{\tilde\sigma\chinr(\tilde\sigma)}].\end{split}\]
If we inject $\sigma\cdot x=x^{\tilde\sigma\chinr(\tilde\sigma)}$ diagonally into $\Lnr$ we obviously get the same element, so that the first map is $\Gal(L/K)$-invariant. The second map is $\Gal(L/K)$-invariant because $(F-1)\times 1$ and $1\times \sigma$ commute.

Let $[x_1,\dots,x_{d_{L/K}}] \in \Lnr$. Then
\[\begin{split}&(1\times\sigma)\cdot[x_1,x_2,\dots,x_{d_{L/K}}]=(F^{n}\times 1)(F^{-n}\times\sigma)\cdot[x_1,x_2,\dots,x_{d_{L/K}}]\\
&\qquad=(F^{n}\times 1)\cdot[x_1^{\tilde\sigma\chinr(\tilde\sigma)},x_2^{\tilde\sigma\chinr(\tilde\sigma)},\dots,x_{d_{L/K}}^{\tilde\sigma\chinr(\tilde\sigma)}]\\
&\qquad=[x_{d_{L/K}+1-n}^{\tilde\sigma\chinr(\tilde\sigma) F_L\chinr(F_L)},\dots,
   x_{d_{L/K}}^{\tilde\sigma\chinr(\tilde\sigma) F_L\chinr(F_L)},x_1^{\tilde\sigma\chinr(\tilde\sigma)},\dots,x_{d_{L/K}-n}^{\tilde\sigma\chinr(\tilde\sigma)}].\end{split}\]
Its valuation is:
\[\begin{split}& \chinr(\tilde\sigma)(\nu_L(x_1)+\dots+\nu_L(x_{d_{L/K}-n})+\chinr(F_L)(\nu_L(x_{d_{L/K}+1-n})+\dots+\nu_L(x_{d_{L/K}})))\\
&\qquad\equiv\chinr(\tilde\sigma)(\nu_L(x_1)+\dots+\nu_L(x_{d_{L/K}}))\pmod{p^{\omega_L}},
\end{split}\]
from which we deduce the $\Gal(L/K)$-invariance of $w_L$.

Next we show the exactness of the sequence. The first map is just the diagonal embedding.

Let $x=[x_1,x_2,\dots,x_{d_{L/K}}]\in\ker((F-1)\times 1)$. Then
\[\begin{split}((F-1)\times 1)(x)&
=\left[\frac{x_{d_{L/K}}^{F_L\chinr(F_L)}}{x_1},\frac{x_1}{x_2},\frac{x_2}{x_3},\dots,\frac{x_{d_{L/K}-1}}{x_{d_{L/K}}}\right]=1\end{split}\]
implies that $x_1=x_2=\dots=x_{d_{L/K}}$ and $x_1=x_{d_{L/K}}^{F_L\chinr(F_L)}=x_1^{F_L\chinr(F_L)}$. It follows that the kernel of $(F-1)\times 1$ is contained in the image of the first injection. Since the opposite inclusion is straightforward, we have exactness at the second non-trivial term of the sequence.

Let us consider exactness at the third non-trivial term. The following is just an adaptation of the arguments used to prove \cite[Lemma~11 and 12]{BleyDebeerst13}. Let $c=[c_1,c_2,\dots,c_{d_{L/K}}]\in \ker(w_L)$ for some $c_i\in \widehat {L_0^\times}(\chinr)$. By assumption $c_1c_2\dots c_{d_{L/K}}$ is in the kernel of $\nu_L$ and hence by (\ref{twisted ses 117}) there exists an element $y\in\widehat{L_0^\times}(\chinr)$ such that $y^{\chinr(F_L)F_L-1}=c_1c_2\dots c_{d_{L/K}}$. Now we set 
\[x=[c_2c_3\dots c_{d_{L/K}}y,c_3\dots c_{d_{L/K}}y,\dots,c_{d_{L/K}-1}c_{d_{L/K}}y,c_{d_{L/K}}y,y]\in \Lnr.\] 
Then
\[\begin{split}
((F-1)\times 1)(x)&=\left[\frac{x_{d_{L/K}}^{F_L\chinr(F_L)}}{x_1},\frac{x_1}{x_2},\frac{x_2}{x_3},\dots,\frac{x_{d_{L/K}-1}}{x_{d_{L/K}}}\right]\\
&=\left[\frac{y^{F_L\chinr(F_L)}}{c_2c_3\dots c_{d_{L/K}}y},\frac{c_2c_3\dots 
  c_{d_{L/K}}y}{c_3c_4\dots c_{d_{L/K}}y},\frac{c_3c_4\dots c_{d_{L/K}}y}{c_4c_5\dots c_{d_{L/K}}y},\dots,\frac{c_{d_{L/K}}y}{y}\right]\\
&=\left[\frac{y^{F_L\chinr(F_L)-1}}{c_2c_3\dots c_{d_{L/K}}},c_2,c_3,\dots,c_{d_{L/K}}\right]\\
&=\left[c_1,c_2,c_3,\dots,c_{d_{L/K}}\right].
\end{split}\]
The opposite inclusion is easy by the definition of $\omega_L$. This proves exactness in the third term.

The surjectivity of the last map is clear.
\end{proof}

We let $M/L/K$ be finite extensions such that both $M/K$ and $L/K$ are Galois. We have a canonical  inclusion
\begin{eqnarray*}
\iota = \iota_{M/L} \colon \Lnr &\lra& \Mnr, \\ 
x &\mapsto& [x^{F_L^{d_{M/L}-1}\chinr(F_L)^{d_{M/L}-1}},x^{F_L^{d_{M/L}-2}\chinr(F_L)^{d_{M/L}-2}},\dots,x],
\end{eqnarray*} 
where $x=[x_1,\dots,x_{d_{L/K}}] \in \left( \widehat{L_0^\times} \right)^{d_{L/K}}$ and the exponents of $x$ have to be understood componentwise. There are three possibilities:
\begin{enumerate}
\item $\chinr(F_M)\neq 1$ and $\chinr(F_L)\neq 1$;
\item $\chinr(F_M)=1$ and $\chinr(F_L)\neq 1$;
\item $\chinr(F_M)=1$ and $\chinr(F_L)=1$.
\end{enumerate}
In case 1 we verify by a long but straightforward computation that the diagram of exact rows
\[\xymatrix@C-=0.5cm{0\ar[r]&(\widehat{L_0^\times}(\chinr))^{\Gal(L_0/L)}\ar[r]\ar[d]^{\subseteq}&\Lnr\ar[rr]^-{(F-1)\times 1}\ar[d]^{{\iota_{M/L}}}&&
\Lnr\ar[r]^-{w_L}\ar[d]^{{\iota_{M/L}}}&\Z/p^{\omega_L}\Z(\chinr)\ar[r]
\ar[d]^{e_{M/L}\sum_{i=0}^{d_{M/L}-1}F_L^i}&0\\
0\ar[r]&(\widehat{M_0^\times}(\chinr))^{\Gal(M_0/M)}\ar[r]&\Mnr\ar[rr]^-{(F-1)\times 1}&&\Mnr\ar[r]^-{w_M}&\Z/p^{\omega_M}\Z(\chinr)\ar[r]&0}\]
commutes.

In case 2 we have
\[\begin{split}
w_M(\iota_{M/L}(x))&=w_M(x)\sum_{i=0}^{d_{M/L}-1}\chinr(F_L)^i=w_M(x)\cdot\frac{\chinr(F_L)^{d_{M/L}}-1}{\chinr(F_L)-1}\\
&=w_M(x)\frac{\chinr(F_M)-1}{\chinr(F_L)-1}=0.\end{split}\]
Therefore we obtain a commutative diagram
\[\xymatrix@C-=0.5cm{0\ar[r]&(\widehat{L_0^\times}(\chinr))^{\Gal(L_0/L)}\ar[r]\ar[d]^{\subseteq}&\Lnr\ar[rr]^-{(F-1)\times 1}\ar[d]^{{\iota_{M/L}}}&&
\Lnr\ar[r]^-{w_L}\ar[d]^{{\iota_{M/L}}}&\Z/p^{\omega_M}\Z(\chinr)\ar[r]
\ar[d]^{0}&0\\
0\ar[r]&(\widehat{M_0^\times}(\chinr))^{\Gal(M_0/M)}\ar[r]&\Mnr\ar[rr]^-{(F-1)\times 1}&&\Mnr\ar[r]^-{w_M}&\Z_p(\chinr)\ar[r]&0.}\]
Note that the bottom exact sequence comes from Theorem \ref{Serrethm}.

In case 3 our diagram takes the form
\[\xymatrix@C-=0.5cm{0\ar[r]&(\widehat{L_0^\times}(\chinr))^{\Gal(L_0/L)}\ar[r]\ar[d]^{\subseteq}&\Lnr\ar[rr]^-{(F-1)\times 1}\ar[d]^{{\iota_{M/L}}}&&
\Lnr\ar[r]^-{w_L}\ar[d]^{{\iota_{M/L}}}&\Z_p(\chinr)\ar[r]
\ar[d]^{[M:L]}&0\\
0\ar[r]&(\widehat{M_0^\times}(\chinr))^{\Gal(M_0/M)}\ar[r]&\Mnr\ar[rr]^-{(F-1)\times 1}&&\Mnr\ar[r]^-{w_M}&\Z_p(\chinr)\ar[r]&0.}\]

\begin{lemma}\label{FpKLurexact}
(a) If $\chinr(F_L)\neq 1$ for every finite Galois extension $L/K$, then we have an exact sequence
\[0\lra\varinjlim_L(\widehat{L_0^\times}(\chinr))^{\Gal(L_0/L)}\lra\varinjlim_L\Lnr\xrightarrow{(F-1)\times 1}\varinjlim_L\Lnr\lra 0,\]
where the direct limit is taken over all finite Galois extensions $L/K$.

(b) If $\chinr(F_L)=1$ for some finite Galois extension $L/K$, then we have an exact sequence
\[0\ra\varinjlim_L(\widehat{L_0^\times}(\chinr))^{\Gal(L_0/L)}\ra\varinjlim_L\Lnr\xrightarrow{(F-1)\times 1}\varinjlim_L\Lnr\ra\Qp(\chinr)\ra0.\]
\end{lemma}

\begin{proof}
For the proof of (a) we use  the diagram of case 1. Since the direct limit functor is exact 
it suffices to show that $\varinjlim_L \left( \Z / p^{\omega_L}\Z (\chinr) \right)$ is trivial.

We first note that
\[\omega_M = v_p\!\left( \frac{1-\chinr(F_M)}{1-\chinr(F_L)}\!\cdot\!(1-\chinr(F_L)) \right) = \omega_L + v_p( 1 + \chinr(F_L) + \ldots + \chinr(F_L)^{d_{M/L}-1}).\]
Let now $x \in \Z / p^{\omega_L}\Z (\chinr)$. Then we let $M/K$ be a finite Galois extension such that $L \sseq M$ and $p^{\omega_L} | e_{M/L}$. Then
\[e_{M/L} (1 + F_L + \ldots + F_L^{d_{M/L}-1}) x = e_{M/L} (1 + \chinr(F_L) + \ldots + \chinr(F_L)^{d_{M/L}-1})x\]
and, since 
\[\begin{split}
& v_p( e_{M/L} (1 + \chinr(F_L) + \ldots + \chinr(F_L)^{d_{M/L}-1}) ) \\
&\qquad\qquad\ge \omega_L + v_p( 1 + \chinr(F_L) + \ldots + \chinr(F_L)^{d_{M/L}-1}) = \omega_M,
\end{split}\]
we see that $x$ becomes trivial in $\Z / p^{\omega_M}\Z (\chinr)$.

The proof of (b) is achieved by taking the direct limit over the diagrams of cases 2 and 3. The result follows since it is straightforward to show that $\varinjlim_L\Z_p(\chinr) \cong \Qp(\chinr)$.

\end{proof}

\begin{lemma}\label{LurGcanNur}
We have
\[\left( \Mnr \right)^{\Gal(M/L)}=\iota_{M/L}\left( \Lnr \right).\]
\end{lemma}

\begin{proof}
For the proof we set $\tilde{d} := d_{M/L}$ and $d := d_{L/K}$. As usual we identify $\Mnr$ and $(\widehat{M_0^\times})^{d\tilde d}$, the $\Gal(K_0/K)\times\Gal(M/K)$-action being described by (\ref{rule 1}) and (\ref{rule 2}). Let us consider $[x_1,x_2,\dots,x_{\tilde d}]\in \left( \Mnr \right)^{\Gal(M/L)}$ 
with $x_1,x_2,\dots,x_{\tilde d}\in ( \widehat{M_0^\times})^{d}$. 

Let $\tilde\sigma\in\Gal(M_0/L_0)$ and $\sigma=\tilde\sigma|_M \in \Gal(M/L)$. We have
\[(1\times\sigma)[x_1,x_2,\dots,x_{\tilde d}]=[x_1^{\tilde\sigma\chinr(\tilde\sigma)},x_2^{\tilde\sigma\chinr(\tilde\sigma)},\dots,x_{\tilde d}^{\tilde\sigma\chinr(\tilde\sigma)}].\]
Since $\tilde\sigma|_{K_0}=1$ it follows that $\chinr(\tilde\sigma)=1$. So from the above calculations, recalling that $[x_1,x_2,\dots,x_{\tilde d}]$ has to be fixed by the action of $1\times\sigma$, we get $x_i^{\tilde\sigma}=x_i$ for $i=1,\dots,\tilde d$. Since this holds for every $\tilde\sigma\in\Gal(M_0/L_0)$ we conclude that $x_{i}\in (\widehat{L_0^\times})^d=\Lnr$ for $i=1,\ldots, \tilde d$.

Now let us take $\tilde\tau\in\Gal(M_0/L)$ such that $\tilde\tau|_{L_0}=F_L$ and set 
$\tau:=\tilde\tau|_M \in \Gal(M/L)$. In particular the element $[x_1,x_2,\dots,x_{\tilde d}]$ must be fixed by the action of $\tau^{-1}$. We have
\[\begin{split}
[x_1,x_2,\dots,x_{\tilde d}]&=(1\times \tau^{-1})[x_1,x_2,\dots,x_{\tilde d}]\\&=(F^d\times 1)(F^{-d}\times \tau^{-1})[x_1,x_2,\dots,x_{\tilde d}]\\
&=(F^d\times 1)[x_1^{\tilde\tau^{-1}\chinr(\tilde\tau)^{-1}},x_2^{\tilde\tau^{-1}\chinr(\tilde\tau)^{-1}},\dots,x_{\tilde d}^{\tilde\tau^{-1}\chinr(\tilde\tau)^{-1}}]\\
&=[x_{\tilde d}^{\tilde\tau^{-1}\chinr(\tilde\tau)^{-1}F_M\chinr(F_M)},x_1^{\tilde\tau^{-1}\chinr(\tilde\tau)^{-1}},\dots,x_{\tilde d-1}^{\tilde\tau^{-1}\chinr(\tilde\tau)^{-1}}].
\end{split}\]
Since $x_i \in \Lnr=(\widehat{L_0^\times})^d$ for $i=1,\ldots, \tilde d$, we can always substitute $\tilde\tau$ by $\tilde\tau|_{L_0}=F_L$ and $F_M$ by $F_M|_{L_0}=F_L^{\tilde d}$. Setting $x=x_{\tilde d}$ and comparing all the components in the above equality, except for the first one, we obtain
\[[x_1,x_2,\dots,x_{\tilde d}]=[x^{F_L^{\tilde d-1}\chinr(F_L)^{\tilde d-1}},x^{F_L^{\tilde d-2}\chinr(F_L)^{\tilde d-2}},\dots,x].\]
Therefore we have proved that
\[[x_1,x_2,\dots,x_{\tilde d}]=\iota_{M/L}(x).\]

For the opposite inclusion we take $x \in \Lnr$ and have to show that $y := \iota_{M/L}(x)$ is fixed by all $\sigma \in \Gal(M/L)$. As before we let $\tilde\tau \in \Gal(M_0/L)$ denote an element such that $\tilde\tau|_{L_0} = F_L$ and set $\tau := \tilde\tau|_M \in \Gal(M/L)$. Then
\[\begin{split}
(1 \times \tau^{-1}) y &= (F^d\times 1)(F^{-d}\times \tau^{-1})[x^{F_L^{\tilde d-1}\chinr(F_L)^{\tilde d-1}},x^{F_L^{\tilde d-2}\chinr(F_L)^{\tilde d-2}},\dots,x]\\
&=(F^d\times 1)[x^{F_L^{\tilde d-2}\chinr(F_L)^{\tilde d-2}},x^{F_L^{\tilde d-3}\chinr(F_L)^{\tilde d-3}},\dots,x^{F_L^{-1}\chinr(F_L)^{-1}}]\\
&=[x^{F_L^{-1}\chinr(F_L)^{-1}F_M\chinr(F_M)},x^{F_L^{\tilde d-2}\chinr(F_L)^{\tilde d_M-2}},\dots,x]\\
&=\iota_{M/L}(x) = y.
\end{split}\]
Let now $\sigma \in \Gal(M/L)$ be arbitrary. Then there exists $i \in \Z$ such that $\sigma\tau^i |_{M \cap L_0} = \id$. We then have $(1 \times \sigma\tau^i)y = y$ and the result follows from the above computation and $(1 \times \sigma) = (1 \times \sigma\tau^i)(1 \times \tau^{-i})$.
\end{proof}

From Lemma \ref{LurGcanNur} it is immediate that 
\begin{equation}\label{GN invariants}
\left( \varinjlim_L\Lnr \right)^{G_N} = \Nnr,
\end{equation}
where the direct limit is taken over all finite Galois extensions $L/K$. From now on, we will mainly be interested in the case $\chinr|_{G_N}\neq 1$, since the corresponding results for the case $\chinr|_{G_N}=1$ are known. Under the assumption $\chinr|_{G_N}\neq 1$, we see that
\begin{equation}\label{yan 1}
(\Qp(\chinr))^{G_N} = 0.
\end{equation}

\begin{lemma}\label{acyclic}
The $G_K$-modules $\varinjlim_L \Lnr$ and $\Qp(\chinr)$ are acyclic with respect to the fixed point functor $A \mapsto A^{G_N}$.
\end{lemma}

\begin{proof}
By \cite[Prop.~1.5.1]{NSW} we have for $i >0$
\[H^i(N, \varinjlim_L\Lnr) = \varinjlim_L H^i(\Gal(L/N), \Lnr)\]
which is trivial by Lemma \ref{LurInd}. The module $\Qp(\chinr)$ is cohomologically trivial since it is divisible (see \cite[Prop.~1.6.2]{NSW}).
\end{proof}

We can now state and prove the main result of this subsection. 

\begin{thm}\label{fundamental}
If $\chinr|_{G_N}\neq 1$, then the complex
\[C^\bullet_{N, \calF} := \left[ \Nnr \xrightarrow{(F-1)\times 1} \Nnr \right]\]
with non-trivial modules in degree 0 and 1 represents $R\Gamma(N,\calF)$.
\end{thm}

\begin{proof}
We write $X^\bullet$ for the standard resolution (as defined in 
\cite[Sec.~1.2]{NSW}) and $I^\bullet$ for an injective resolution of $\varinjlim_L(\widehat{L_0^\times}(\chinr))^{\Gal(L_0/L)}$. Let $C^\bullet$ be the complex 
\[\varinjlim_L \Lnr \xrightarrow{(F-1) \times 1} \varinjlim_L \Lnr\lra\Qp(\chinr)\]
or the complex
\[\varinjlim_L \Lnr \xrightarrow{(F-1) \times 1} \varinjlim_L \Lnr,\]
according to whether $\chinr$ becomes trivial for some $G_L$ or not.

The $G$-modules $\Map(G^n,\varinjlim_L(\widehat{L_0^\times}(\chinr))^{\Gal(L_0/L)})$ appearing in $X^\bullet$  are induced in the sense of \cite[Sec.~1.3]{NSW}. By Proposition 1.3.7 of loc.cit. the standard resolution consists therefore of acyclic modules with respect to the fixed point functor. By \cite[Th.~XX.6.2]{LangAlgebra} there exists a morphism of complexes $X^\bullet \lra I^\bullet$ inducing the identity on $\varinjlim_L(\widehat{L_0^\times}(\chinr))^{\Gal(L_0/L)}$ and isomorphisms on cohomology. By Lemma \ref{FpKLurexact} and \ref{acyclic} we may apply \cite[Th.~XX.6.2]{LangAlgebra} once again and obtain a morphism of complexes $C^\bullet \lra I^\bullet$, also inducing the identity on $\varinjlim_L(\widehat{L_0^\times}(\chinr))^{\Gal(L_0/L)}$ and isomorphisms on cohomology. By Lemma \ref{UL1hat},
\[\left(\varinjlim_L(\widehat{L_0^\times}(\chinr))^{\Gal(L_0/L)}\right)^{G_N}\cong (\widehat{N_0^\times}(\chinr))^{\Gal(N_0/N)}\cong\calF(\p_{N}).\]
Therefore applying the $G_N$-fixed point functor together with (\ref{GN invariants}) and (\ref{yan 1}) shows that $C_{N, \calF}^\bullet = \left( C^\bullet \right)^{G_N}$ is quasi-isomorphic to $R\Gamma(N, \calF)$.
\end{proof}

\begin{coroll}
  \label{calF coh}
With the assumptions of Theorem \ref{fundamental} we have
\[H^i(N, \calF) =
\begin{cases}
  \calF(\p_N) & \text{if } i=0 \\ (\Z /p^\omega\Z)(\chinr) &\text{if } i=1 \\ 0 &\text{if } i \ne 0,1.
\end{cases}\]
\end{coroll}
\begin{proof}
This is an immediate consequence of Theorems \ref{FpnZpomega} and \ref{fundamental}.
\end{proof}

In the sequel we always identify $R\Gamma(N, \calF)$ and $C^\bullet_{N, \calF}$. We again put $G := \Gal(N/K)$ and recall that a complex is said to be $\ZpG$-perfect, if it is quasi-isomorphic to a bounded complex of finitely generated $\ZpG$-projectives.

\begin{coroll}\label{Pbullet}
If $\chinr|_{G_N}\neq 1$, then the complex $R\Gamma(N,\calF)$ is $\ZpG$-perfect. More precisely, there exists a complex 
\[P^\bullet := [P^{-1}\lra P^{0}\lra P^1]\]
of finitely generated $\ZpG$-projectives together with a quasi-isomorphism $\eta \colon P^\bullet \lra R\Gamma(N, \calF)$.
\end{coroll}

\begin{proof}
By \cite[Prop.~XXI.1.1]{LangAlgebra} and Theorem \ref{fundamental} there is a quasi-isomorphism $f\colon K^\bullet \lra R\Gamma(N, \calF)$, where $K^\bullet := [A \lra B]$ is centered in degree $0$ and $1$ with $B$ a $\ZpG$-projective module. As the cohomology groups $H^i(N, \calF)$ are finitely generated, the proof of \cite[Prop.~XXI.1.1]{LangAlgebra} actually shows that we can assume that $A$ and $B$ are both finitely generated.  By our Lemma \ref{LurInd} the module $\Nnr$ is cohomologically trivial. We hence  apply \cite[Prop.~XXI.1.2]{LangAlgebra} with $\mathfrak F$ being the family of cohomologically trivial modules. We obtain that $A$ is cohomologically trivial and hence has a two term resolution $0\lra P^{-1} \lra P^{0} \lra A\lra 0$ with finitely generated projective $\ZpG$-modules $P^{-1}$ and $P^{0}$. It follows that $R\Gamma(N,\calF)$ is quasi-isomorphic to $[P^{-1}\lra P^{0}\lra P^1]$ with $P^1 = B$.
\end{proof}

\end{subsection}

\begin{subsection}{\texorpdfstring{From $R\Gamma(N,\calF)$ to $R\Gamma(N,T)$}{From R Gamma(N,F) to R Gamma(N,T)}}

If $C^\bullet$ is a complex and $p \in \Z$, then $C^\bullet[p]$ denotes the shifted complex, that is $C^i[p] = C^{i+p}$. In this subsection we will show that the complex $P^\bullet[-1]$ of Corollary \ref{Pbullet} is quasi-isomorphic to $R\Gamma(N, T)$. Together with Theorem \ref{fundamental} this will prove Theorem \ref{main theorem 2} whose formulation we recall once again.
\begin{thm}\label{fundamental 2}
The complex
\[C^\bullet_{N, T} := \left[ \Nnr\xrightarrow{(F-1)\times 1}\Nnr \right]\]
with non-trivial modules in degree 1 and 2 represents $R\Gamma(N,T)$.
\end{thm}

\begin{remark}\label{our coh proof}
Theorem \ref{fundamental 2} combines with Theorem \ref{FpnZpomega} to give a new proof of Lemma \ref{galois cohomology}, part (a).
\end{remark}

In the following we adapt the argument used to prove \cite[Th.~4.20]{BreuPhd}. We will need the following Lemma.

\begin{lemma}\label{homotopicto0}
Let $h\colon Y^\bullet\lra W^\bullet$ be a map of cochain complexes. Suppose $h=0$ in $D(\ZpG)$. Suppose also that $Y^\bullet$ is bounded from above, comprising only projectives. Then $h$ is homotopic to zero.
\end{lemma}

\begin{proof}
By \cite[Coroll.~10.3.9]{Weibel}, there exists a quasi-isomorphism $t\colon Z^\bullet\lra Y^\bullet$ such that $ht\colon Z^\bullet\lra W^\bullet$ is null homotopic. By the dual of \cite[Lemma~10.4.6]{Weibel} there exists $s\colon Y^\bullet\lra Z^\bullet$ such that $ts$ is homotopy equivalent to the identity. Hence $hts$ is both homotopic to $0$ and to $h$.
\end{proof}

\begin{lemma}\label{compatiblephi}
There exist quasi-isomorphisms of complexes
\[\varphi_n\colon P^\bullet/p^n\lra R\Gamma(N,\calF[p^n])[1]\]
which are compatible with the inverse systems. 
\end{lemma}

\begin{proof}
The proof proceeds in two steps. First we construct morphisms of complexes $\varphi_n$, and then in a second step we modify the definition in such a way that the $\varphi_n$ are compatible with the inverse systems.

By Kummer theory  and \cite[Exercise I.3.1]{NSW} we have a short exact sequence of complexes
\[0 \lra R\Gamma(N,\calF[p^n])\lra R\Gamma(N,\calF)\stackrel{p^n}\lra R\Gamma(N,\calF)\lra 0.\]
As a consequence of \cite[Example~10.4.9]{Weibel} we obtain an exact triangle
\[R\Gamma(N,\calF[p^n])\lra R\Gamma(N,\calF)\stackrel{p^n}\lra R\Gamma(N,\calF)\lra .\]
Analogously we see that 
\[P^\bullet \stackrel{p^n}\lra P^\bullet \lra P^\bullet/p^n \lra\]
is an exact triangle.

Let $\eta \colon P^\bullet \lra R\Gamma(N, \calF)$ be the quasi-isomorphism of Corollary \ref{Pbullet}. Then we obtain a commutative diagram in $D(\ZpG)$ of the following form
\[\xymatrix{
P^\bullet\ar[r]^{p^n}\ar[d]^\eta& P^\bullet\ar[r]^{\pi_n} \ar[d]^\eta   & P^\bullet/p^n\ar[r]               & \\
R\Gamma(N,\calF)\ar[r]^{p^n}   & R\Gamma(N,\calF)\ar@{..>}[r]^-{\lambda_n} &R\Gamma(N,\calF[p^n])[1]\ar[r] & .}\]
For clarity we use dotted arrows for morphisms in $D(\ZpG)$ which are not actual maps of complexes. As it is common we use $\xrightarrow{\sim}$ to indicate a quasi-isomorphism.
By the definition of a morphism in the derived category there exists a complex $Z^\bullet$ such that
\[R\Gamma(N,\calF)\stackrel{\sim}\lla Z^\bullet\lra R\Gamma(N,\calF[p^n])[1].\]
By \cite[Lemma~VI.8.17]{Milne} we construct a morphism of complexes 
\[\psi_n\colon P^\bullet\lra R\Gamma(N,\calF[p^n])[1],\]
which makes the following diagram 
\[\xymatrix{&P^\bullet\ar[r]^{p^n}\ar[d]^\eta& P^\bullet\ar[dr]^{\psi_n}\ar[r]^{\pi_n}\ar[d]^\eta& P^\bullet/p^n\ar[r]&\\
&R\Gamma(N,\calF)\ar[r]^{p^n}&R\Gamma(N,\calF)\ar@{..>}[r]^-{\lambda_n} &R\Gamma(N,\calF[p^n])[1] \ar[r]&}\]
commutative in $D(\ZpG)$. As $R\Gamma(N, \calF[p^n])$ is $p^n$-torsion, the map $\psi_n$ factors through $P^\bullet/p^n$ and we obtain morphisms $\varphi_n$ of complexes such that

\begin{equation}\label{firstdefphin}
\begin{split}
\xymatrix{
&P^\bullet\ar[r]^{p^n}\ar[d]^\eta& P^\bullet\ar[dr]^{\psi_n}\ar[r]^{\pi_n}\ar[d]^\eta& P^\bullet/p^n\ar[d]^{\varphi_n}\ar[r]&\\
&R\Gamma(N,\calF)\ar[r]^{p^n}&R\Gamma(N,\calF)\ar@{..>}[r]^-{\lambda_n} &R\Gamma(N,\calF[p^n])[1]\ar[r]&
}
\end{split}
\end{equation}
commutes. 
We want to show that (\ref{firstdefphin}) is actually a morphism of triangles. To that end we have to show that in
\[\xymatrix{
&P^\bullet\ar[r]^{p^n}\ar[d]^\eta& P^\bullet\ar[r]^{\pi_n}\ar[d]^\eta& P^\bullet/p^n\ar[d]^{\varphi_n}\ar[r]^\omega& P^\bullet[1]\ar[d]^{\eta[1]}\\
&R\Gamma(N,\calF)\ar[r]^{p^n}&R\Gamma(N,\calF)\ar@{..>}[r]^-{\lambda_n} &R\Gamma(N,\calF[p^n])[1]\ar[r]^-\nu& R\Gamma(N,\calF)[1]}\]
the right hand rectangle commutes. Since $P^\bullet \stackrel{\pi_n}\lra P^\bullet/p^n$ is surjective it suffices 
to show that $\eta[1]\circ\omega\circ\pi_n = \nu\circ\varphi_n\circ\pi_n$. By the commutativity of the central square this is equivalent to $\eta[1]\circ\omega\circ\pi_n=\nu\circ\lambda_n\circ\eta$ which holds true by \cite[Sec.~10.2, TR3]{Weibel}.

We can therefore apply the 5-Lemma (\cite[Exercise~10.2.2]{Weibel}) and derive that the morphism $\varphi_n$ is a quasi-isomorphism. 

In a second step we now show that the maps $\varphi_n$ can be chosen so that they are compatible with the inverse systems. We recall the identification $C^\bullet_{N,\calF}=R\Gamma(N,\calF)$ and we set $Q_n^\bullet:=R\Gamma(N,\calF[p^n])[1]$. We write $\pi_n^Q \colon Q^\bullet_n \lra Q^\bullet_{n-1}$ for the canonical map induced by the epimorphism $\calF[p^n] \lra \calF[p^{n-1}]$. Note that by \cite[Exercise I.3.1]{NSW} each $\pi_n^Q$ is an epimorphism. We also write $\pi_n^P\colon P^\bullet/p^n\lra P^\bullet/p^{n-1}$ for the canonical projection.

We proceed by induction on $n$. We have the following diagram in $D(\ZpG)$
\[\xymatrix{P^\bullet\ar[rr]^{p^n}\ar[ddd]^p\ar[dr]^\eta&& P^\bullet\ar[ddd]^{=}\ar[rr]^{\pi_n}\ar[rd]^\eta&& P^\bullet/p^n\ar[ddd]^{\pi_n^P}\ar[dr]^{\varphi_n}\\
&C^\bullet_{N,\calF}\ar[ddd]^p\ar[rr]^{p^n}&&C^\bullet_{N,\calF}\ar[ddd]^=\ar@{..>}[rr]&&Q_n^\bullet\ar[ddd]^{\pi_n^Q}\\
\\
P^\bullet\ar[rr]_{p^{n-1}}\ar[dr]_\eta&& P^\bullet\ar[rr]_{\pi_{n-1}}\ar[rd]^\eta&& P^\bullet/p^{n-1}\ar[dr]^{\varphi_{n-1}}\\
&C^\bullet_{N,\calF}\ar[rr]^{p^{n-1}}&&C^\bullet_{N,\calF}\ar@{..>}[rr]&&Q_{n-1}^\bullet}\]
and want to show that we can modify $\varphi_n$ so that the square on the right commutes. By a diagram chase we can show that
\[(\varphi_{n-1}\pi_n^P-\pi_n^Q\varphi_n)\pi_n=0\]
in $D(\ZpG)$. By Lemma \ref{homotopicto0} we obtain a homotopy 
\[\tilde{h}_n \colon P^\bullet \lra Q^\bullet_{n-1}[-1]\]
such that 
\[(\varphi_{n-1}^i\pi_n^{P,i}-\pi_n^{Q,i}\varphi_n^i)\pi_n^i=d_{Q_{n-1}}^{i-1}\tilde h_n^i+\tilde h_n^{i+1}d_{P,n}^i.\]
Since $Q_{n-1}^\bullet[-1]=R\Gamma(N,\calF[p^{n-1}])$ is $p^{n-1}$-torsion, $\tilde h_n$ induces a homotopy
\[h_n\colon P^\bullet/p^n\lra Q_{n-1}^{\bullet}[-1]\]
such that $\tilde h_n=h_n\pi_n$. Hence we obtain
\[(\varphi_{n-1}^i\pi_n^{P,i}-\pi_n^{Q,i}\varphi_n^i)\pi_n^i=d_{Q_{n-1}}^{i-1}h_n^i\pi_n^i+h_n^{i+1}
\pi_n^{i+1}d_{P,n}^i=d_{Q_{n-1}}^{i-1}h_n^i\pi_n^i+h_n^{i+1}d_{P,n}^i\pi_n^i,\]
where $d_{P,n}^i$ is the differential of $P^\bullet/p^n$. Since $\pi_n^i$ is surjective, we deduce that
\[\varphi_{n-1}^i\pi_n^{P,i}-\pi_n^{Q,i}\varphi_n^i = d_{Q_{n-1}}^{i-1}h_n^i+h_n^{i+1}d_{P,n}^i.\]
We are now precisely in the situation of \cite[page~1367]{BurnsFlach98}. For the sake of completeness we recall the argument.

We view $P^i/p^n$, $Q_n^{i-1}$ and $Q_{n-1}^{i-1}$ as $\Z/p^n\Z[G]$-modules and observe that $P^i/p^n$ is $\Z/p^n\Z[G]$-projective.
Hence there exists a lift $\tilde{\tilde h}_n^i$ of $h_n^i$ such that 
\[\xymatrix{P^i/p^n\ar[d]^{\tilde{\tilde h}_n^i}\ar[drr]^{h_n^i}\\
Q_n^{i-1}\ar[rr]^{\pi_n^{Q,i-1}}&&Q_{n-1}^{i-1}\ar[r]&0}\]
commutes. We now replace each $\varphi_n^i$ by
\[\varphi_n^i+d_{Q_n}^{i-1}\tilde{\tilde h}_n^i+\tilde{\tilde h}_n^{i+1}d_{P,n}^i.\]
Since we have only changed $\varphi_n$ by a homotopy the diagram (\ref{firstdefphin}) still commutes in $D(\ZpG)$. By a straightforward computation we see that $\pi_n^{Q}\varphi_n = \varphi_{n-1}\pi_n^{P}$. This concludes the proof of the inductive step.
\end{proof}

We now provide the proof of Theorem \ref{fundamental 2}.
\begin{proof}
From Lemma \ref{compatiblephi} we obtain a map $\varphi := \varprojlim_n \varphi_n$ of complexes,
\[\varphi \colon P^\bullet \lra R\Gamma(N, T)[1].\]
Here we have naturally identified $\varprojlim_n P^\bullet/p^n$ with $P^\bullet$ and $\varprojlim_n R\Gamma(N, \calF[p^n])$ with $R\Gamma(N, T)$. Since each $\varphi_n$ is a quasi-isomorphism we obtain for each $i$ an isomorphism
\[H^i(\varphi) \colon \varprojlim_n H^i(P^\bullet/p^n) \lra \varprojlim_n H^{i+1}(N, \calF[p^n]).\]
By  \cite[Lemma~9]{BurnsFlach98} we conclude that  $\varprojlim_n H^i(P^\bullet/p^n) = H^i(P^\bullet)$ and also that $\varprojlim_n H^{i+1}(N, \calF[p^n]) = H^{i+1}(N, T)$, so that $\varphi$ is a quasi-isomorphism.
\end{proof}
\end{subsection}
\end{section}

\begin{section}{\texorpdfstring{The term $\Ucris$}{The term Ucris}}\label{Ucris}
The term $\Ucris$ arises from the comparison of the trivialization $\beta'(N/K, V)$ defined in (\ref{defbeta'}) and the trivialization $\lambda$ used to define the refined Euler characteristic
\[C_{N/K}=-\chi_{\ZpG,\BdR[G]}(M^\bullet,\lambda^{-1})\]
in (\ref{CNK}). In this section we will make this comparison explicit, and in the case $\chinr|_{G_N} = 1$ also 
compare $\Ucris$ to the generalization of Breuning's correction term $M_{N/K}$.

\begin{subsection}{\texorpdfstring{The computation of $\Ucris$}{The computation of Ucris}}

We define $\Ucris$ by the  equation
\begin{equation}\label{Ucris def}
i_{\ZpG, \BdR[G]}(\DEPT, \beta'(N/K, V)) = C_{N/K} + \partial^1_{\ZpG, \BdR[G]}(t) + \Ucris,
\end{equation}
where  $i_{\ZpG, \BdRG}$ denotes the natural isomorphism from  (\ref{yar 2}).
\begin{lemma}\label{conj equiv first}
Conjecture \ref{Ca} and Conjecture \ref{Cb} are equivalent.
\end{lemma}
\begin{proof}
Since $\dim_\Qp(V) = 1$ for $V = \Qp \tensor_\Zp T$, the representation $V$ is obviously potentially semistable. Conjecture \ref{Ca} states that
\[i_{\ZpG, \BdR[G]}((\DEPT, \delta(N/K, V)))\]
becomes $0$ in $K_0(\tilde\Lambda,\BdR[G])$, while Conjecture \ref{Cb} states that $\tilde R_{N/K}$ becomes $0$ in $K_0(\tilde\Lambda,\BdR[G])$. Recalling the definition of $\tilde R_{N/K}$ and $\Ucris$ Conjecture \ref{Cb} is therefore equivalent to the assertion that
\[i_{\ZpG, \BdR[G]}(\DEPT, \beta'(N/K, V))+\partial^1_{\ZpG, \BdR[G]}(\epsilon_D(N/K,V))\]
becomes trivial. Recall that in our special situation we have $\Gamma^*(V)=1$. Hence the definition of $\delta(N/K, V)$ in (\ref{yar 3})
implies
\[\delta(N/K, V)=\epsilon_D(N/K, V)\beta'(N/K, V).\]
The equivalence of the two conjectures follows easily.
\end{proof}

We proceed now to explicitly compute $\Ucris$.

We write $H_e^1(N, V)$ and $H_f^1(N, V)$ for the exponential and finite part of $H^1(N,V)$ as defined by Bloch and
Kato. For basic properties as used in our context we refer the reader to \cite[\S 1.4]{BenBer}.

For our calculation of $\Ucris$ we need to recall the fundamental exact sequences from \cite[(1.1) and (1.2)]{BenBer}:
\begin{equation}\label{splice 0}
0 \lra H^0(N, V) \lra \Dcris^N(V)^{\phi = 1} \lra t_V(N) \xrightarrow{\exp_{V}} H^1_e(N, V) \lra 0
\end{equation}
and 
\begin{equation}\label{splice 1}
0 \lra H^0(N, V) \lra \Dcris^N(V) \xrightarrow{1-\phi} \Dcris^N(V) \oplus t_V(N) \stackrel{\epsilon_V}\lra H^1_f(N, V) \lra 0.
\end{equation}
Note that $\epsilon_V|_{t_V(N)}=\exp_V$. We will also need  the dual of (\ref{splice 1}) with $V^*(1)$ in place of $V$:
\begin{equation}\label{splice 2}
0 \!\ra\! H^1_f(N, V^*(1))^* \!\ra\! \Dcris^N(V^*(1))^* \oplus t_{V^*(1)}^*(N) \!\ra\!  \Dcris^N(V^*(1))^* \!\ra\! H^0(N, V^*(1))^* \!\ra\! 0.
\end{equation}
By local duality we have 
\begin{equation}\label{local duality}
  H^0(N, V^*(1))^*  \cong H^2(N, V), \quad
  H^1_f(N, V^*(1))^* \cong \frac{H^1(N,V)}{H^1_f(N,V)}.
\end{equation}
Splicing together (\ref{splice 1}) and (\ref{splice 2}) we obtain the fundamental $7$-term exact sequence (\ref{7terms}) which induces the
trivialization $\delta'({N/K,V})$ from (\ref{cep3}). 

Recall that $V = \Qp(\chinr)(1)$ and $V^*(1) =  \Qp((\chinr)^{-1})$.

\begin{lemma}\label{degenerate}
We have:
\begin{enumerate}
\item [(a)] $t_{V^*(1)}(N)=0$.
\item [(b)]  $\dim_\Qp H_f^1(N,V^*(1))=\begin{cases}0&\text{if $\chinr|_{G_N}\neq 1$}\\1&\text{if $\chinr|_{G_N}=1$.}\end{cases}$
\item [(c)] $H_f^1(N,V)=H_e^1(N,V)\cong\begin{cases}H^1(N,V)&\text{if $\chinr|_{G_N}\neq 1$}\\U_N^{(1)}(\chinr)\otimes_\Zp\Qp&\text{if $\chinr|_{G_N}=1$.}\end{cases}$
\item[(d)] $H^1(N,V)=\left( \widehat{N_0^\times}(\chinr) \right)^{\Gal(N_0/N)}\otimes_\Zp\Qp$. In particular, if $\chinr|_{G_N}=1$, we have $H^1(N,V)=\widehat{N^\times}(\chinr)\otimes_\Zp\Qp$.
\end{enumerate}
\end{lemma}
\begin{proof}
\begin{enumerate}
\item[(a)] By \cite[page~148]{FO} or \cite[page~5]{IV} we have 
\[t_{V^*(1)}(N)\hookrightarrow \C_p((\chinr)^{-1})(-1)^{G_N}=0\]
and (a) follows.
\item[(b)] From \cite[Coroll.~3.8.4]{BK} and local duality we have 
\[\begin{split}\dim_\Qp H_f^1(N,V^*(1))&=\dim_\Qp(t_{V^*(1)}(N))+\dim_\Qp H^0(N,V^*(1))\\
&= \dim_\Qp H^0(N,V^*(1))^* = \dim_\Qp H^2(N,V)\end{split}\]
so that (b) follows from Lemma \ref{galois cohomology}.

\item[(c)] The proof of \cite[Lemma~6.3]{IV} 
shows that $1-\phi \colon \Dcris^N(V) \lra \Dcris^N(V)$ is an isomorphism. Hence $H^1_f(N,V)=H^1_e(N,V)$ is immediate from \cite[Lemma~1.3]{BenBer}.

By \cite[(3.8.5)]{BK} we have
\[\dim_\Qp H_f^1(N,V)+\dim_\Qp H_f^1(N,V^*(1))=\dim_\Qp H^1(N,V).\]
Combining this with (b) we obtain the equality $H^1_f(N, V) = H^1(N, V)$ in the case $\chinr|_{G_N}\neq 1$. For $\chinr|_{G_N}=1$, the isomorphism $H^1_f(N,V)\cong U_N^{(1)}(\chinr)\otimes_\Zp\Qp$ is well-known, see for example \cite[Example~3.9]{BK}.
\item[(d)] Immediate by Lemma \ref{galois cohomology}.
\end{enumerate}
\end{proof}

\end{subsection}
\begin{subsection}{\texorpdfstring{$\Ucris$ in the case $\chinr|_{G_N} \ne 1$}{Ucris in the case chinr|GN neq 1}}
We start considering the case $\chinr|_{G_N}\neq 1$.

By Lemma \ref{degenerate}, Lemma \ref{galois cohomology} (a) and the duality results (\ref{local duality}) the $7$-term exact sequence (\ref{7terms}) degenerates into the two exact sequences 
\[
0 \lra \Dcris^N(V) \xrightarrow{1- \phi} \Dcris^N(V) \oplus t_V(N) \lra H^1(N, V) \lra 0
\]
and 
\[
0 \lra \Dcris^N(V^*(1))^* \xrightarrow{1- \phi^*} \Dcris^N(V^*(1))^* \lra 0.
\] 
By (\ref{splice 0}), Lemma \ref{degenerate} and the fact that $\epsilon_V|_{t_V(N)}=\exp_V$, the first sequence is split by $\exp_V^{-1}$. 
Recall the definition of $i_{\ZpG, \BdR[G]}$ and the definition of the trivialisation in (\ref{triv def}). If we compare the 
trivialisations $\kappa=\kappa(M^\bullet,\lambda^{-1})$ and $\beta'(N/K, V)$, then we see that (\ref{Ucris def}) holds, if and only if

\begin{equation}\label{Ucriskappabeta}\Ucris =  \partial_{\ZpG, \QpG}^1 (\kappa \circ \beta'^{-1})-\partial^1_{\ZpG, \BdR[G]}(t),\end{equation}
where we interprete $\kappa \circ \beta'^{-1} \in \pi_1(V(\QpG))$ as an element in $K_1(\QpG)$. It follows that
\begin{equation}\label{yar 4}
\Ucris = \partial_{\ZpG, \QpG}^1([\Dcris^N(V), 1 - \phi]) - \partial_{\ZpG, \QpG}^1([\Dcris^N(V^*(1))^*, 1 - \phi^*]).
\end{equation}
Note that the sign originates from the fact
that the first term in the first sequence is in odd degree whereas the first term in the second sequence is in even degree. 

We introduce the following notation. If $x \in Z(\QpG)$ we let ${}^*x \in Z(\QpG)^\times$ denote the invertible element which
on the Wedderburn decomposition $Z(\QpG) = \bigoplus_{i=1}^r F_i$ for suitable finite extensions $F_i/\Qp$ is given by $({}^*x_i)$ with
${}^*x_i = 1$ if $x_i = 0$ and ${}^*x_i = x_i$ otherwise.

We now compute $\Ucris$ explicitly. Recall that $F = F_K = \varphi^{d_K}$.

\begin{lemma}\label{lemmaforUcris 1}
If $\chinr|_{G_N} \ne 1$, the endomorphism $1-\phi^*$ of $\Dcris^N(V^*(1))^*$ is an isomorphism. Furthermore we have
\[
\partial_{\ZpG, \QpG}^1([\Dcris^N(V^*(1))^*\!, 1 - \phi^*])\! =\! 
\hat\partial_{\ZpG, \QpG}^1({}^*((1-u^{d_K}F^{-1})e_I))\!\in\! K_0(\ZpG,\!\QpG).
\]
\end{lemma}

\begin{proof}
We let $v^*$ be a $\Qp$-basis of $V^*(1)$. In particular,
$\sigma v^* = \chinr(\sigma)^{-1} v^*$ for all $\sigma \in G_K$. 

We consider the $\Zp$-module $\left( \overline{\Zpnr}(\chinr_\Qp) \right)^{G_\Qp}$. For $s \in \overline{\Zpnr}$ one has
\[
s \in \left( \overline{\Zpnr}(\chinr_\Qp) \right)^{G_\Qp} \iff u \varphi(s) - s = 0.
\]
Adapting the proof of \cite[Lemma~V.2.1]{Neukirch92} we construct $\tur \in \overline{\Zpnr}$ such that
$(\tur)^{p-1} \equiv u^{-1} \pmod{p}$ and $u\varphi(\tur) = \tur$. It follows that $\tur \in \overline{\Zpnr}^\times$,
and furthermore, it is easy to show that $\left( \overline{\Zpnr}(\chinr) \right)^{G_\Qp} = \Zp\tur$. Note that $\tur$ is unique
up to multiplication by units of $\Zp$. 

By \cite[Th.~6.14]{FO} we know that $\overline{\Zpnr} \sseq \Bcris$, so that the element 
\[
e_0^* := (\tur)^{-1} \tensor v^*
\]
is contained in $\Bcris \tensor_\Qp V^*(1)$.  Since $\sigma(\tur) = \chinr_\Qp(\sigma)^{-1} \tur$ for all $\sigma \in G_\Qp$ we derive that
$e_0^* \in \left( \Bcris \tensor_\Qp V^*(1)\right)^{G_K} \sseq \Dcris^N(V^*(1))$. 
Let $N_1 \sseq N$ denote the maximal unramified subextension of $N/\Qp$. So we have $[N_1 : \Qp] = d_N$. 
Since $\dim_{N_1} \Dcris^N(V^*(1)) \le \dim_\Qp(V^*(1)) = 1$, the element $e_0^*$ generates 
$\Dcris^N(V^*(1))$ as an $N_1$-vector space,
i.e., $\Dcris^N(V^*(1)) = N_1 e_0^*$. By the $\varphi$-semilinearity of $\phi$ we obtain $\phi(e_0^*) = \chinr_\Qp(\varphi)e_0^* = ue_0^* $.

We now fix a normal basis element $\theta$ of $N_1/\Qp$. Then $w_j := \varphi^{-j}\theta e_0^*$, $j=0, \ldots, d_K-1$, is a
$\Qp[G/I]$-basis of $\Dcris^N(V^*(1))$. Let $\psi_l \in  \Dcris^N(V^*(1))^*$, $l = 0, \ldots, d_K-1$, be the dual basis. Then
\[
\Dcris^N(V^*(1))^* = \bigoplus_{l=0}^{d_K-1} \Qp[G/I]\psi_l
\]
and for $0<j<d_K$ we have
\[ \phi^*(\psi_l)(F^iw_j) = \psi_l(\phi(F^i\varphi^{-j}\theta e_0^*))  
= \psi_l(uF^i\varphi^{-j+1}\theta e_0^*)= u \psi_l(F^iw_{j-1}),\]
which is $u$ for $i=0$ and $l=j-1$, $0$ otherwise. For $j=0$ we get
\[ \phi^*(\psi_l)(F^iw_0) =\psi_l(\phi(F^i\theta e_0^*))= \psi_l(uF^{i+1}\varphi^{-d_K+1}\theta e_0^*)  = u\psi_l(F^{i+1}w_{d_K-1}),\]
which is $u$ for $i=-1$ and $l=d_K-1$, $0$ otherwise. Hence $\phi^*(\psi_l) = u \psi_{l+1}$ for $0\leq l<d_K-1$ and $\phi^*(\psi_{d_K-1}) = uF^{-1} \psi_{0}$.

With respect to the $\Qp[G/I]$-basis $\psi_0, \ldots, \psi_{d_K-1}$ the matrix associated to $1-\phi^*$ is given by
\[
A=\begin{pmatrix}1&0&0&\cdots&0&-uF^{-1}\\
-u&1&0&\cdots&0&0\\
0&-u&1&\cdots&0&0\\
\vdots&\vdots&\vdots&\ddots&\vdots&\vdots\\
0&0&0&\cdots&1&0\\
0&0&0&\cdots&-u&1\end{pmatrix}.
\]
Since $\det(A) = 1 - u^{d_K}F^{-1}$ the lemma follows. Note that this is invertible if and only if $\chinr|_{G_N}\neq 1$.
\end{proof}

\begin{lemma}\label{lemmaforUcris 2}
The endomorphism $1 - \phi$ of $\Dcris^N(V)$ is an isomorphism and we have
\[
\partial_{\ZpG, \QpG}^1([\Dcris^N(V), 1 - \phi])\! = \!
\hat\partial_{\ZpG, \QpG}^1({}^*((1-p^{-d_K}u^{-d_K}F)e_I)) \in K_0(\ZpG, \QpG).
\]
\end{lemma}
\begin{proof}
The proof is analogous to the previous lemma, following the computations in \cite[Lemma~6.3]{IV}. 
\end{proof}

\end{subsection}
\begin{subsection}{\texorpdfstring{$\Ucris$ in the case $\chinr|_{G_N} = 1$}{Ucris in the case chinr|GN = 1}}
We now study the case $\chinr|_{G_N} = 1$.

By Lemma \ref{degenerate}, Lemma \ref{lemmaforUcris 2} (note that this lemma holds also for $\chinr|_{G_N}=1$), (\ref{splice 0}), (\ref{splice 1}), (\ref{splice 2}), (\ref{local duality}) and Lemma \ref{galois cohomology} (b) the $7$-term exact sequence (\ref{7terms}) degenerates
to
\begin{equation} 0 \lra \Dcris^N(V) \xrightarrow{1- \phi} \Dcris^N(V) \lra 0, \end{equation}
\begin{equation}\label{degenerate 4}
 0 \lra t_V(N) \xrightarrow{\exp_V} H^1(N,V) \lra \frac{H^1(N,V)}{H^1_f(N, V)} \lra 0 \end{equation}
and 
\begin{equation}\label{degenerate 3}
0 \lra \frac{H^1(N,V)}{H^1_f(N, V)} \lra \Dcris^N(V^*(1))^* \xrightarrow{1- \phi^*} \Dcris^N(V^*(1))^* \lra H^2(N, V) \lra 0.
\end{equation}

In the case $\chinr|_{G_N} = 1$ taking cohomology commutes with twisting, so that we may and will assume that $V = \Qp(1)$ (see Remark \ref{assumption} below).
By \cite[Example~3.10.1]{BK} and Lemma \ref{degenerate} the short exact sequence (\ref{degenerate 4}) fits into the following diagram:
\begin{equation}
\begin{split}
\xymatrix{
0 \ar[r]& t_V(N) \ar[r]^-{\exp_V}\ar[d]&H^1(N, V) \ar[r]\ar[d]^-{\partial^{-1}_{Ku}}   & \frac{H^1(N, V)}{H^1_f(N, V)}  \ar[r]\ar[d]   &0\\
0 \ar[r]& N \ar[r]^-{\exp}\ar[d]^-{=}& \widehat{N^\times} \tensor_\Zp \Qp \ar[r]\ar[d]^-{=}   & 
    \frac{\widehat{N^\times} \tensor_\Zp \Qp}{U_N^{(1)} \tensor_\Zp \Qp}  \ar[r]\ar[d]^-{\nu_N}    &0\\
0 \ar[r]& N \ar[r]^-{\exp}& \widehat{N^\times} \tensor_\Zp \Qp \ar[r]^-{\nu_N}  & \Qp  \ar[r] &0
}
\end{split}
\end{equation}
where $\exp$ is induced by the $p$-adic exponential map and $\partial_{Ku}$ denotes the connecting homomorphism induced by the Kummer sequence. 
We have a morphism $\gamma_1$ in $V(\QpG)$
\[\begin{split}
\gamma_1 &\colon 
\left[  \frac{H^1(N,V)}{H^1_f(N, V)} \right]_\QpG^{-1}  \tensor \left[  H^2(N,V) \right]_\QpG \\
&\to
\left[  \frac{H^1(N,V)}{H^1_f(N, V)} \right]_\QpG^{-1} \!\!\!\!\!\!\!\!\!\tensor \left[ \Dcris^N(V^*(1))^*  \right]_\QpG 
\tensor \left[ \Dcris^N(V^*(1))^*  \right]_\QpG^{-1} \tensor \left[  H^2(N,V) \right]_\QpG \\
&\to \unit_{V(\QpG)}
\end{split}\]
where the last arrow is induced by the exact sequence (\ref{degenerate 3}).
From Lemma \ref{degenerate} together with the isomorphism $\inv_N \colon H^2(N, V) \lra \Qp$ we deduce a morphism $\gamma_2$ in $V(\QpG)$
\[
\gamma_2 \colon \left[  \frac{H^1(N,V)}{H^1_f(N, V)} \right]_\QpG^{-1}  \tensor \left[  H^2(N,V) \right]_\QpG \lra \unit_{V(\QpG)}
\]
which is induced by the valuation map $\nu_N$. We set $\gamma := \gamma_2\circ \gamma_1^{-1} \in \pi_1(V(\QpG)) \cong K_1(\QpG)$. 
The comparison of the trivializations $\kappa = \kappa(M^\bullet, \lambda^{-1})$ and $\beta'(N/K, V)$ shows (as explained in some more detail in the appendix) that
\begin{equation}\label{yar 5}
\Ucris = \partial_{\ZpG, \QpG}^1([\Dcris^N(V), 1 - \phi]) + \partial_{\ZpG, \QpG}^1(\gamma).
\end{equation}
We now proceed to compute $\Ucris$ explicitly. The first summand is computed by Lemma \ref{lemmaforUcris 2}.
For non-trivial characters the same computation as in the proof of Lemma \ref{lemmaforUcris 1} shows that
\[\begin{split}
\partial_{\ZpG, \QpG}^1([(1-e_G) \Dcris^N(V^*(1)), 1 - \phi^*]) &= 
\hat\partial_{\ZpG, \QpG}^1({}^*((   1-F^{-1})(e_I-e_G)))\\
&=\hat\partial_{\ZpG, \QpG}^1({}^*((   1-F^{-1})e_I)).
\end{split}\]

It therefore remains to compute the trivial component of $\gamma$. For this we will need some preparations.

We recall that the exact sequence (\ref{degenerate 3}) is derived by dualizing and then using local duality from
\begin{equation}\label{no dual}
0 \lra H^0(N, V^*(1)) \stackrel{\iota}\lra \Dcris^N(V^*(1)) \xrightarrow{1- \phi} \Dcris^N(V^*(1)) \lra H^1_f(N, V^*(1)) \lra 0,
\end{equation}
which is (\ref{splice 0}) with $V$ replaced by $V^*(1)$. 
Recall that $N_1 = N \cap \Kur$.
The sequence (\ref{no dual}) is explicitly given by
\[
0 \lra \Qp \stackrel{\iota}\lra N_1 \xrightarrow{1- \varphi} N_1 \stackrel{\xi}\lra H^1_f(N, V^*(1)) \lra 0,
\]
where $\xi$ is the connecting homomorphism. Explicitly, for $x \in N_1$ and $\sigma \in G_N$ we have $\xi_x(\sigma) = \sigma(y) - y$,
if $y \in N_0$ is such that $(1-\varphi)(y) = x$. By \cite[Example~3.9]{BK} $H^1_f(N, V^*(1)) \sseq \homcont(G_N, \Qp) = H^1(N, V^*(1))$
are the unramified homomorphisms, hence
\[
H^1_f(N, V^*(1)) = \homcont(\overline{\langle F_N \rangle}, \Qp).
\]
We will always identify $H^1_f(N, V^*(1))$ with $\Qp$ via
\[
\mu \colon H^1_f(N, V^*(1)) \lra \Qp, \quad \eta \mapsto \eta(F_N).
\]
Recall that $\T_{N_1/\Qp}$ denotes the trace map. Since
\[\begin{split}
  \xi_x(F_N) &= F_N(y) - y = \varphi^{d_N}(y) - y \\
  &=\sum_{i=0}^{d_N -1}\left( \varphi^{i+1}(y) - \varphi^i(y) \right) = - \sum_{i=0}^{d_N -1}\varphi^i(x) = -\T_{N_1/\Qp}(x),
\end{split}\]
we have the commutative diagram
\begin{equation}\label{commdiagtriv}
\begin{split}
\xymatrix{
0  \ar[r]& \Qp \ar[r]^{\iota} \ar[d]^{=} &  N_1 \ar[r]^{1- \varphi}\ar[d]^{=} & N_1 \ar[r]^{\xi\phantom{xxxxx}}\ar[d]^{=} & H^1_f(N, V^*(1)) \ar[r] \ar[d]^{\mu} & 0 \\
0  \ar[r]& \Qp \ar[r]^{\iota}           &  N_1 \ar[r]^{1- \varphi}          & N_1 \ar[r]^{-\T_{N_1/\Qp}}    & \Qp \ar[r] & 0.
}
\end{split}
\end{equation}
By dualizing we obtain
\begin{equation}\label{Ucris diagram}
\begin{split}
\xymatrix{
0  \ar[r]& \Qpx \ar[r]^{(-\T_{N_1/\Qp})^*} \ar[d]^{\mu^*} &  N_1^* \ar[r]^{(1- \varphi)^*}\ar[d]^{=} & N_1^* \ar[r]^{\iota^*}\ar[d]^{=} & \Qpx \ar[r] \ar[d]^{=} & 0 \\
0  \ar[r]& H^1_f(N, V^*(1))^* \ar[r]^{\phantom{xxxx}\xi^*} \ar[d]^{s_1}  &  N_1^* \ar[r]^{(1- \varphi)^*}          & N_1^* \ar[r]^{\iota^*}    & \Qpx \ar[r]\ar[d]^{s_2} & 0 \\
         & H^1(N,V) / H^1_f(N,V) \ar[d]^{\cong} & & & H^2(N,V) \ar[d]^{\inv_N} & \\
         & \frac{\widehat{N^\times} \tensor \Qp}{U_N^{(1)} \tensor \Qp}
         \ar[rrr]^{\nu_N} &&& \Qp, &
}
\end{split}
\end{equation}
where $s_1$ and $s_2$ denote the duality maps from (\ref{local duality}).

\begin{lemma}\label{yal 1}\begin{enumerate}
\item[(a)] Let $\eta_1 \in \Qpx = \Hom_\Qp(\Qp, \Qp)$ be defined by $\eta_1(1) = 1$. Then  $s_1(\mu^*(\eta_1)) = [\pi_N]$ where
$[\pi_N]$ denotes the class of $\pi_N$ in $\frac{\widehat{N^\times} \tensor \Qp}{U_N^{(1)} \tensor \Qp}$.
\item[(b)] Let $\eta_2 \in \Qpx$ be defined by $\eta_2(1) = 1$. Then $\inv_N(s_2(\eta_2)) = 1$.
\end{enumerate}
\end{lemma}

\begin{proof}
\begin{enumerate}
\item[(a)] We want to prove that $\langle \pi_N, \cdot \rangle = \mu^*(\eta_1)$, where
\begin{equation}\label{fix duality pairing}
\langle \cdot, \cdot \rangle \colon  H^1(N, V)\times H^1(N, V^*(1)) \lra H^2(N, \Qp(1)) \xrightarrow{\inv_N} \Qp
\end{equation}
denotes the duality pairing.  As in the proof of Lemma \ref{lemmaforUcris 1} we let 
$\theta \in N_1$ be a normal basis element. In addition, we require the property $\T_{N_1/\Qp}(\theta) = -1$. 
Since $H^1_f(N, V^*(1)) = \Qp\xi_\theta$, it is enough to prove that $\langle  \pi_N, \xi_\theta \rangle = \mu^*(\eta_1)(\xi_\theta)$.
By the defining property  of $\theta$ we obtain
\[
\mu^*(\eta_1)(\xi_\theta) = \eta_1(\mu(\xi_\theta)) = \eta_1(-\T_{N_1/\Qp}(\theta)) = \eta_1(1)=1.
\] 
By \cite[Coroll.~7.2.13]{NSW} we obtain
\[\langle  \pi_N, \xi_\theta \rangle =  \inv_N(\pi_N \cup \xi_\theta) =  \xi_\theta(\rec_N(\pi_N))=  \xi_\theta(F_N) = -\T_{N_1/\Qp}(\theta) = 1.\]

\item[(b)] If 
\[
\langle \cdot, \cdot \rangle \colon  H^2(N, V) \times H^0(N, V^*(1)) \lra H^2(N, \Qp(1)) \xrightarrow{\inv_N} \Qp,
\]
denotes the local duality pairing, then by definition $\langle s_2(\eta_2),\cdot\rangle=\eta_2$ and therefore
\[\inv_N(s_2(\eta_2))=\langle s_2(\eta_2),1\rangle=\eta_2(1)=1.\]
\end{enumerate}
\end{proof}

As explained in \cite[Sec.~2.5]{BuFl2001} we can and will identify here
$V(\Qp)$ with the category of graded invertible modules with the commutativity constraint $\psi$ 
as in \cite[(5)]{BuFl2001}. Explicitly, for objects $(L, \alpha)$ and $(M, \beta)$ and $l\in L$, $m\in M$ one has
\[
\psi(l \tensor m) = (-1)^{\alpha\beta}m \tensor l.
\]
We also point out that for an object $(L, \alpha)$ we fix a right inverse by $(L, \alpha)^{-1} := (\Hom_\Qp(L, \Qp), -\alpha)$, i.e. $(L,\alpha)\otimes (L,\alpha)^{-1}=(\Qp,0)$.

\begin{lemma}\label{lemmaforUcris 3}
Under our assumption that $\chinr = 1$,
\[
\partial^1_{\ZpG, \QpG}(\gamma) = 
\frac{{}^*(d_{N/K}e_{G})}{{}^*((1-F^{-1})e_I)}.
\]
\end{lemma}

\begin{proof} As already mentioned the computation of the $(1-e_G)$-component is analogous to the proof of Lemma \ref{lemmaforUcris 1}. 

For the computation of the trivial component, by Lemma \ref{yal 1} we can identify $\eta_1^{-1}\otimes\eta_2$ with $[\pi_N]^{-1}\otimes 1$, so that we need to compute $\gamma_1(\eta_1^{-1}\otimes\eta_2)$ and $\gamma_2([\pi_N]^{-1}\otimes 1)$, where by abuse of notation we have denoted the $e_G$-component of $\gamma_1$ and $\gamma_2$ again by the same symbol.

Using the commutativity constraint we obtain
\begin{equation}\label{yar 6}
\gamma_2( [\pi_N]^{-1} \tensor 1 ) = -1.
\end{equation}

For the computation of the trivial component  it remains by Lemma \ref{yal 1} 
to compute the image of $\eta_1^{-1} \tensor \eta_2$ under the morphism $\gamma_1$.
By (\ref{yar 6}) we have to show 
\[
\gamma_1( \eta_1^{-1} \tensor \eta_2 ) = -\frac{1}{d_{N/K}}.
\]
The $e_G$-component of the top sequence in diagram (\ref{Ucris diagram}) is given by
\[
0 \lra \Qpx \xrightarrow{(-\T_{N_1/\Qp})^*} e_GN^*_1 \xrightarrow{(1-\varphi)^*} e_GN_1^* \stackrel{\iota^*}\lra \Qpx \lra 0.
\]  
The elements $e_G\psi_0, \ldots, e_G\psi_{d_K-1}$ constitute a $\Qp$-basis of $e_GN^*_1$, where $\psi_1,\dots\psi_{d_K-1}$ are defined as in the proof of Lemma \ref{lemmaforUcris 1}. We set
\[
\psi_i' := e_G\psi_{i}, \quad i=0, \ldots, d_K-1.
\]
Since $(1 - \varphi)^*(\psi_i') = \psi_i' - \psi_{i+1}'$ we set $z_i :=  \psi_i' - \psi_{i+1}'$, so that 
\[
Z := \langle z_0, \ldots, z_{d_K-2} \rangle_\Qp
\]
is the image of $(1 - \varphi)^*$. We consider the two short exact sequences 
\begin{equation}\label{Ucris split 1}
  0 \lra \Qpx \xrightarrow{(-\T_{N_1/\Qp})^*} e_GN^*_1 \xrightarrow{(1-\varphi)^*} Z \lra 0 
\end{equation}
\begin{equation}\label{Ucris split 2}
  0 \lra Z \stackrel{\sseq}\lra e_GN_1^* \stackrel{\iota^*}\lra \Qpx \lra 0
\end{equation}
and compute the image of $\eta_1^{-1} \tensor \psi' \tensor (\psi')^{-1} \tensor \eta_2$ with $\psi' := \psi_0' \wedge \ldots \wedge \psi_{d_K-1}'$
under the natural map
\[
[\Qpx]^{-1}_\Qp \picprod [e_GN_1^*]_\Qp \picprod  [e_GN_1^*]^{-1}_\Qp \picprod [\Qpx]_\Qp \lra {\bf 1}_{V(\Qp)},
\]
which is induced by the top sequence of (\ref{Ucris diagram}). 

We express $(-\T_{N_1/\Qp})^*(\eta_1)$ in terms of the $\Qp$-basis $\psi_0, \ldots, \psi_{d_N-1}$ of $N_1^*$ which is dual to the basis $w_0,\dots,w_{d_N-1}$ of $N_1$ defined by $w_j=\varphi^{-j}\theta$. As usual $\theta$ is a normal basis element with the additional property that $\trace_{N/\Qp}(\theta)=-1$. One has
\[
(-\T_{N_1/\Qp})^*(\eta_1)(w_i) = \eta_1(-\T_{N_1/\Qp}(w_i)) = \eta_1(1) = 1
\]
for all $i = 0, \ldots, d_N-1$.
We set $b_1 := (-\T_{N_1/\Qp})^*(\eta_1)$ and obtain
\[
b_1 =  \sum_{i=0}^{d_N-1}\psi_i =  \left( \sum_{j=0}^{d_{N/K}-1} \varphi^{d_Kj} \right) \sum_{i=0}^{d_K-1}\psi_i =  d_{N/K} \sum_{i=0}^{d_K-1}\psi_i'.
\]
Hence $(b_1, \psi_0', \ldots, \psi_{d_K-2}') = (\psi_0', \ldots, \psi_{d_K-1}') A$ with the matrix $A \in \mathrm{Gl}_{d_K}(\Qp)$ given by
\[A = \begin{pmatrix}
d_{N/K} & 1 & 0 & \cdots & 0 \\ 
d_{N/K} & 0 & 1 & \cdots & 0 \\ 
\vdots & \vdots & \vdots & \ddots & \vdots \\
d_{N/K} & 0 & 0 & \cdots & 1  \\ 
d_{N/K} & 0 & 0 & \cdots & 0 
\end{pmatrix}.\]
A splitting $\tau$ of $(1-\varphi)^*$ is given by $\tau(z_i) := \psi_i', i = 0, \ldots, d_K-2$.
Therefore, under the isomorphism 
\[
[e_GN_1^*]_\Qp \cong [\Qpx]_\Qp \tensor [Z]_\Qp
\]
induced by (\ref{Ucris split 1}) we obtain
\[
\psi' \mapsto \det(A)^{-1} (\eta_1 \tensor z),
\]
where $z := z_0 \wedge \ldots \wedge z_{d_K-2}$. 

We define 
\[
\sigma \colon \Qpx \lra N_1^*, \quad \eta_2 \mapsto - \frac{1}{d_N} \sum_{i=0}^{d_N-1}\psi_i.
\]
Then $\sigma$ is a splitting of $\iota^*$ because 
\[
\iota^*(\sigma(\eta_2))(1) = - \frac{1}{d_N}\sum_{i=0}^{d_N-1}\psi_i(1) = \frac{1}{d_N}\sum_{i=0}^{d_N-1}\psi_i 
\left( \sum_{j=0}^{d_N-1}w_j \right) = 1.
\] 
We put $b_2 := \sigma(\eta_2)$. Then
\[
b_2 = - \frac{1}{d_N} \sum_{i=0}^{d_N-1} \psi_i = - \frac{1}{d_K} \sum_{i=0}^{d_K-1} \psi_i'.
\]
Hence we have $(z_0, \ldots, z_{d_K-2}, b_2) = (\psi_0', \ldots, \psi_{d_K-1}') B$ with the matrix $B \in \mathrm{Gl}_{d_K}(\Qp)$ given by
\[B = \begin{pmatrix}
1 & 0 & \cdots & 0 & 0 & - \frac{1}{d_K} \\ 
-1 & 1 & \cdots & 0 & 0 & - \frac{1}{d_K}  \\  
0 & -1 & \cdots & 0 & 0 & - \frac{1}{d_K}  \\  
\vdots & \vdots & \ddots & \vdots& \vdots & \vdots \\  
0 & 0 & \cdots & -1 & 1 & - \frac{1}{d_K} \\  
0 & 0 & \cdots & 0 & -1 & - \frac{1}{d_K}
\end{pmatrix}.\]
Therefore 
\[
\psi' \mapsto \det(B)^{-1} (z \tensor \eta_2)
\]
under the isomorphism induced by (\ref{Ucris split 2}), respectively $(\psi')^{-1} \mapsto \det(B) (\eta_2^{-1} \tensor z^{-1})$.
Note that
\[
\det(A) = (-1)^{d_K+1} d_{N/K}, \quad \det(B) = -1.
\]
Putting things together we obtain
\[\begin{split}
 & \eta_1^{-1} \tensor \psi' \tensor (\psi')^{-1} \tensor \eta_2 \\
 &\quad\mapsto (-1)^{d_K+1} \frac{1}{d_{N/K}} (\eta_1^{-1} \tensor (\eta_1 \tensor z)  \tensor (\psi')^{-1} \tensor \eta_2) \\
 &\quad\mapsto (-1)^{d_K} \frac{1}{d_{N/K}} (z  \tensor (\psi')^{-1} \tensor \eta_2) \\
 &\quad\mapsto (-1)^{d_K+1} \frac{1}{d_{N/K}} (z  \tensor (\eta_2^{-1} \tensor z^{-1}) \tensor \eta_2) \\
 &\quad\mapsto (-1)^{d_K+1+d_K-1} \frac{1}{d_{N/K}} (z  \tensor z^{-1} \tensor \eta_2^{-1} \tensor \eta_2) \\
 &\quad\mapsto \frac{1}{d_{N/K}} (\eta_2^{-1} \tensor \eta_2) \\
 &\quad\mapsto - \frac{1}{d_{N/K}}.
\end{split}\]

\end{proof}

\begin{remark}\label{assumption}
In the case $\chinr|_{G_N}=1$ but $\chinr\neq 1$ the restriction of $\chinr$ to $N$ defines an abelian character of $G = \Gal(N/K)$. 
We write 
\[
e_{\chinr} = \frac{1}{|G|} \sum_{\sigma \in G} \chinr(\sigma) \sigma^{-1}
\]
for the associated idempotent. The calculations in this case are basically the same as in the case $V = \Qp(1)$. 
Diagram (\ref{commdiagtriv}) becomes

\[\xymatrix{
0  \ar[r]& \Qp((\chinr)^{-1}) \ar[r]^{\iota} \ar[d]^{=} &  N_1((\chinr)^{-1}) \ar[r]^{f}\ar[d]^{=} & N_1((\chinr)^{-1}) \ar[r]^{\xi\phantom{xxxxx}}\ar[d]^{=} & H^1_f(N, V^*(1)) \ar[r] \ar[d]^{\mu} & 0 \\
0  \ar[r]& \Qp((\chinr)^{-1}) \ar[r]^{\iota}           &  N_1((\chinr)^{-1}) \ar[r]^{f}          & N_1((\chinr)^{-1}) \ar[r]^{-\T_{N_1/\Qp}}    & \Qp((\chinr)^{-1}) \ar[r] & 0.
}\]

Note that the maps in (\ref{commdiagtriv}) and also their splittings do not change: 
this is the reason why we write $f$ instead of $1-\varphi$ which otherwise could falsely be interpreted
as $x \mapsto x - \chinr(\varphi)\varphi(x)$. 
For the rest of the proof of the second part of the next proposition one just has to take into account that 
only the action of $G$ on the modules changes by the twist so that we obtain the non-trivial contribution from
non-torsion cohomology for the character $\chinr$.
\end{remark}

We summarize the main result of Section \ref{Ucris}.
\begin{prop}\label{Ucris prop}

a) If $\chinr|_{G_N} \ne 1$, then $\Ucris = \partial(m_{N/K})$ with
\[
m_{N/K} = \frac{{}^*((1-p^{-d_K}u^{-d_K}F)e_I)}{{}^*((1-u^{d_K}F^{-1})e_I)}.
\]

b) If $\chinr|_{G_N} = 1$, then $\Ucris = \partial(m_{N/K})$ with
\[
m_{N/K} = \frac{{}^*((1-p^{-d_K}u^{-d_K}F)e_I) \cdot {}^*(d_{N/K}e_{\chinr})}{{}^*((1-u^{d_K}F^{-1})e_I)}.
\]
\end{prop}

\begin{proof}
Combine the equations (\ref{yar 4}) and  (\ref{yar 5}) with Lemma \ref{lemmaforUcris 1}, \ref{lemmaforUcris 2}, \ref{lemmaforUcris 3}
and Remark \ref{assumption}.
\end{proof}

\begin{remark}\label{Breu MNK}
  In the case $\chinr = 1$ we precisely recover the correction term $M_{N/K}^\Breu$ in Breuning's conjecture \cite[Conj.~3.2]{Breuning04}.
\end{remark}
\end{subsection}
\end{section}

\begin{section}{Epsilon constants of Weil-Deligne representations}\label{eps}
In this section we study the epsilon constant defined in \cite[Sec.~3]{IV} and in the situation described in Subsection \ref{sit} we rewrite it in terms of Galois Gau\ss\  sums. 

\begin{subsection}{Epsilon constants and Galois Gau\ss\ sums}
If $L/\Qp$ is a finite extension and $W$ is an Artin representation of $G_L$,  we let $\frf(W)$ denote the Artin conductor of $W$. Let $\psi=\psi_\xi$ be the standard additive character, as defined in Subsection \ref{subsectCEP}. By \cite[Lemma~3.1]{IV} we have

\begin{equation}\label{eps explicit}
\epsilon_D(N/K, V) = \left( \epsilon\left( \Ind_{K/\Qp}(\chi), -\psi, \mu_\Qp \right)^{-1} 
\chiQpnr\left( \frf(\Ind_{K/\Qp}(V_\chi))^{-1} \right) \right)_{\chi \in \Irr(G)}.
\end{equation}

Here we view $\chiQpnr$ as a character $\Qptimes \xrightarrow{\rec} G_\Qp^\mathrm{ab} \xrightarrow{\chiQpnr} \Qptimes$ via the local reciprocity map.

\begin{remark}\label{rec normalization}
Contrary to the convention in \cite{IV} we normalize reciprocity maps by sending a uniformizer $\pi_L$ to the arithmetic Frobenius map $F_L$.
\end{remark}

\begin{lemma}\label{abelian eps}
Assume that $M/L$ is an abelian extension of $p$-adic number fields. Let $\eta$ be an irreducible character of $\Gal(M/L)$. Let $\frD_L = \pi_L^{s_L}\OL$ be the different of $L/\Qp$. Then
\[\epsilon(\eta, \psi, \mu_L) = p^{d_Ls_L}\tau_L(\eta),\]
where $\tau_L(\eta)$ is a Galois Gau\ss\ sum (for the definition see e.g. \cite[page~1184]{PickettVinatier}).
\end{lemma}
\begin{proof}
We will apply some of the properties listed in \cite[Sec.~2.3]{BenBer}. If $\eta$ is unramified we derive from (1) of loc.cit. the equality
\[\epsilon(\eta, \psi, \mu_L) = \eta(\pi_L^{-s_L})p^{d_Ls_L}.\]
We point out that we have $\eta(\pi_L^{-s_L})$ (and not $\eta(\pi_L^{s_L})$ as in \cite{BenBer} or \cite{IV}) due to our different normalization of the reciprocity map. By the definition of Galois Gau\ss\  sums one has $\tau_L(\eta) = \eta(\pi_L^{-s_L})$ and the lemma follows for unramified characters $\eta$.

Let now $\eta$ be a ramified character. We write $\frf(\eta) = \pi_L^{m_\eta}\OL$ for the conductor of $\eta$. By (1) of loc.cit. and \cite[(3.4.3.2)]{Del73} we have 
\[\begin{split}
\epsilon(\eta,\psi,\mu_L)&=\int_{\{v_L(x)=-s_L - m_\eta\}}\eta(x)\psi(x)d\mu_L
\\
&=\int_{x\in\calO_L^\times}\frac{1}{\mu_L(\p_L^{s_L + m_\eta})}\eta\left(\frac{x}{\pi_L^{s_L + m_\eta}}\right)\psi\left(\frac{x}{\pi_L^{s_L + m_\eta}}\right)d\mu_L\\
& =\frac{1}{\mu_L(\p_L^{s_L + m_\eta})}\int_{x\in\OL^\times}\eta\left(\frac{x}{\pi_L^{s_L+m_\eta}}\right)\psi\left(\frac{x}{\pi_L^{s_L+m_\eta}}\right)d\mu_L \\
&= \frac{1}{\mu_L(\p_L^{s_L + m_\eta})}\mu_L(U_L^{(m_\eta)}) \sum_{x \in U_L/U_L^{(m_\eta)}} 
    \eta\left(\frac{x}{\pi_L^{s_L+m_\eta}}\right)\psi\left(\frac{x}{\pi_L^{s_L+m_\eta}}\right) \\
&= \frac{\mu_L(U_L^{(m_\eta)})}{\mu_L(\p_L^{s_L + m_\eta})} \tau_L(\eta)= p^{d_Ls_L}\tau_L(\eta),
\end{split}\]
where the next-to-last equality results from the definition of Galois Gau\ss\  sums.
\end{proof}
We will now use Brauer induction in degree $0$ to deduce the following proposition.

\begin{prop}\label{general eps}
Let $\chi$ be an arbitrary character of $G_\Qp$ with open kernel. Then
\[\epsilon(\chi, \psi, \mu_\Qp) = \tau_\Qp(\chi).\]
\end{prop} 

\begin{proof}
By \cite[Exercise~10.6]{SerreRep} there exist integers $a_\phi$ such that
\[\chi - \chi(1)1_{G_\Qp} = \sum_\phi a_\phi \Ind_{K_\phi/\Qp}(\phi - 1)\]
with linear characters $\phi$ of $G_\Qp$ and $K_\phi = (\Qpc)^{\ker(\phi)}$. We will again apply some of
the properties (1) - (6) listed in \cite[Sec.~2.3]{BenBer}. By (2) of loc.cit. local constants are additive and Lemma \ref{abelian eps} applied for $M=L=\Qp$ together with the fact that $\tau_K(1_{G_\Qp})=1$ shows that $\epsilon(1_{G_\Qp}, \psi, \mu_\Qp) = 1$. Therefore we conclude
\[\begin{split}
  \epsilon(\chi, \psi, \mu_\Qp) 
&= \epsilon(\chi - \chi(1) 1_{G_\Qp}, \psi, \mu_\Qp) \epsilon(1_{G_\Qp}, \psi, \mu_\Qp)^{\chi(1)} \\
&= \epsilon(\chi - \chi(1) 1_{G_\Qp}, \psi, \mu_\Qp).
\end{split}\]
We now use (2) and (4) of loc.cit. in order to derive
\[\begin{split}
  \epsilon(\chi, \psi, \mu_\Qp) 
&= \prod_\phi \epsilon(\Ind_{K_\phi/\Qp}(\phi - 1), \psi, \mu_\Qp)^{a_\phi} \\
&= \prod_\phi \epsilon(\phi - 1, \psi_\phi, \mu_{K_\phi})^{a_\phi},
\end{split}\]
where we write $\psi_\phi \colon K_\phi \lra (\Qpc)^\times$ for the standard additive character of $K_\phi$. Because of $\dim (\Ind_{K_\phi/\Qp}(\phi - 1)) = 0$ the $\lambda$-constants which arise in (4) of loc.cit. do not show up.

We again apply additivity followed by our Lemma \ref{abelian eps} and obtain
\[\begin{split}
  \epsilon(\chi, \psi, \mu_\Qp) 
&= \prod_\phi \epsilon(\phi, \psi_\phi, \mu_{K_\phi})^{a_\phi} \cdot 
    \prod_\phi \epsilon(1, \psi_\phi, \mu_{K_\phi})^{-a_{\phi}} \\
&= \prod_\phi \left( p^{d_{K_\phi}s_{K_\phi}} \tau_{K_\phi}(\phi) \right)^{a_\phi} \cdot 
    \prod_\phi \left( p^{d_{K_\phi}s_{K_\phi}} \tau_{K_\phi}(1) \right)^{-a_{\phi}} \\
&= \prod_\phi \left( \tau_{K_\phi}(\phi)\tau_{K_\phi}(1)^{-1} \right)^{a_\phi}.
\end{split}\]
Now we recall from \cite{Froehlich83}  that Galois Gau\ss\  sums are additive, inductive in degree zero and equal to $1$ for the trivial character. It follows that
\[\begin{split}
\epsilon(\chi, \psi, \mu_\Qp) 
&=\prod_\phi \tau_{K_\phi}(\phi - 1)^{a_\phi} 
=\prod_\phi \tau_\Qp(\Ind_{K_\phi/\Qp}(\phi - 1))^{a_\phi} \\
&=\tau_\Qp(\chi - \chi(1) 1_{G_\Qp}) 
=\tau_\Qp(\chi).
\end{split}\]
\end{proof}

\end{subsection}

\begin{subsection}{Epsilon constants for unramified twists}

We apply the results of this section to unramified twists of $\Zp(1)$.

\begin{prop}\label{eps coroll}
Assume that $N/K$ is an arbitrary Galois extension of $p$-adic number fields with Galois group $G$.
Let $\chiQpnr$ be as usual and $\chiQpnr(\varphi) = u$. 
Let $\chi \in \Irr(G)$ and write $\frf(\chi) = \pi_K^{m_\chi}\OK$ for the Artin conductor.
Then
\[\epsilon_D(N/K, V)_\chi = \left( \det(\Ind_{K/\Qp}\chi) \right)(-1) \cdot u^{-d_K(s_K\chi(1) + m_\chi)} \cdot \tau_\Qp( \Ind_{K/\Qp}(\chi) )^{-1}.\]
\end{prop}

\begin{proof}
By property (3) of \cite[Sec.~2.3]{BenBer} we have
\[\epsilon( \Ind_{K/\Qp}(\chi), -\psi, \mu_\Qp) = \det\left(\Ind_{K/\Qp}(\chi) \right)(-1) \cdot \epsilon(\Ind_{K/\Qp}(\chi), \psi, \mu_\Qp).\]
By Proposition \ref{general eps} we conclude further
\begin{equation}\label{eps 2}
 \epsilon( \Ind_{K/\Qp}(\chi), -\psi, \mu_\Qp) = \det\left(\Ind_{K/\Qp}(\chi) \right)(-1) \cdot
                                                \tau_\Qp(\Ind_{K/\Qp}(\chi)).
\end{equation}
Finally, by \cite[VII.11.7]{Neukirch92}, we deduce
\[
\frf(\Ind_{K/\Qp}(\chi)) = \frd_{K}^{\chi(1)} N_{K/\Qp}(\frf(\chi)) = p^{d_Ks_K\chi(1)} p^{d_Km_\chi},
\]
where $ \frd_{K}$ denotes the discriminant of $K / \Qp$. Since $\chiQpnr(p) = \chiQpnr(\varphi) = u$ we obtain
\begin{equation}\label{eps 3} 
\chiQpnr(\frf(\Ind_{K/\Qp}(\chi))) = u^{d_K (s_K\chi(1)+m_\chi)}.
\end{equation}
Combining (\ref{eps explicit}), (\ref{eps 2}) and  (\ref{eps 3}) the assertion of the proposition follows.
\end{proof}

For later reference we record the following lemma.

\begin{lemma}\label{yal 3}
Assume that $N/K$ is an arbitrary Galois extension of $p$-adic number fields with Galois group $G$.
If $\chi \in \Irr(G)$, then
\[
\partial^1_{\ZpG, \QpG}\left(\left( \det(\Ind_{K/\Qp}\chi) \right)(a)\right)=0
\]
for all $a \in \Qptimes$. 
\end{lemma}
\begin{proof}
We write $\sigma \in G_\Qp$ for any lift of $\rec_\Qp(a) \in G_\Qp^{ab}$. Let $\mathrm{Ver} \colon G_\Qp^{ab} \lra G_K^{ab}$ denote the transfer homomorphism, see e.g. \cite[Ch.I.5]{NSW}. By \cite[\S 1]{Del73} we have the rule
\[
\det(\Ind_{K/\Qp}\chi) = \det(\chi) \circ \mathrm{Ver}.
\]
Hence $\det(\Ind_{K/\Qp}\chi)(a) = (\det(\chi))(\tau)$ with $\tau := \mathrm{Ver}(\sigma)$. The linear character $\det\chi$ factors through $\Gal(K^{ab} \cap N/K)$. If $\sigma' \in G$ denotes any lift of $\tau|_{K^{ab} \cap N}$, then $(\det(\chi))(\tau) = (\det(\chi))(\sigma')$. It follows that $\det(\Ind_{K/\Qp}\chi)(a)$ is the image of $\sigma' \in K_1(\ZpG)$ under the canonical homomorphism $K_1(\ZpG) \lra K_1(\QpG)$. The result follows now from the localisation sequence.
\end{proof}

\end{subsection}
\end{section}

\begin{section}{The computation of the cohomological term}\label{computationCNK}
This section contains the more technical part of the paper, concerning the computation of the cohomological term $C_{N/K}$ in the special case we are considering. This will allow to prove Theorem \ref{main theorem} in the last section.
\begin{subsection}{\texorpdfstring{Some general results about $C_{N/K}$}{Some general results about C N/K}}\label{CNK 1}
We start with some general results about the cohomological term, 
for which we do not need the restrictive assumptions of Theorem \ref{main theorem}. We will use these results to prove the rationality of $R_{N/K}$ and the equivalence of Conjectures \ref{Ca}, \ref{Cb} and  \ref{Cc} as stated in Proposition \ref{Conj equiv}. 
Furthermore we will show that Breuning's conjecture \cite[Conj.~3.2]{Breuning04} is a special case of Conjecture \ref{Cc}.

We recall that $T = \Zp(\chinr)(1)$ with an unramified character $\chinr$ which is the restriction to $G_K$ of an unramified character $\chinr_\Qp \colon G_\Qp \lra \Z_p^\times$. Then $T \cong T_p\calF$ where $\calF$ is the Lubin-Tate formal group associated with $\pi = up$ and $u = \chinr_\Qp(\varphi)$. 

Let $\calL \sseq \ON$ be a full projective $\ZpG$-sublattice such that the exponential map $\exp_\calF \colon \mathbb{G}_a \lra \calF$ converges on $\calL$. We first assume that $\chinr|_{G_N} \ne 1$. We set $X(\calL) := \exp_\calF(\calL)$ and obtain in this way a full $\ZpG$-projective sublattice of $\calF(\p_N)$, which by Lemma \ref{galois cohomology} (a) we identify with $H^1(N, T)$. If $\chinr|_{G_N} = 1$ we note that $\exp_{\mathbb{G}_m} = \exp - 1$ where $\exp$ denotes the $p$-adic exponential map. In this case we define $X(\calL) := \exp(\calL)$ and obtain a full $\ZpG$-projective sublattice of $U_N \sseq \widehat{N^\times}$. Recall that by  Lemma \ref{galois cohomology} (b) we identify $H^1(N, T)$ with $ \widehat{N^\times} (\chinr)$.

So in either case we have a natural embedding $X(\calL) \hookrightarrow H^1(N, T)$. As in the proof of Corollary \ref{Pbullet} we fix once and for all a representative $[A \lra B]$ of $R\Gamma(N,T)$, i.e., $A$ and $B$ are finitely generated cohomologically trivial $\ZpG$-modules centered in degrees $1$ and $2$. In addition, we will always assume that $B$ is $\ZpG$-projective. Thus we have a perfect $2$-extension
\[0 \lra H^1(N, T) \lra A \lra B \lra H^2(N, T) \lra 0.\]
The  embedding $X(\calL) \hookrightarrow H^1(N, T)$ induces an injective map of complexes
\[X(\calL)[-1] \lra M^\bullet,\]
where $M^\bullet = R\Gamma(N, T) \oplus \Ind_{N/\Qp}(T)[0]$ was defined in (\ref{Mbullet def}). 

We put
\begin{eqnarray*}
K^\bullet(\calL) &:=& \Ind_{N/\Qp}(T)[0] \oplus X(\calL)[-1], \\
M^\bullet(\calL) &:=& [A/X(\calL) \lra B] \quad \text{ with modules in degrees 1 and 2},
\end{eqnarray*}
and have thus constructed an exact sequence of complexes
\[0 \lra K^\bullet(\calL) \lra M^\bullet \lra M^\bullet(\calL) \lra 0.\]
By the additivity of Euler characteristics (see \cite[Th.~5.7]{BrBuAdditivity}) and the definition of $C_{N/K}$ in (\ref{CNK}) we obtain
\begin{equation}\label{CNK 11}
C_{N/K} = -\chi_{\ZpG, \BdR[G]}(K^\bullet(\calL), \lambda_1^{-1}) -\chi_{\ZpG, \BdR[G]}(M^\bullet(\calL), \lambda_2^{-1})
\end{equation}
with $\lambda_1 = \comp_V \circ \exp_V^{-1}$ and 
\[\lambda_2 \colon \left( H^1(N,T)/X(\calL) \right) \tensor_\ZpG \BdR[G] \lra H^2(N,T)  \tensor_\ZpG \BdR[G]\] 
given by
\[\lambda_2 = \begin{cases}\nu_N & \text{if } \chinr|_{G_N} = 1 \\
0 & \text{if } \chinr|_{G_N} \ne 1.
\end{cases}\]
From (\ref{old and new}) we obtain
\begin{equation}\label{CNK 12}
\begin{split}
& \chi_{\ZpG, \BdR[G]}(K^\bullet(\calL), \lambda_1^{-1})\\
&\quad= -\chiold_{\ZpG, \BdR[G]}(K^\bullet(\calL), \lambda_1)
- \partial^1_{\ZpG, \BdR[G]}((B^\od(K^\bullet(\calL)), -\id)).
\end{split}\end{equation}
By the definition of $\chiold_{\ZpG, \BdR[G]}$ we derive
\begin{equation}\label{CNK 13}
\chiold_{\ZpG, \BdR[G]}(K^\bullet(\calL), \lambda_1) = [X(\calL), \comp_V\circ\exp_V^{-1}, \Ind_{N/\Qp}(T)].
\end{equation}
Furthermore, by  \cite[Lemma~6.3]{BrBuAdditivity}, we have 
\begin{equation}\label{CNK 14}
\partial^1_{\ZpG, \BdR[G]}((B^\od(K^\bullet(\calL)), -\id)) = \partial^1_{\ZpG, \BdR[G]}((X(\calL) \tensor_\ZpG \BdR[G], -\id)).
\end{equation}
Since $X(\calL) \tensor_\ZpG  \QpG \cong \QpG^{[K:\Qp]}$ is a free $\QpG$-module, the right hand side is trivial,
so that we deduce from (\ref{CNK 11}), (\ref{CNK 12}), (\ref{CNK 13}) and  (\ref{CNK 14})
\begin{equation}\label{CNK 15}
C_{N/K} = [X(\calL), \comp_V \circ \exp_V^{-1}, \Ind_{N/\Qp}(T)] - \chi_{\ZpG, \BdR[G]}(M^\bullet(\calL), \lambda_2^{-1}).
\end{equation}
Similarly, it can be shown that
\begin{equation}\label{MLoldnew}
\chi_{\ZpG, \BdR[G]}(M^\bullet(\calL), \lambda_2^{-1}) = -\chiold_{\ZpG, \BdR[G]}(M^\bullet(\calL), \lambda_2).\end{equation}

\begin{remark}\label{choice of calL}
If $\p_N^n$ is $\ZpG$-projective for some $n \in \N$ and $n$ is large enough so that $\exp_\calF$ converges on $\p_N^n$, then we can use $\calL = \p_N^n$. In this case we obtain $X(\calL) = \calF(\p^n_N)$ resp. $X(\calL) = U_N^{(n)}$. For example, if $N/K$ is at most tamely ramified, every ideal $\p_N^n$ is $\ZpG$-projective. If $N/K$ is weakly ramified, then $\p^n_N$ is $\ZpG$-projective, if and only if $n \equiv 1 \pmod {|G_1|}$. Both properties are implied by \cite[Th.~(1.1)]{Koeck}.
\end{remark}

The computation of $\chi_{\ZpG, \BdR[G]}(M^\bullet(\calL), \lambda_2^{-1})$ in the special case of Theorem \ref{main theorem} will be postponed to the next subsections. For the purpose of this subsection we just remark that
\begin{equation}\label{rat 1}
\chi_{\ZpG, \BdR[G]}(M^\bullet(\calL), \lambda_2^{-1})  \in K_0(\ZpG, \QpG). 
\end{equation}

The term $[X(\calL), \comp_V \circ \exp_V^{-1}, \Ind_{N/\Qp}(T)]$ can be made more explicit if we let $\calL$ be of the following special form. Let $b \in N$ be a normal basis element of $N/K$, i.e. $N = K[G]b$, and suppose that $\calL = \OKG b$ is such that $\exp_\calF$  converges on $\calL$. Let
\[\theta = \left( \theta_\chi\right)_{\chi \in \Irr(G)} \in Z(\Qpc[G])^\times = \prod_{\chi \in \Irr(G)} (\Qpc)^\times\]
be defined by 
\[\theta_\chi = \frd_K^{\chi(1)} \calN_{K/\Qp}(b | \chi),\]
where $\frd_K$ denotes the discriminant of $K/\Qp$ and $\calN_{K/\Qp}(b | \chi)$ the usual norm resolvent, see e.g. \cite[Sec.~2.2]{PickettVinatier}.

\begin{lemma}\label{IV6.1}
Let $\calL = \OKG b$ be as above. Then
\[[X(\calL),\comp_V \circ \exp_V^{-1},\Ind_{N/\Qp}(T)]+\partial_{\ZpG,\BdR[G]}^1(t)=\hat\partial_{\ZpG,\BdR[G]}^1(\theta)\]
in $K_0(\ZpG,\BdR[G])$.
\end{lemma}
\begin{proof}
The proof is similar to the proof of \cite[Lemma~6.1]{IV}. 
\end{proof}

Combining (\ref{CNK 15}) and Lemma \ref{IV6.1} we obtain
\begin{equation}\label{CNK44}
C_{N/K} =  -\partial_{\ZpG, \BdR[G]}^1(t) + \partial_{\ZpG, \BdR[G]}^1(\theta) - \chi_{\ZpG,\BdR[G]}(M^\bullet(\calL), \lambda_2^{-1}).
\end{equation}

We recall the basic properties of the so-called \emph{unramified term} $U_{N/K}$ defined by Breuning in  \cite[Prop.~2.12]{Breuning04}. We write $\Opt$ for the ring of integers in the maximal tamely ramified extension of $\Qp$ in $\Qpc$. For subrings $\ZpG \sseq \Lambda_1 \sseq \Lambda_2 \sseq \QpcG$ we write $j_{\Lambda_1, \Lambda_2} \colon K_0(\Lambda_1, \QpcG) \lra K_0(\Lambda_2, \QpcG)$ for the natural scalar extension map. We write $\iota \colon K_0(\ZpG, \Qpc[G]) \lra K_0(\Opt[G], \Qpc[G])$ for $j_{\ZpG,\Opt[G]}$. We recall that by Taylor's fixed point theorem the restriction of $\iota$ to the subgroup $K_0(\ZpG, \Qp[G])$ is injective.

Let $\Qpab$ be the maximal abelian extension of $\Qp$ inside $\Qpc$. Then $\Qpab = \Qpnr \Qpram$, where $\Qpnr$ denotes the maximal unramified extension of $\Qp$ and $\Qpram := \Qp(\mu_{p^\infty})$. For each $\omega \in G_\Qp$ we define elements $\omeganr, \omegaram \in \Gal(\Qpab / \Qp)$ by
\begin{eqnarray*}
&&  \omeganr |_\Qpnr = \omega |_\Qpnr, \quad \omeganr |_\Qpram = 1, \\ 
&&  \omegaram |_\Qpnr = 1, \quad \omegaram|_\Qpram = \omega |_\Qpram.
\end{eqnarray*}
Then $U_{N/K}$ is the unique element in $K_0(\ZpG, \Qpc[G])$ satisfying the following two properties (see \cite[Prop.~2.12]{Breuning04})
\begin{itemize}
\item [(1)] $\iota(U_{N/K}) = 0.$
\item [(2)] If $(\alpha_\chi)_{\chi \in \Irr(G)} \in \prod_{\chi \in \Irr(G)}\left(\Qpc\right)^\times$ is any preimage of $U_{N/K}$ under $\hat\partial^1_{\ZpG, \Qpc[G]}$, then
\begin{equation}\label{invar 0}
\omega(\alpha_{\omega^{-1}\circ\chi}) = \alpha_\chi \det\nolimits_{\Ind_{K/\Qp}(\chi)}(\omeganr)
\end{equation}
for all $\omega \in G_\Qp$.
\end{itemize}
Recall the definition of (\ref{defRNK}):
\[R_{N/K}=C_{N/K} + \Ucris+\partial^1_{\ZpG, \BdR[G]}(t) - U_{N/K} + \partial^1_{\ZpG, \BdR[G]}(\epsilon_D(N/K,V)).\]
We are now able to prove the following result.
\begin{prop}\label{RNK rational}
The element $R_{N/K}$ is rational, i.e., $R_{N/K} \in K_0(\ZpG, \Qp[G])$.
\end{prop}
\begin{proof}
For elements $x,y \in K_0(\ZpG, \Qpc[G])$ we use the notation $x \equiv y$ when $x-y \in K_0(\ZpG, \QpG)$. Since $\Ucris$ is rational, we derive from Proposition \ref{eps coroll}, (\ref{rat 1}) and (\ref{CNK44})
\[R_{N/K} \equiv \partial_{\ZpG, \BdR[G]}^1(\theta) - U_{N/K} - T_{N/K},\]
where $T_{N/K} := \partial_{\ZpG, \BdR[G]}^1\left( \tau_\Qp(\Ind_{K/\Qp}(\chi))_{\chi \in \Irr(G)}\right)$ is precisely the element defined by Breuning in \cite[Sec.~2.3]{Breuning04}.

By \cite[Lemma~2.4]{Breuning04} we obtain for all preimages $(\beta_\chi)_{\chi \in \Irr(G)} \in \prod_{\chi \in \Irr(G)}\left(\Qpc\right)^\times$ of $T_{N/K}$  under $\hat\partial^1_{\ZpG, \QpcG}$ and all $\omega \in G_\Qp$
\begin{equation}\label{invar 1}
\omega(\alpha_{\omega^{-1}\circ\chi}) = \beta_\chi \det\nolimits_{\Ind_{K/\Qp}(\chi)}(\omegaram).
\end{equation}
From the proof of \cite[Lemma~2.8]{Breuning04} we derive
\begin{equation}\label{invar 2}
\omega(\theta_{\omega^{-1}\circ\chi}) = \theta_\chi \det\nolimits_{\Ind_{K/\Qp}(\chi)}(\omega^{}).
\end{equation}  
The result now follows by combining (\ref{invar 0}), (\ref{invar 1}) and (\ref{invar 2}).
\end{proof}

At last we provide the proof of Proposition \ref{Conj equiv}. We have to show the equivalence of the following assertions. 

\begin{itemize}
\item [(i)] $j_{\ZpG, \tilde\Lambda}([\DEPT, \delta(N/K, V)]) = 0$ in $\pi_0(V(\tilde\Lambda, \tilde\Omega)) \cong K_0(\tilde\Lambda, \tilde\Omega)$.
\item [(ii)] $\tilde{R}_{N/K} = 0$ in $K_0(\tilde\Lambda, \tilde\Omega)$.
\item [(iii)] $R_{N/K} = 0$ in $K_0(\ZpG, \tilde\Omega)$.
\end{itemize}

\begin{proof}
By Taylor's fixed point theorem the restriction of $j_{\ZpG, \tilde\Lambda}$ to $K_0(\ZpG, \QpG)$ is injective. In order to show the equivalence of (ii) and (iii)  it suffices by Proposition \ref{RNK rational} to show that $j_{\ZpG, \tilde\Lambda}(R_{N/K}) = \tilde{R}_{N/K}$. For this purpose we just remark that the proof of \cite[Prop.~2.12]{Breuning04} actually shows that $j_{\ZpG, \tilde\Lambda}(U_{N/K}) = 0$.

The equivalence of (i) and (ii) was already shown in Lemma \ref{conj equiv first}.
\end{proof}

To conclude this subsection we prove the following proposition.

\begin{prop}\label{Breu equiv}
If $T=\Zp(1)$, then Conjecture \ref{Cc} and \cite[Conj.~3.2]{Breuning04} are equivalent.
\end{prop}

\begin{proof}
\cite[Conj.~3.2]{Breuning04} states that
\[R_{N/K}^\Breu=T_{N/K}^\Breu+C_{N/K}^\Breu+U_{N/K}^\Breu-M_{N/K}^\Breu=0,\]
where we write the index Breu to refer to the terms defined by Breuning. The terms $U_{N/K}^\Breu$ and $U_{N/K}$ are defined in the same way. By Remark \ref{Breu MNK}, $M_{N/K}^\Breu=\Ucris$. By Proposition \ref{eps coroll} together with 
Lemma \ref{yal 3} we have
\[
T_{N/K}^\Breu=\partial^1_{\ZpG, \BdR[G]}\tau_\Qp( \Ind_{K/\Qp}(\chi) )=-\partial^1_{\ZpG, \BdR[G]}\epsilon_D(N/K, V)_\chi.
\]
Recalling the definition of $E(X)$ in \cite[(19)]{BlBu03}, the proof of \cite[Lemma 3.7]{BlBu03} and Equation (\ref{MLoldnew}), we see that 
\begin{equation}\label{Eexp}\begin{split}E(\exp(\calL))_p&=\chiold_{\ZpG, \BdR[G]}(M^\bullet(\calL)[1], \lambda_2^{-1})=-\chiold_{\ZpG, \BdR[G]}(M^\bullet(\calL), \lambda_2)\\&=\chi_{\ZpG, \BdR[G]}(M^\bullet(\calL), \lambda_2^{-1}).\end{split}\end{equation}

Using \cite[Prop.~2.6]{Breuning04}, \cite[Lemma~2.7]{Breuning04}, (\ref{Eexp}), Lemma \ref{IV6.1} and (\ref{CNK 15}) we obtain:
\[\begin{split}C_{N/K}^\Breu&=E(\exp(\calL))_p-[\calL,\rho_L,H_L]\\
&=\chi_{\ZpG, \BdR[G]}(M^\bullet(\calL), \lambda_2^{-1})-\hat\partial^1_{\tilde\Lambda,\BdR[G]}(\theta)\\
&=\chi_{\ZpG, \BdR[G]}(M^\bullet(\calL), \lambda_2^{-1})\!-\![X(\calL),\comp_V \circ \exp_V^{-1},\Ind_{N/\Qp}(T)]-\partial_{\tilde\Lambda, \BdR[G]}^1(t)\\
&=-C_{N/K} -\partial_{\tilde\Lambda, \BdR[G]}^1(t).\end{split}\]

Substituting all those terms in $R_{N/K}^\Breu$ we get $-R_{N/K}$.
\end{proof}
\end{subsection}

\begin{subsection}{\texorpdfstring{The computation of $C_{N/K}$}{The computation of C N/K}}\label{CNK 2}
Our aim now is to make Formula (\ref{CNK44}) more explicit by calculating the Euler characteristic $\chi_{\ZpG, \BdR[G]}(M^\bullet(\calL), \lambda_2^{-1})$ in the special situation of Theorem \ref{main theorem}. We will use several preliminary results and constructions of \cite{BC} and therefore stick closely to the notations used there. 

We will consider local field extensions as follows.
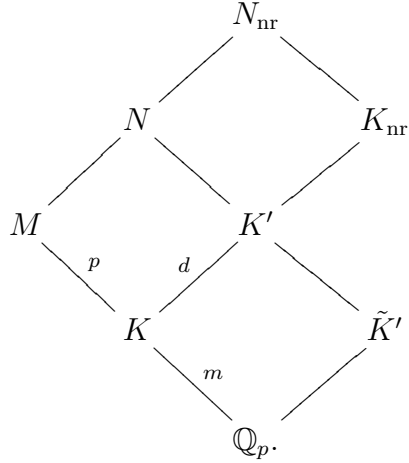
\[\xymatrix{
&&\Nur\ar@{-}[dr]\ar@{-}[dl]\\
&N\ar@{-}[dr]\ar@{-}[dl]&&\Kur\ar@{-}[dl]\\
M\ar@{-}[dr]^p&&K'\ar@{-}[dl]_d\ar@{-}[dr]\\
&K\ar@{-}[dr]^m&&\tilde K'\ar@{-}[dl]\\&&\Qp.}\]
Here $K/\Qp$ is the unramified extension of degree $m$, hence $d_K = m = [K:\Qp]$. Furthermore, $K'/K$ is the maximal unramified subextension of $N/K$  and we set $d := d_{N/K} = [K':K]$. We assume throughout that $(m,d)=1$. Finally, $M/K$ is a weakly and wildly ramified cyclic extension of degree $p$ and $N = MK'$. Since $(m,d)=1$, there exists $\tilde K'/\Qp$ of degree $d$ such that $K'=K\tilde K'$.

Let $F := F_K \in\Gal(\Nur/M)\cong \Gal(\Kur/K)$ be the Frobenius automorphism; we will keep the notation $F_N$ for the element $F^d\in\Gal(\Nur/N)\cong \Gal(\Kur/K')$, which in \cite{BC} used to be called $F_0$. We put $q=p^m$, $b=F^{-1}$ and consider an element $a\in \Gal(\Nur/K)$ such that $\Gal(M/K)=\langle a|_M\rangle$, $a|_{\Kur}=1$. Since there will be no ambiguity, we will denote by the same letters $a,b$ their restrictions to $N$. Then $\Gal(N/K) = \langle a,b \rangle$ and $\mathrm{ord}(a)=p, \mathrm{ord}(b) = d$. We also define $\T_a := \sum_{i=0}^{p-1}a^i$.

Any irreducible character $\psi$ of $G$ decomposes as $\psi = \chi\phi$ where $\chi$ is an irreducible character of $\langle a \rangle$ and $\phi$ an irreducible character of $\langle b \rangle$. 

We will always identify $K_0(\ZpG, \Qpc[G])$ with $\Qpc[G]^\times / \ZpG^\times$ where the isomorphism is induced by $\hat\partial_{\ZpG, \Qpc[G]}^1$. We say that $\alpha \in \Qpc[G]^\times$ represents $c \in K_0(\ZpG, \Qpc[G])$, if $\hat\partial_{\ZpG, \Qpc[G]}^1(\alpha) = c$.

If $\chinr|_{G_N} = 1$ the Euler characteristic $\chi_{\ZpG, \BdR[G]}(M^\bullet(\calL), \lambda_2^{-1})$ is just a twist of the element $E(\exp(\calL))_p$ which we computed in \cite[Prop.~4.3.1]{BC}.

From now on we assume that  $\chinr|_{G_N} \ne 1$. We will have to study separately the cases $\omega=0$ and $\omega>0$, where $\omega = \omega_N = v_p(1 - \chinr(F_N))$ is as in (\ref{omega def}). We start with some preliminary results and constructions which will be used in both cases.

\begin{lemma}\label{projectivity}
For $n \ge 2$ one has
\[\calF(\p_N^n) \text{ is }\ZpG \text{-projective}\iff n\equiv 1\pmod{p}.\]
Moreover,
\[\calF(\p_N) \text{ is }\ZpG \text{-projective} \iff \omega = 0.\]
\end{lemma}
\begin{proof}
For $n \ge 2$ the formal logarithm induces an isomorphism $\calF(\p_N^n) \cong \p_N^n$ of $\ZpG$-modules by \cite[Th.~IV.6.4]{Sil}. Hence the first assertion follows from \cite[Th.~1.1 and Prop.~1.3]{Koeck}.

We henceforth assume $n=1$. By Lemma \ref{UL1hat}, Lemma \ref{LurInd} and Theorem \ref{FpnZpomega} we already know that $\calF(\p_N)$ is cohomologically trivial, if and only if $\omega = 0$. Hence it suffices to prove that if $\omega = 0$ then $\calF(\p_N)$ is torsionfree. By \cite[Th.~IV.6.4]{Sil} $\log_\calF$ converges on $\calF(\p_N)$, so that we may set $J := \log_\calF(\calF(\p_N)) \sseq N$. We consider the commutative diagram
\[\xymatrix{
0\ar[r]& \calF(\p_N^2) \ar[r]\ar[d]^{\log_\calF}& \calF(\p_N) \ar[r] \ar[d]^{\log_\calF} 
& \calF(\p_N) / \calF(\p_N^2) \ar[r] \ar[d]^{\lambda_\calF}   & 0\\
0\ar[r]& \p_N^2 \ar[r]& J \ar[r] & J / \p_N^2 \ar[r] & 0,}\]
where $\lambda_\calF(x + \calF(\p_N^2)) = \log_\calF(x) + \p_N^2$. Recall that in Subsection \ref{sit} we introduced a Lubin-Tate formal group $\calF$ associated with $\pi=up$. By \cite[Sec.~V.4, Aufgabe~4]{Neukirch92} we may choose $\calF$ such that
\[\log_\calF(x) = \sum_{j=0}^\infty \frac{x^{p^j}}{\pi^j}.\]
We recall that $v_N(\pi) = p$. If $v_N(x) = 1$, then $\frac{x^p}{\pi}$ is a unit and so is also $x+\frac{x^p}{\pi}$. Since for $j\geq 2$
\[v_N\left(\frac{x^{p^j}}{\pi^j}\right) = p^j - jp \ge 2,\]
we deduce that $v_N(\log_\calF(x)) = 0$ for $x \in \calF(\p) \setminus \calF(\p^2)$, so that $\lambda_\calF$ is injective. By the snake lemma the middle vertical map is therefore also injective and we have an isomorphism $\log_\calF \colon \calF(\p_N) \lra J$ of $\Zp$-modules. Obviously, $J \sseq N$ is torsionfree.
\end{proof}

By the above lemma we can use $\calL = \p_N^{p+1}$ in our constructions (see also Remark \ref{choice of calL}). In this case we have $X(\calL) = \calF(\p_N^{p+1})$.

We recall now some results and constructions of \cite{BC}. Let $\theta_1\in M$ be such that $\calO_K[\Gal(M/K)]\theta_1=\p_M$ and $\T_{M/K}\theta_1=p$ (see \cite{BC}, Lemma 3.1.2]). Let $\theta_2$ (resp. $A$) be a normal integral basis generator of trace one for the extension $\tilde K'/\Qp$ (resp. $K/\Qp$). Such elements exist by \cite[Lemma~3.1.1]{BC}.

Since $a\in G_1\setminus G_2$, where $G_i$ is the $i$-th ramification group of $G=\Gal(N/K)$, we know by \cite[Sec.~IV.2, Prop.~5]{SerreLocalFields} that $\theta_1^{a-1}\equiv 1-\alpha_1\theta_1\pmod{\p_M^2}$ for some unit $\alpha_1\in\calO_M^\times$. Since $\alpha_1$ can be replaced by any element in the same residue class in $\calO_M/\p_M=\calO_K/\p_K$, we can assume that $\alpha_1\in\calO_K^\times$.

By our choice of $A$, we know that $A,A^{\varphi},\dots A^{\varphi^{m-1}}$ is a basis of $\calO_K$ over $\Z_p$, where $\varphi$ denotes as before the Frobenius automorphism of $\Kur/\Qp$. Since $1=\T_{K/\Qp}A=\sum_{i=0}^{m-1}A^{\varphi^i}$ and $\alpha_1\in\calO_K^\times$ it easily follows that also
\begin{equation}\label{choice of alpha}
\alpha_1,\alpha_2=\alpha_1A,\alpha_3=\alpha_1A^{\varphi},\dots, \alpha_m=\alpha_1A^{\varphi^{m-2}}
\end{equation}
constitute a basis of $\calO_K$ over $\Z_p$. In particular, we have the equality $A=\frac{\alpha_2}{\alpha_1}$.

We recall from \cite[Lemma~3.1.3]{BC} that the polynomial $X^p - X + A\theta_2$ divides $X^{q^d} - X + 1$ in $\calO_{K'}/ \p_{K'}[X]$. As in \cite{BC} we choose $x_2 \in \calO_{\Kur}$ such that $\frac{x_2}{\alpha_1} \mod{\p_{K'}}$ is a root of  $X^p - X + A\theta_2$. The following lemma is analogous to \cite[Lemma~3.1.4]{BC}.

\begin{lemma}\label{BC 3.1.4}
Assume  $\omega>0$. If we we consider $1 + x_2\theta_1$ as an element of $\widehat{N_0^\times}$, then
\[(1 + x_2\theta_1)^{\chinr(F_N)F_N - 1} \equiv \theta_1^{a-1} \equiv 1 - \alpha_1\theta_1 \pmod{\p_{N_0}^2}.\]
\end{lemma}
\begin{proof}
From $\left(\frac{x_2}{\alpha_1}\right)^{q^d}-\frac{x_2}{\alpha_1}\equiv -1\pmod{\p_{K'}}$ we obtain $x_2^{q^d} - x_2 \equiv -\alpha_1 \pmod{\p_{N_0}}$. 
Since $\omega > 0$ we have $u^{dm} \equiv 1 \pmod{p}$ and hence obtain
\begin{equation}\label{aux eq 217}
u^{dm}x_2^{q^d} - x_2 \equiv -\alpha_1 \pmod{\p_{N_0}}.
\end{equation}
Recalling $\chinr(F_N) = u^{dm} \equiv 1 \pmod{p}$ and  $v_N(\theta_1) = 1$ we conclude
\[\begin{split}
(1 + x_2\theta_1)^{\chinr(F_N)F_N - 1} &= (1 + x_2^{F_N}\theta_1)^{\chinr(F_N)}(1 + x_2\theta_1)^{-1} \\
&\equiv (1 + u^{dm}x_2^{q^d}\theta_1)(1 - x_2\theta_1) \pmod{\p_{N_0}^2}\\
&\equiv 1 + u^{dm}x_2^{q^d}\theta_1 - x_2\theta_1  \pmod{\p_{N_0}^2}\\
&\equiv 1 - \alpha_1\theta_1 \pmod{\p_{N_0}^2},
\end{split}\]
where the last congruence follows from (\ref{aux eq 217}).
\end{proof}

As in \cite[Lemma~3.1.5]{BC}, we can use Lemma \ref{Neukirch} to find an element $\gamma\in U_{N_0}^{(1)}$ such that $\gamma^{\chinr(F_N)F_N-1}=\theta_1^{a-1}$. If $\omega>0$, using also Lemma \ref{BC 3.1.4}, we can assume that $\gamma\equiv 1 + x_2\theta_1\pmod{\p_{N_0}^2}$.

Let
\[W'=\ZpG z_1\oplus \ZpG z_2,\]
\[W_{\geq n}=\bigoplus_{j=n}^{p-1}\bigoplus_{k=1}^{m}\ZpG v_{k,j}\]
and put 
\[W=W_{\geq 0}.\]
Note that these are the same as the $p$-completions of $\calF'$, $\calF_{\geq n}$, $\calF=\calF_{\geq 0}$ of \cite{BC}, but here we prefer to change the notation to avoid confusion with the notation for the Lubin-Tate formal group.

Note that the assignment $v_{k,j} \mapsto \alpha_kw_j$ induces an isomorphism
\begin{equation}\label{ff iso}
W  \stackrel{\cong}\lra \bigoplus_{j=0}^{p-1}\bigoplus_{k=1}^m\ZpG\alpha_kw_j = \bigoplus_{j=0}^{p-1}\OKG w_j
\end{equation}
of free $\ZpG$-modules. We will always identify $W$ with $\bigoplus_{j=0}^{p-1}\bigoplus_{k=1}^m\ZpG\alpha_kw_j$.

We set 
\begin{equation}\label{E def}
E=
\begin{cases}
  1 &\text{ if } \omega = 0 \\
  \sum_{i=0}^{dm-1}u^{-i}(A\theta_2)^{\varphi^i} &\text{ if } \omega > 0
\end{cases}
\end{equation}
and
\begin{equation}\label{tilde u def}
\tilde u=
\begin{cases}
  1 &\text{ if } \omega = 0 \\
  u &\text{ if } \omega > 0.
\end{cases}
\end{equation}

Recall that in Lemma \ref{UL1hat} we defined $\epsilon\in \overline{\Zpnr}^\times$ such that $u=\epsilon^{\varphi-1}$.

\begin{lemma}\label{E lemma}
We have $E\in\calO_{K'}^\times$ and $E^\varphi\equiv \tilde uE\pmod{p\calO_{K'}}$. In particular, if $\omega>0$ we can also assume $\epsilon\equiv E\pmod{p\calO_{K'}}$.
\end{lemma}

\begin{proof}
For $\omega=0$ there is nothing to prove. So let us assume $\omega>0$.

Since $A\theta_2$ is an integral normal basis generator for the unramified extension $K'/\Qp$, the elements $(A\theta_2)^{\varphi^i}$ are linearly independent modulo $p\calO_{K'}$. Hence $E \not\equiv 0 \pmod{p\calO_{K'}}$ and we deduce $E\in\calO_{K'}^\times$.

Moreover, we have
\[E^\varphi = u \sum_{i=0}^{dm-1} u^{-(i+1)} (A\theta_2)^{\varphi^{i+1}} \equiv uE \pmod{p\calO_{K'}}\]
because $u^{dm} \equiv 1 \pmod{p\calO_{K'}}$ as a consequence of $0 < \omega = v_p(1 - \chinr(F_N))$ and $\chinr(F_N) = u^{dm}$.
\end{proof}

We define
\[\tfW\colon W \lra \calF(\p_N)\]
by
\[\tfW(v_{k,j})=E\alpha_k(a-1)^j\theta\]
for all $k$ and $j$, where $\theta=\theta_1\theta_2$. We denote by $\fW$ the composition of $\tfW$ with the projection to $\calF(\p_N)/\calF(\p_N^{p+1})$.

The following lemma is the analogue of \cite[Lemma~4.1.3]{BC}.

\begin{lemma}\label{f4surj}
The map $\fW$ is surjective. More precisely, for $j\geq 0$, 
\[\fW(W_{\geq j})=\calF(\p_N)^{j+1}/\calF(\p_N^{p+1}).\]
\end{lemma}

\begin{proof}
For $j=p$, $W_{\geq p}=\{0\}$ and $\calF(\p_N^{p+1})/\calF(\p_N^{p+1})=\{0\}$, so the  result is trivial.

We assume the result for $j+1$ and proceed by descending induction. Recall that by \cite[Lemma~3.2.6]{BC}, $\p_N^{j+1}=(p,(a-1)^j)\theta$. Hence if $x\in \calF(\p_N)^{j+1}$, then
\[E^{-1}x\equiv \mu p\theta+\nu(a-1)^j\theta\pmod{\p_N^{j+2}}\]
for some $\mu,\nu\in\calO_K[G]$. We write $\nu=\sum_{k,\ell}\nu_{k,\ell}a^kb^\ell$. Let $\tilde\nu=\sum_{k,\ell}\tilde u^{m\ell}\nu_{k,\ell}a^kb^\ell$. By Lemma \ref{E lemma} we have
\[\begin{split}\tilde u^{ml}a^kb^l(E(a-1)^j\theta) &= \tilde u^{ml} E^{\varphi^{-ml}} \cdot a^kb^l(a-1)^j\theta \\
& \equiv E \cdot a^kb^l(a-1)^j\theta \pmod{\p_N^{p+1}}.
\end{split}\]
Since  $\mu p\theta\in\p_N^{p+1}$ we conclude further
\[x \equiv E\nu(a-1)^j\theta \equiv \tilde \nu E(a-1)^j\theta \pmod{\p_N^{p+1}}.\]
By the analogue of \cite[Lemma~4.1.7]{BC} we have $\tilde \nu E(a-1)^j\theta\equiv \tfW(\tilde\nu w_j)\pmod{\p_N^{j+2}}$ and thus $x \equiv \tfW(\tilde\nu w_j)\pmod{\p_N^{j+2}}$. This means that $\pi(x)$ is the sum of an element in the image of $W_{\geq j}$ and an element in $\calF(\p_N)^{j+2}/\calF(\p_N^{p+1})$, which is by assumption in the image of $W_{\geq j+1}\subseteq W_{\geq j}$.
\end{proof}

As in \cite{BC} we will need to construct some particular elements in the kernel of $\fW$. We will proceed as in \cite[Lemma~4.2.4, 4.2.5, 4.2.6]{BC}.

\begin{lemma}\label{defsjk}
Let $0\leq j\leq p-1$, $1\leq k\leq m$. Then there exists $\mu_{j,k}\in (W_{\geq j+2})_p$ such that the element
\[s_{j,k}=\alpha_k(a-1) w_j-\alpha_kw_{j+1}+\mu_{j,k}\]
is in the kernel of $\fW$. Here $w_p$ should be interpreted as $0$.
\end{lemma}

\begin{proof}
The formal subtraction of $X$ and $Y$ takes the form
\[X-_\calF Y=X-Y+AXY+BX^2+CY^2+\deg \ge 3\]
with $A, B, C \in \Zp$. Taking $Y=0$ we see that $B=0$, taking $X=Y$ we then get $C=-A$. Hence
\[X-_\calF Y=X-Y+AXY-AY^2+\deg \ge 3.\]
For
\[x := E\alpha_k (a-1)^ja\theta, \quad y := E\alpha_k(a-1)^j\theta, \quad z := x-y=E\alpha_k(a-1)^{j+1}\theta\]
we obtain from \cite[Lemma~3.2.5]{BC} and Lemma \ref{E lemma}
\[v_N(x) = v_N(y) = j+1, \quad v_N(z)=v_N(x-y) \ge j+2.\]
In the following computation all congruences are modulo $\p_N^{j+3}$.
\[\begin{split}
x -_\calF y -_\calF z &\equiv x - y + Axy - Ay^2 - z + A(x-y+Axy - Ay^2)z - Az^2 \\
&\equiv z + Ayz-z+A(z+Ayz)z - Az^2\equiv 0.
\end{split}\]
Therefore
\[\begin{split}
&\tfW(\alpha_k(a-1) w_j-\alpha_kw_{j+1})\\
&\quad=E\alpha_k(a-1)^ja\theta-_\calF E\alpha_k(a-1)^j\theta-_\calF E\alpha_k(a-1)^{j+1}\theta
\equiv 0\pmod{\calF(\p_{N}^{j+3})}.
\end{split}\]
Using Lemma \ref{f4surj}, by the same arguments as in the proof of \cite[Lemma~4.2.4]{BC} the result follows.
\end{proof}

Recall that $(m,d) = 1$. Let $\tilde m$ denote an integer such that $m\tilde m \equiv 1 \pmod d$. The next lemma is the analogue of \cite[Lemma~4.2.5 and 4.2.6]{BC}.
\begin{lemma}\label{defrk}
The elements
\[r_1=\alpha_1 \T_a w_0+(u^{-1}\tilde u^{1-m\tilde m} b^{-\tilde m}\alpha_1-\alpha_1)w_{p-1},\]
\[r_k=\alpha_k \T_a w_0+(u^{-1}\tilde u^{1-m\tilde m} b^{-\tilde m}\alpha_{k+1}-\alpha_k)w_{p-1},\]
for $1< k<m$, and
\[r_m=\alpha_m \T_a w_0+\left(u^{-1}\tilde u^{1-m\tilde m}b^{-\tilde m}\alpha_1-u^{-1}\tilde u^{1-m\tilde m}b^{-\tilde m}\sum_{i=2}^m\alpha_i-\alpha_m\right)w_{p-1}\]
are in the kernel of $\fW$.
\end{lemma}

\begin{proof}
Recall the isomorphism
\[\fFN \colon \calF(\p_N)\lra \left(\widehat{N_0^\times}(\chinr) \right)^{\Gal(N_0/N)}\]
from Lemma \ref{UL1hat}. By \cite[Prop.~V.6.8]{SerreLocalFields}, we know that $\norm_{N_0/K_0}U_{N_0}^{(2)}\subseteq U_{K_0}^{(2)}\subseteq U_{N_0}^{(2p)}\subseteq U_{N_0}^{(p+1)}$. This implies that
\[\norm_{N_0/K_0}(\fFN(E\alpha_k\theta))\equiv \norm_{N_0/K_0}(1+\epsilon^{-1}E\alpha_k\theta)\pmod{\p_{N_0}^{p+1}}.\]
Now we calculate $\tfW(\alpha_k\T_aw_0)$, using the results and calculations of \cite{BC} and the fact that $\theta_2^p\equiv\theta_2^{p^{m\tilde m}}\equiv\theta_2^{q^{\tilde m}}\equiv\theta_2^{b^{-\tilde m}}\pmod{\p_{\tilde K'}}$.
\[\begin{split}\tfW(\alpha_k\T_aw_0)&=\T_a\cdot E\alpha_k\theta=\fFN^{-1}(\T_a\cdot(\fFN(E\alpha_k\theta)))\\
&=\fFN^{-1}(\norm_{N_0/K_0}(\fFN(E\alpha_k\theta)))\\
&\equiv \fFN^{-1}(\norm_{N_0/K_0}(1+\epsilon^{-1}E\alpha_k\theta))\pmod{\p_{N_0}^{p+1}}\\
&\equiv \fFN^{-1}\left( 1+\left(\frac{\epsilon^{-1}E\alpha_k\theta_2}{\alpha_1}-\left(\frac{\epsilon^{-1}E\alpha_k\theta_2}{\alpha_1}\right)^p\right)\alpha_1p\right)\pmod{\p_{N_0}^{p+1}}\\
&\equiv \left(\frac{\alpha_k\theta_2}{\alpha_1}-\epsilon^{1-p}E^{p-1}\left(\frac{\alpha_k\theta_2}{\alpha_1}\right)^p\right)E\alpha_1p\pmod{\p_{N_0}^{p+1}}\\
&\equiv \left(\alpha_k\theta_2-u^{-1}\tilde u\alpha_1\left(\frac{\alpha_k}{\alpha_1}\right)^p\theta_2^{b^{-\tilde m}}\right)Ep\pmod{\p_{N_0}^{p+1}}.\end{split}\]
By (\ref{choice of alpha}) we have for $1<k<m$,
\[\left(\frac{\alpha_k}{\alpha_1}\right)^p=\left(A^{\varphi^{k-2}}\right)^p \equiv A^{p^{k-1}} \equiv \frac{\alpha_{k+1}}{\alpha_1} \pmod{\p_K}\]
and
\[\left(\frac{\alpha_m}{\alpha_1}\right)^p=\left(A^{\varphi^{m-2}}\right)^p \equiv A^{\varphi^{m-1}} = 1-\sum_{i=0}^{m-2}A^{\varphi^i} = 1-\sum_{i=0}^{m-2}\frac{\alpha_{i+2}}{\alpha_1}
\pmod{\p_K}.\]
Hence we obtain the following congruences modulo $\p_N^{p+1}$:
\[\tfW(\alpha_k\T_aw_0)\equiv\begin{cases}\left(\alpha_1\theta_2-u^{-1}\tilde u\alpha_1\theta_2^{b^{-\tilde m}}\right)Ep&\text{if $k=1$}\\
\left(\alpha_k\theta_2-u^{-1}\tilde u\alpha_{k+1}\theta_2^{b^{-\tilde m}}\right)Ep&\text{if $1<k<m$}\\
\left(\alpha_m\theta_2+u^{-1}\tilde u\left(-\alpha_1+\sum_{i=2}^{m}\alpha_{i}\right)\theta_2^{b^{-\tilde m}}\right)Ep&\text{if $k=m$.}
\end{cases}\]
Recalling Lemma \ref{E lemma} and \cite[Lemma~3.2.4]{BC}, for $1\leq i\leq m$ we calculate
\[\begin{split}&\tfW(u^{-1}\tilde u^{1-m\tilde m}b^{-\tilde m}\alpha_iw_{p-1})=u^{-1}\tilde u^{1-m\tilde m}b^{-\tilde m}\cdot (E\alpha_i(a-1)^{p-1}\theta)\\
&\qquad=u^{-1}\tilde u^{1-m\tilde m}(E^{b^{-\tilde m}}\alpha_i^{b^{-\tilde m}}(a-1)^{p-1}\theta^{b^{-\tilde m}})\\
&\qquad=u^{-1}\tilde u^{1-m\tilde m}E^{\varphi^{m\tilde m}}\alpha_i\theta_2^{b^{-\tilde m}}(a-1)^{p-1}\theta_1\\
&\qquad\equiv u^{-1}\tilde u E\alpha_i\theta_2^{b^{-\tilde m}}p\pmod{\p_N^{p+1}}
\end{split}\]
and
\[\tfW(\alpha_iw_{p-1})=E\alpha_i(a-1)^{p-1}\theta\equiv E\alpha_i\theta_2p\pmod{\p_N^{p+1}}.\]
For $x,y\in\p_N^p$ we see that $x -_\calF y \equiv x-y+Axy - Ay^2 \equiv x-y \pmod{\p_N^{p+1}}$ so that we easily deduce the lemma by the above calculations.
\end{proof}
From now on we will need to distinguish between the cases $\omega=0$ and $\omega>0$.
\end{subsection}

\begin{subsection}{\texorpdfstring{The Euler characteristic when $\omega = 0$}{The Euler characteristic when omega = 0}}

As already mentioned we use $\calL = \p_N^{p+1}$ and thus have $X(\calL) = \calF(\p_N^{p+1})$.
Since $\omega = 0$ the second cohomology group of $M^\bullet(\calL)$ vanishes and we have the short exact sequence
\[0 \lra \calF(\p_N) / \calF(\p_N^{p+1}) \lra A / X(\calL) \lra B \lra 0.\]
Recall that for $\omega=0$ both $\calF(\p_N^{})$ and $\calF(\p_N^{p+1})$ are $\ZpG$-projective by Lemma \ref{projectivity}. Hence, recalling also (\ref{MLoldnew}), we derive

\[\chi_{\ZpG, \BdR[G]}(M^\bullet(\calL), \lambda_2^{-1}) = [\calF(\p_N^{p+1}), \id, \calF(\p_N^{})].\]

By Lemma \ref{f4surj}, the map $\fW \colon W \lra \calF(\p_N) / \calF(\p_N^{p+1})$ is onto. Thus we obtain the short exact sequence
\[0 \lra \ker(\fW) \lra W \xrightarrow{\fW}  \calF(\p_N) / \calF(\p_N^{p+1}) \lra 0\]
which implies 
\[[\calF(\p_N^{p+1}), \id, \calF(\p_N^{})] = [\ker(\fW), \id, W].\]
Recall the definition of the elements $s_{j,k}$ and $r_k$ in Lemmas \ref{defsjk} and \ref{defrk}.
\begin{lemma}\label{genkerh}
The $pm$ elements $r_{k},s_{j,k}$ for $0\leq j\leq p-2$, $1\leq k\leq m$ constitute a $\ZpG$-basis of $\ker \fW$.
\end{lemma}
\begin{proof}
At the end we want to adapt the proofs of \cite[Lemmas~4.2.7 and 4.2.9]{BC}, but first we need some preparations. We write the coefficients of the  $\alpha_iw_{p-1}$-components, $i=1,\dots,m$, of the elements $r_j$, $j=1,\dots,m$, into the columns of an $m \times m$  matrix which we call $\calM$,
\[\calM=\begin{pmatrix}
u^{-1}b^{-\tilde m}-1&0&0&\cdots&0&0&u^{-1}b^{-\tilde m}\\
0&-1&0&\cdots&0&0&-u^{-1}b^{-\tilde m}\\
0&u^{-1}b^{-\tilde m}&-1&\cdots&0&0&-u^{-1}b^{-\tilde m}\\
0&0&u^{-1}b^{-\tilde m}&\cdots&0&0&-u^{-1}b^{-\tilde m}\\
\vdots&\vdots&\vdots&\ddots&\vdots&\vdots&\vdots\\
0&0&0&\cdots&u^{-1}b^{-\tilde m}&-1&-u^{-1}b^{-\tilde m}\\
0&0&0&\cdots&0&u^{-1}b^{-\tilde m}&-1-u^{-1}b^{-\tilde m}
\end{pmatrix}.\]
By an easy induction argument we see that each $n \times n$-matrix 
\[M_{C,n}=\begin{pmatrix}
-1&0&\cdots&0&0&-C\\
C&-1&\cdots&0&0&-C\\
0&C&\cdots&0&0&-C\\
\vdots&\vdots&\ddots&\vdots&\vdots&\vdots\\
0&0&\cdots&C&-1&-C\\
0&0&\cdots&0&C&-1-C
\end{pmatrix}\]
has determinant $(-1)^n\sum_{i=0}^nC^i$. Setting $C := u^{-1}b^{-\tilde m}$ we obtain
\[\begin{split}
\det\calM&=(u^{-1}b^{-\tilde m}-1)(-1)^{m-1}\sum_{i=0}^{m-1}(u^{-1}b^{-\tilde m})^i\\
&=(-1)^{m-1}(u^{-m}b^{-\tilde m m}-1)=(-1)^{m-1}(u^{-m}b^{-1}-1).
\end{split}\]
We claim that $\det\calM$ is a unit in $\ZpG$. To that end it suffices to show that $\det\calM$ is a unit in the maximal $\Zp$-order $\Lambda'$ in $\QpG$ because $(\Lambda')^\times \cap \ZpG = \ZpG^\times$. For any irreducible character $\psi$ of $G$ we have $\psi(\det(\calM)) = (-1)^{m-1}(u^{-m}\zeta_d-1)$ for some $d$-th root of unity $\zeta_d$ (depending on $\psi$). If now $\psi( \det\calM )$ were not a unit, we would conclude that $u^{dm} = \chinr(F_N) \equiv 1 \pmod{p}$ contradicting the assumption $\omega = 0$.

We are now prepared to apply the arguments used to prove  \cite[Lemmas~4.2.7 and 4.2.9]{BC}. Following the proof of Lemma 4.2.7 of loc.cit., we can demonstrate that the $pm+m$ elements $r_{k},s_{j,k}$ for $0\leq j\leq p-1$, $1\leq k\leq m$ generate $\ker \fW$ as a $\ZpG$-module. Since the determinant of $\calM$ is a unit, it can be shown as in the proof of Lemma 4.2.9 of loc.cit. that for $k=1,\dots,m$ we have $s_{p-1,k} \in \langle r_1, r_2, \ldots, r_m \rangle_{\ZpG}$. Therefore the elements listed in the statement of our lemma generate $\ker \fW$. We actually have a basis since $\ker(\fW)$ is free of $\ZpG$-rank $\rk_\ZpG(W) = pm$.
\end{proof}

\begin{prop}\label{prop 117}
Assume that $\omega = 0$. For $\calL = \p_N^{p+1}$ the element 
\[\chi_{\ZpG, \BdR[G]}(M^\bullet(\calL), \lambda_2^{-1}) \in K_0(\ZpG, \BdR[G])\]
is contained in $K_0(\ZpG, \QpG)$ and represented by $\epsilon \in \QpG^\times$ where
\[\begin{split}
\epsilon_{\chi\phi}^{}&=\begin{cases}p^m&\text{if $\chi=\chi_0$}\\
(-1)^{m-1} (u^{-m}\phi(b)^{-1} - 1) (\chi(a)-1)^{m(p-1)}&\text{if $\chi\neq\chi_0$.}\end{cases}\\
\end{split}\]
\end{prop}
\begin{proof}
We choose the elements $r_k$ and $s_{j,k}$ of Lemma \ref{genkerh} as a $\ZpG$-basis of $\ker(\fW)$ and fix the canonical 
$\ZpG$-basis of $W$. Then $ \chi_{\ZpG, \BdR[G]}(M^\bullet(\calL), \lambda_2^{-1}) = [\ker(\fW), \id, W]$ 
is represented by the determinant of
\[\mm=\begin{pmatrix}
\T_aI&(a-1)I&0&\cdots&0&0\\
0&-I&(a-1)I&\cdots&0&0\\
0&*&-I&\cdots&0&0\\
\vdots&\vdots&\vdots&\ddots&\vdots&\vdots\\
0&*&*&\cdots&-I&(a-1)I\\
\calM&*&*&\cdots&*&-I
\end{pmatrix}.\]
Recalling that $p$ is odd, we get:
\[\begin{split}\det(\chi\phi(\mm))&=\begin{cases}p^m(-1)^{m(p-1)}&\text{if $\chi=\chi_0$}\\(-1)^{m^2(p-1)}\det(\chi\phi(\calM))(\chi(a)-1)^{m(p-1)}&\text{if $\chi\neq\chi_0$}\end{cases}\\
&=\begin{cases}p^m&\text{if $\chi=\chi_0$}\\ (-1)^{m-1}(u^{-m}\phi(b)^{-1}-1)(\chi(a)-1)^{m(p-1)}&\text{if $\chi\neq\chi_0$.}\end{cases}\end{split}\]
\end{proof}
\end{subsection}

\begin{subsection}{\texorpdfstring{The Euler characteristic when $\omega>0$}{The Euler characteristic when omega > 0}}
Recall the definition of $E$ and $\tilde u$ in (\ref{E def}) and (\ref{tilde u def}) respectively. 

\begin{lemma}\label{deff4}
There is a commutative diagram of $\ZpG$-modules with exact rows
\[\xymatrix{0\ar[r]& X(2)\oplus W\ar[r]\ar[d]^{\tilde f_4}&W'\oplus W\ar[rr]^-{\delta_2}\ar[d]^{\tilde f_3}&&\ZpG z_0
\ar[r]^-{\pi}\ar[d]^{\tilde f_2}&\Z_p/p^\omega\Z_p(\chinr)\ar[r]\ar[d]^{=}&0\\
0\ar[r]&\calF(\p_N)\ar[r]^{\fFN}&\Nnr\ar[rr]^-{(F-1)\times 1}&&\Nnr\ar[r]&\Z_p/p^\omega\Z_p(\chinr)\ar[r]&0}\]
where
\[\begin{split}
&\delta_2(z_1)=(u^mb-1)z_0,\\
&\delta_2(z_2)=(a-1)z_0,\\
&\delta_2(v_{j,k})=0\text{ for all $k$ and $j$},\\
&\tilde f_2(z_0)=[\theta_1,1,\dots,1],\\
&\tilde f_3(z_1)=[\theta_1,1,\dots,1],\\
&\tilde f_3(z_2)=[\gamma,\dots,\gamma],\\
&\tilde f_3(v_{k,j})=\fFN(\tfW(v_{k,j})),\\
&\pi(z_0)=1.
\end{split}\]
Further, $X(2)=\ker(\delta_2|_{W'})$ and $\tilde f_4$ is the restriction of $\tilde f_3$ to $X(2) \oplus W$.
\end{lemma}

\begin{proof}
We first remark that the natural $G$-structure on $\Nnr$ given by the inclusion $G \hookrightarrow \Gal(\Kur/K) \times G$ induces on $\Z_p/p^\omega\Z_p(\chinr)$ the $G$-structure characterized by
\[a \cdot 1 = 1, \quad b \cdot 1 = \chinr(F^{-1}).\]
Note that this is well-defined because $\chinr(F)^d = \chinr(F_N) \equiv 1\pmod{p^\omega}$ by the definition of $\omega$.

The exactness of the bottom $2$-extension is shown in Theorem \ref{FpnZpomega}. For the exactness of the top row we have to verify exactness at $\ZpG z_0$; the rest is clear by the definitions. If we set $H := \langle (a-1)z_0,(u^mb-1)z_0\rangle$, then we must prove $H = \ker(\pi)$. The inclusion $H \sseq \ker(\pi)$ is implied by
\begin{eqnarray*}
  && \pi((a-1)z_0)=(a-1) \cdot 1 =0, \\
  && \pi((u^mb-1)z_0)=(\chinr(F)b-1) \cdot 1 =0.
\end{eqnarray*}
Let $\lambda z_0 \in \ker(\pi)$ with $\lambda \in \ZpG$. Then clearly $\lambda z_0 \equiv xz_0 \pmod H$ for some $x \in \Zp$ with $v_p(x) \ge \omega$. Since $\omega = v_p(1 - \chinr(F_N))$ and $\chinr(F_N) = u^{dm}$, we can write $x =  (u^{dm} - 1)y$ with $y \in \Zp$. Hence $xz_0 = ((u^mb)^d - 1) y z_0 \in H$.

The proof of commutativity is straightforward using $\theta_1^{a-1} = \gamma^{\chinr(F_N)F_N - 1}$ which holds by the definition of $\gamma$. For the convenience of the reader we give the computations
\[\begin{split}\tilde f_2(\delta_2(z_1))&=\tilde f_2((\chinr(F)b-1)z_0)=\frac{(F\times 1)(F^{-1}\times b)[\theta_1^{\chinr(F)},1,1,\dots,1]}{[\theta_1,1,\dots,1]}\\
&=\frac{(F\times 1)[\theta_1,1,1,\dots,1]}{[\theta_1,1,\dots,1]}=((F-1)\times 1)\tilde f_3(z_1),
\end{split}\]

\[\begin{split}
\tilde f_2(\delta_2(z_2))&=\tilde f_2((a-1)z_0)=\frac{(1\times a)[\theta_1,1,1,\dots,1]}{[\theta_1,1,1,\dots,1]}=[\theta_1^{a-1},1,1,\dots,1]\\
&=[\gamma^{\chinr(F_N)F_N-1},1,1,\dots,1]=\frac{[\gamma^{\chinr(F_N)F_N},\gamma,\gamma,\dots,\gamma]}{[\gamma,\gamma,\dots,\gamma]}\\
&=((F-1)\times 1)[\gamma,\gamma,\dots,\gamma]=((F-1)\times 1)\tilde f_3(z_2).
\end{split}\]
\end{proof}
We let $f_4 \colon X(2) \oplus W \lra \calF(\p_N) / \calF(\p_N^{p+1})$ denote the composite of $\tilde f_4$ with the canonical projection.

\begin{lemma}\label{nicerepresentative}
The complex
\[
F^\bullet := [\ker f_4\lra W'\oplus W\lra\ZpG z_0]
\]
with modules in degrees $0,1$ and $2$ is a representative of $M^\bullet(\calL)$ for $\calL = \p_N^{p+1}$.
\end{lemma}

\begin{proof}
By Lemma \ref{deff4} and Lemma \ref{f4surj} we have a quasi-isomorphism of complexes
\[
F^\bullet \lra \left[ \Nnr / \fFN(\calF(\p_N^{p+1})) \lra \Nnr \right].
\]
By Theorem \ref{fundamental 2} the right hand side represents $M^\bullet(\calL)$. 
\end{proof}

Our aim is to use Lemma \ref{nicerepresentative} to compute the Euler characteristic
\[\chi_{\ZpG, \BdRG}(M^\bullet(\calL), \lambda_2^{-1}).\]
We therefore
must compute $\ker f_4$ explicitly. A first step in this direction is the 
 following lemma which is similar to \cite[Lemma~4.1.5]{BC}.

\begin{lemma}\label{X2generators}
We have
\[X(2)=\langle \T_az_2,(a-1)z_1-(u^mb-1)z_2\rangle_\ZpG.\]
\end{lemma}

\begin{proof}
The inclusion ''$\supseteq$'' is clear from the definition of $\delta_2$. For the inverse inclusion we let
\[x=\sum_{i=0}^{p-1}\sum_{j=0}^{d-1}\alpha_{i,j}a^ib^jz_1+
\sum_{i=0}^{p-1}\sum_{j=0}^{d-1}\beta_{i,j}a^ib^jz_2\]
be an element in $X(2)$ with $\alpha_{i,j}, \beta_{i,j} \in \Zp$. From $\delta_2(x) = 0$ we derive
\begin{equation}\label{eqX2*}
u^m\alpha_{i,j-1}-\alpha_{i,j}+\beta_{i-1,j}-\beta_{i,j}=0
\end{equation}
for all $0\leq i<p$ and $0\leq j<d$. Here and in the following we regard all indices as integers modulo $p$ and $d$ respectively.

Taking the sum over the index $i$ we deduce
\begin{equation}
u^m\sum_{i=0}^{p-1}\alpha_{i,j-1} = \sum_{i=0}^{p-1}\alpha_{i,j}\label{eqX2}
\end{equation}
for all $0\leq j<d$. Applying (\ref{eqX2}) repeatedly we get
\[\sum_{i=0}^{p-1}\alpha_{i,0} = u^m\sum_{i=0}^{p-1}\alpha_{i,d-1} = \ldots = u^{dm}\sum_{i=0}^{p-1}\alpha_{i,0}.\]
Since $u^{dm}=\chinr(F_N) \ne 1$ by assumption, we obtain $\sum_{i=0}^{p-1}\alpha_{i,0}=0$. Using (\ref{eqX2}) again, it follows that
\begin{equation}\label{eqX2**}
\sum_{i=0}^{p-1}\alpha_{i,j}=0
\end{equation}
for all $0\leq j<d$. So in particular the above sum does not depend on the choice of $j$, so that we can continue our proof exactly as in \cite[Lemma~4.1.5]{BC}. Note that with the notations of \cite{BC} we would obtain $\mu_0=0$, so that in the present situation two generators are enough. We include the computations for the convenience of the reader.

We want to find $\gamma_{i,j}$ and $\nu_j$ in $\Zp$ such that
\[\begin{split}
x&=\sum_{i=0}^{p-1}\sum_{j=0}^{d-1}\gamma_{i,j}a^ib^j((a-1)z_1-(u^mb-1)z_2))
+\sum_{j=0}^{d-1}\nu_jb^j\T_az_2\\
&=\sum_{i=0}^{p-1}\sum_{j=0}^{d-1}(\gamma_{i-1,j}-\gamma_{i,j})a^ib^jz_1+
   \sum_{i=0}^{p-1}\sum_{j=0}^{d-1}(-u^m\gamma_{i,j-1}+\gamma_{i,j}+\nu_j)a^ib^jz_2.
\end{split}\]
So we need to solve
\begin{equation}\label{eqX2+}
\gamma_{i-1,j}-\gamma_{i,j}=\alpha_{i,j}
\end{equation}
and
\begin{equation}\label{eqX2++}
-u^m\gamma_{i,j-1}+\gamma_{i,j}+\nu_j=\beta_{i,j}.
\end{equation}
We set $\gamma_{i,j} := -\sum_{1\leq\ell\leq i}\alpha_{\ell,j}$ and $\nu_j=\beta_{0,j}$. Then equation (\ref{eqX2+}) is clearly satisfied for $i=1,\dots,p-1$ and for $i=0$ it is immediately implied by (\ref{eqX2**}). Equality (\ref{eqX2++}) is proved by an easy induction on $i$ using (\ref{eqX2*}).
\end{proof}

Still following the conventions in \cite{BC} we use the notation $[x]$ to denote the element $[x,x,\dots,x]\in \Nnr$ or even in $\Nnr/\fFN(\calF(\p_N^{p+1}))$.

Now we need to evaluate $f_4$ (or equivalently $\tilde f_3$) at the generators of $W(2)$ determined in the previous lemma, as we did in \cite[Lemma~4.1.6]{BC}.
\begin{lemma}\label{calcf4}
We have
\begin{eqnarray*}
&& \tilde f_3((a-1)z_1-(u^mb-1)z_2)=[\gamma^{1-b}], \\
&& \tilde f_3(\T_az_2)=[\gamma^{\T_a}].
\end{eqnarray*}
\end{lemma}
\begin{proof}
Using (\ref{rule 1}), (\ref{rule 2}), $\theta_1^{a-1} = \gamma^{\chinr(F_N)F_N -1 }$, $\theta_1^b = \theta_1$ and $\chinr(F^{-1})=u^{-m}$ we obtain
\[\begin{split}
\tilde f_3((a-1)z_1-(u^mb-1)z_2)&=
\frac{[\theta_1^{a-1},1,\dots,1][\gamma,\dots,\gamma]}{(F\times 1)(F^{-1}\times b)[\gamma^{u^m},\dots,\gamma^{u^m}]}\\
&=\frac{[\theta_1^{a-1}\gamma,\gamma,\dots,\gamma]}{(F\times 1)[\gamma^b,\dots,\gamma^b]}=\frac{[\theta_1^{a-1}\gamma,\gamma,\dots,\gamma]}{[\gamma^{\chinr(F_N)F_Nb},\gamma^b,\dots,\gamma^b]}\\
&=\frac{[\theta_1^{a-1}\gamma,\gamma,\dots,\gamma]}{[\theta_1^{(a-1)b}\gamma^b,\gamma^b,\dots,\gamma^b]}=
[\gamma^{1-b},\dots,\gamma^{1-b}].
\end{split}\]
The second equality is obvious by the definitions.
\end{proof}

The following results are the analogues of \cite[Lemma~4.2.1 and Lemma~4.2.2]{BC}.

\begin{lemma}\label{deft1}
Let $\tilde m$ be an integer such that $m\tilde m \equiv 1 \pmod d$.
Set 
\[\tilde t_1 := (a-1)z_1-(u^mb-1)z_2+
\left(\sum_{i=2}^{m}\alpha_{i}(u^mb)^{1-(i-2)\tilde m}+\left(\alpha_1-\sum_{i=2}^{m}\alpha_{i}\right)(u^mb)^{\tilde m}\right)w_0.\]
Then there exists $y_1\in W_{\geq 1}$, such that $t_1 := \tilde t_1 + y_1 \in \ker(f_4)$.
\end{lemma}

\begin{proof}
Recalling that $\gamma$ has been chosen such that $\gamma \equiv 1 + x_2\theta_1 \pmod{\p_{N_0}^2}$ and following the calculations in the proof of \cite[Lemma~4.2.1]{BC} we obtain
\[\gamma^{1-b}\equiv (1+x_2\theta_1)^{1-b}=1-\sum_{i=2}^{m} \alpha_{i}\theta^{b^{1-(i-2)\tilde m}}-\alpha_1\theta^{b^{\tilde m}}+\sum_{i=2}^{m}\alpha_{i}\theta^{b^{\tilde m}}\pmod{\p_{N_0}^2}.\]
By Lemma \ref{UL1hat} and Lemma \ref{E lemma}
\begin{equation}\label{f equiv}
\fFN(X) \equiv 1+E^{-1}X + \deg \ge 2 \pmod{p\calO_{K'}}
\end{equation}
and so by Lemma \ref{calcf4} we deduce
\[\tilde f_4((a-1)z_1-(u^mb-1)z_2)\equiv -\sum_{i=2}^{m} \alpha_{i}E\theta^{b^{1-(i-2)\tilde m}}-\alpha_1E\theta^{b^{\tilde m}}+
\sum_{i=2}^{m}\alpha_{i}E\theta^{b^{\tilde m}}\pmod{\p_{N_0}^2}.\]
Since by Lemma \ref{E lemma}, for any integer $i$,
\[(u^{m}b)^i\cdot (E\theta)=u^{mi}E^{b^i}\theta^{b^i}\equiv u^{mi}u^{-mi}E\theta^{b^i}\equiv E\theta^{b^i}\pmod{\p_{N_0}^2},\]
we have
\[\begin{split}
&\tilde f_4\left(\sum_{i=2}^{m}\alpha_{i}(u^mb)^{1-(i-2)\tilde m}+\left(\alpha_1-\sum_{i=2}^{m}\alpha_{i}\right)(u^mb)^{\tilde m}\right)w_0\\
&\qquad\equiv \left(\sum_{i=2}^{m}\alpha_{i}(u^mb)^{1-(i-2)\tilde m}+\left(\alpha_1-
\sum_{i=2}^{m}\alpha_{i}\right)(u^mb)^{\tilde m}\right)\cdot (E\theta)\pmod{\p_{N_0}^2}\\
&\qquad\equiv\sum_{i=2}^{m}\alpha_{i}E\theta^{b^{1-(i-2)\tilde m}}+\alpha_1E\theta^{b^{\tilde m}}-
\sum_{i=2}^{m}\alpha_{i}E\theta^{b^{\tilde m}}\pmod{\p_{N_0}^2}
\end{split}\]
and we can conclude that $\tilde f_4(\tilde t_1)\equiv 0\pmod{\p_{N_0}^2}$. Therefore $f_4(\tilde t_1) \in \calF(\p_N^{2})/\calF(\p_N^{p+1})$ and by Lemma \ref{f4surj} there exists $y_1\in W_{\geq 1}$ such that 
$f_4(-y_1)=f_4(\tilde t_1)$, i.e. $\tilde t_1+y_1\in\ker (f_4)$.
\end{proof}

Note that $t_1$ is not exactly the same as in \cite{BC} since there is a $u^m$ factor which appears at all the occurrences of $b$.

\begin{lemma}\label{deft2}
The element
\[t_2:=\T_az_2-\beta w_{p-1}\text{ with }\beta=\begin{cases}\alpha_1&\text{if }m=1\\ \alpha_2&\text{if }m>1\end{cases}\]
is in the kernel of $f_4$. 
\end{lemma}
\begin{proof}
As for the previous lemma, we can follow the calculations in the proof of \cite[Lemma~4.2.2]{BC}:
\[\gamma^{T_a}\equiv \norm_{\Nur/\Kur}(1+x_2\theta_1)\equiv 1+x_2p-x_2^p\alpha_1^{1-p}p\equiv 1+\beta\theta_2p\pmod{\p_{N_0}^{p+1}}.\]
As in the proof of the previous lemma, we deduce that
\[\tilde f_4(\T_az_2)\equiv E\beta \theta_2 p\pmod{\p_{N_0}^{p+1}}.\]
Next using \cite[Lemma 3.2.4]{BC} we calculate
\[\tilde f_4(\beta w_{p-1})\equiv E\beta(a-1)^{p-1}\theta\equiv E\beta \theta_2p\pmod{\p_{N_0}^{p+1}}.\]
Hence we easily conclude since $\calF(\p_{N_0}^{p}) / \calF(\p_{N_0}^{p+1}) \cong \p_{N_0}^{p} / \p_{N_0}^{p+1}$.
\end{proof}

\begin{lemma}
The element
\[r_1=\T_a t_1 + (u^mb-1)t_2\]
belongs to $\ker f_4\cap W$ and its $\alpha_1w_0$-component is $(u^mb)^{\tilde m} \T_a$.
\end{lemma}

\begin{proof}
Straightforward calculations as in the proof of \cite[Lemma~4.2.5]{BC}.
\end{proof}

Following the strategy in \cite{BC} we consider the $\alpha_iw_{p-1}$-components, $i=1,\dots,m$, of  $t_2$ and $r_2,\dots,r_m$.
We write these components as the columns of a $m \times m$-matrix $\calM$. If $m > 1$, then we obtain
\[\calM=\begin{pmatrix}
0&0&0&\cdots&0&0&(u^mb)^{-\tilde m}\\
-1&-1&0&\cdots&0&0&-(u^mb)^{-\tilde m}\\
0&(u^mb)^{-\tilde m}&-1&\cdots&0&0&-(u^mb)^{-\tilde m}\\
0&0&(u^mb)^{-\tilde m}&\cdots&0&0&-(u^mb)^{-\tilde m}\\
\vdots&\vdots&\vdots&\ddots&\vdots&\vdots&\vdots&\\
0&0&0&\cdots&(u^mb)^{-\tilde m}&-1&-(u^mb)^{-\tilde m}\\
0&0&0&\cdots&0&(u^mb)^{-\tilde m}&-1-(u^mb)^{-\tilde m}
\end{pmatrix}.\]
If $m=1$, then the matrix is determined by $t_2$ and recalling the definition from Lemma \ref{deft2} we get $\calM = (-1)$.

\begin{lemma}\label{detMomega>0}
The determinant of $\calM$ is $(-1)^{m}(u^mb)^{-\tilde m(m-1)}$.
\end{lemma}

\begin{proof}
This is an easy calculation. 
\end{proof}

The next proposition is the analogue of \cite[Prop.~4.2.10]{BC}.
 
\begin{lemma}\label{pm+1generators}
The $pm+1$ elements $t_1,t_2$, $r_k$, for $k=2,\dots,m$, $s_{j,k}$ for $0\leq j\leq p-2$, $1\leq k\leq m$ constitute a $\ZpG$-basis of  $\ker f_4$.
\end{lemma}

\begin{proof}
We first note that $s_{p-1,k} \in \langle t_2, r_2, \ldots, r_m \rangle_\ZpG$ which follows as in the proof of \cite[Lemma~4.2.9]{BC}. The main input here is the fact that the matrix $\calM$ is invertible. 

Moreover, by Lemma \ref{X2generators}, $t_1$ and $t_2$ generate $(X(2)+W)/W$, which therefore must coincide with $(\ker f_4+W)/W$, and the same proof as for \cite[Lemma~4.2.7]{BC} shows that the elements $r_k$ and $s_{j,k}$, $0 \le j \le p-1, 1 \le k \le m$, generate $\ker(f_4) \cap W$. We can thus argue as in the proof of \cite[Prop.~4.2.10]{BC}.
\end{proof}

By Lemma \ref{nicerepresentative} we have the equality 
\[\chi_{\ZpG, \BdR[G]}(M^\bullet(\calL), 0) = \chi_{\ZpG, \BdR[G]}(F^\bullet, 0).\]
Since all cohomology groups are torsion we have a short exact sequence of $\QpG$-spaces
\[0 \lra \left( \ker f_4 \right)_\QpG \stackrel{\iota}\lra \left( W' \oplus W \right)_\QpG  \stackrel{\delta_2}\lra \QpG z_0 \lra 0,\]
where $\iota$ denotes the inclusion. By the definition of the (old) refined Euler characteristic we have
\[\chiold_{\ZpG, \BdR[G]} (F^\bullet[1], 0) = [\ker f_4 \oplus \ZpG z_0, \iota \oplus \sigma, W'\oplus W],\]
where $\sigma$ is a $\QpG$-equivariant splitting of $\delta_2$. By (\ref{old and new}) we obtain
\[\chi_{\ZpG, \BdR[G]} (F^\bullet[1], 0) = 
-\chiold_{\ZpG, \BdR[G]} (F^\bullet[1], 0) - \partial_{\ZpG, \BdR[G]}^1(B^\od(F^\bullet[1]_{\BdR[G]}), -\id).\]
Note that $F^\bullet[1]$ has non-trivial cohomology only in degrees $0$ and $1$ and both cohomology groups are torsion. By \cite[Lemma~6.3]{BrBuAdditivity} we therefore obtain
\[\partial_{\ZpG, \BdR[G]}^1(B^\od(F^\bullet[1]_{\BdR[G]}), -\id) =
\partial_{\ZpG, \BdR[G]}^1(H^{1}(F^\bullet[1])_{\BdR[G]}, -\id) = 0.\]
Furthermore, we recall that $\chi_{\ZpG, \BdR[G]} (F^\bullet, 0) =-\chi_{\ZpG, \BdR[G]} (F^\bullet[1], 0)$.
In conclusion, we have derived the equality
\[\chi_{\ZpG, \BdR[G]} (F^\bullet, 0) = [\ker f_4 \oplus \ZpG z_0, \iota \oplus \sigma, W'\oplus W].\]

Taking the $\ZpG$-basis of $\ker f_4$ determined in Lemma \ref{pm+1generators} and the obvious $\ZpG$-basis of $W'\oplus W$ we can represent the inclusion $\iota$ by the following matrix:
\[\mm=\begin{pmatrix}
a-1&0&0&0&0&\cdots&0&0\\
 1-u^mb&\T_a&0&0&0&\cdots&0&0\\
v&0&\T_a\tilde I&(a-1)I&0&\cdots&0&0\\
*&0&0&-I&(a-1)I&\cdots&0&0\\
*&0&0&*&-I&\cdots&0&0\\
\vdots&\vdots&\vdots&\vdots&\vdots&\ddots&\vdots&\vdots\\
*&0&0&*&*&\cdots&-I&(a-1)I\\
*&\calM_1&\tilde \calM&*&*&\cdots&*&-I
\end{pmatrix},\]
where $\tilde\calM$ is the matrix $\calM$ without the first column $\calM_1$ (which in the case $m>1$ is just $-e_2$), $\tilde I$ is defined analogously from the identity matrix $I$ and $v$ is a vector whose first component is $(u^mb)^{\tilde m}$. Note that this looks very similar to the corresponding matrix in \cite{BC}, up to some twists by powers of $u$. Next we need to choose a $\sigma$. We set
\[\sigma(z_0)=\frac{1}{u^{dm}-1} \left( \sum_{i=0}^{d-1}(u^{m}b)^i \right)z_1,\]
and easily check that $\delta_2\circ\sigma=\id$.

By the above discussion the refined Euler characteristic $\chi_{\ZpG, \BdR[G]}(M^\bullet(\calL), 0)$ is represented by the determinant of the matrix $(w,\mm) \in \mathrm{Gl}_{mp+2}(\QpG)$ which represents the $\QpG$-equivariant map $\iota + \sigma$ with respect to the chosen $\ZpG$-basis. Here $w$ is the column vector
\[w=\left(\frac{\sum_{i=0}^{d-1}(u^{m}b)^i}{u^{dm}-1}, 0, \ldots, 0\right)^t.\]
We get
\[(w,\mm)=
\begin{pmatrix}
\frac{\sum_{i=0}^{d-1}(u^{m}b)^i}{u^{dm}-1}&a-1&0&0&0&0&\cdots&0&0\\
0& \!\!\!1-u^mb\!\!\!&\T_a&0&0&0&\cdots&0&0\\
0&v&0&\T_a\tilde I&(a-1)I&0&\cdots&0&0\\
0&*&0&0&-I&\!\!(a-1)I\!\!&\cdots&0&0\\
0&*&0&0&*&-I&\cdots&0&0\\
\vdots&\vdots&\vdots&\vdots&\vdots&\vdots&\ddots&\vdots&\vdots\\
0&*&0&0&*&*&\cdots&-I&(a-1)I\\
0&*&\calM_1&\tilde \calM&*&*&\cdots&*&-I
\end{pmatrix}.
\]
We recall that every irreducible character $\psi$ of $G = \Gal(N/K)$ decomposes as $\psi = \phi\chi$ where $\chi$ is an irreducible character of $\langle a \rangle$ and $\phi$ an irreducible character of $\langle b \rangle$. We first compute
\[\chi\phi\left( \frac{\sum_{i=0}^{d-1}(u^{m}b)^i}{u^{dm}-1} \right)=
\frac{u^{dm}-1}{(u^m\chi\phi(b)-1)(u^{dm}-1)}=\frac{1}{u^m\chi\phi(b)-1},\]
and note that $u^m\chi\phi(b) \ne 1$ because otherwise we would have $\chinr(F_N)=u^{dm}=1$. For the computation of the determinant we distinguish two cases.

\medskip

\noindent
{\bf Case 1:} $\chi=1$.\\
Here $ (\chi\phi)((w,\mm))=\phi((w,\mm))$ is of the form
\[\begin{pmatrix}
\frac{1}{u^m\phi(b)-1}&0&0&0&0&0&\cdots&0&0\\
0&1-u^m\phi(b)&p&0&0&0&\cdots&0&0\\
0&\phi(v)&0&p\tilde I&0&0&\cdots&0&0\\
0&*&0&0&-I&0&\cdots&0&0\\
0&*&0&0&*&-I&\cdots&0&0\\
\vdots&\vdots&\vdots&\vdots&\vdots&\vdots&\ddots&\vdots&\vdots\\
0&*&0&0&*&*&\cdots&-I&0\\
0&*&\phi(\calM_1)&\phi(\tilde\calM)&*&*&\cdots&*&-I
\end{pmatrix},\]
where we recall that the first component of the vector $v$ is $(u^mb)^{\tilde m}$. The determinant is
\[\frac{-u^{m\tilde m}\phi(b)^{\tilde m}p^m(-1)^{m(p-1)}}{u^m\phi(b)-1}=\frac{-u^{m\tilde m}\phi(b)^{\tilde m}p^m}{u^m\phi(b)-1}.\]

\medskip

\noindent
{\bf Case 2:} $\chi\neq 1$.\\
The matrix $ (\chi\phi)((w,\mm))$ is here given by
\[\begin{pmatrix}
\frac{1}{u^m\phi(b)-1}&\chi(a)-1&0&0&0&\cdots&0&0\\
0& \!1-u^m\phi(b)\!&0&0&0&\cdots&0&0\\
0&\chi\phi(v)&0&(\chi(a)-1)I&0&\cdots&0&0\\
0&*&0&-I&(\chi(a)-1)I&\cdots&0&0\\
0&*&0&*&-I&\cdots&0&0\\
\vdots&\vdots&\vdots&\vdots&\vdots&\ddots&\vdots&\vdots\\
0&*&0&*&*&\cdots&-I&(\chi(a)-1)I\\
0&*&\!\!\!\chi\phi(\calM)\!\!\!&*&*&\cdots&*&-I
\end{pmatrix}.\]
Using Lemma \ref{detMomega>0} we compute for the determinant
\[\begin{split}
&\frac{1-u^m\phi(b)}{u^m\phi(b)-1}(-1)^{(p-1)m^2}\det(\chi\phi(\calM))(\chi(a)-1)^{m(p-1)}\\
&\qquad\qquad=(-1)^{m-1}(u^m\phi(b))^{-\tilde m(m-1)}(\chi(a)-1)^{m(p-1)}.
\end{split}\]

We summarize the above discussion in

\begin{prop}\label{prop 118}
Assume that $\omega > 0$. For $\calL = \p_N^{p+1}$ the element 
\[\chi_{\ZpG, \BdR[G]}(M^\bullet(\calL), \lambda_2^{-1}) \in K_0(\ZpG, \BdR[G])\]
is contained in $K_0(\ZpG, \QpG)$ and represented by
$\epsilon \in \QpG^\times$ where
\[\epsilon_{\chi\phi}^{}=
\begin{cases}\frac{-u^{m\tilde m}\phi(b)^{\tilde m}p^m}{u^m\phi(b)-1}&\text{if $\chi=\chi_0$}\\
(-1)^{m-1}(u^m\phi(b))^{-\tilde m(m-1)}(\chi(a)-1)^{m(p-1)}&\text{if $\chi\neq\chi_0$.}
\end{cases}\]
\end{prop}
\end{subsection}
\end{section}

\begin{section}{Proof of Theorem \ref{main theorem}}\label{proof of theorem}
We have to show that the element
\[\tilde R_{N/K} = C_{N/K} + \Ucris + \partial_{\ZpG, \BdR[G]}^1(t)  +  \partial_{\ZpG, \BdR[G]}^1(\epsilon_D(N/K, V))\in K_0(\ZpG, \tilde\Omega)\]
becomes trivial in $K_0(\tilde\Lambda, \tilde\Omega)$. By (\ref{CNK44}) we get
\[\tilde R_{N/K} = \Ucris+\partial_{\ZpG, \BdR[G]}^1(\theta) +  \partial_{\ZpG, \BdR[G]}^1(\epsilon_D(N/K, V)) - \chi_{\ZpG,\BdR[G]}(M^\bullet(\calL), \lambda_2^{-1}).\]

By Proposition \ref{eps coroll} together with Lemma \ref{yal 3} the element 
\[
\left( u^{-mm_{\chi\phi}} \cdot \tau_\Qp(\Ind_{K/\Qp}\chi\phi)^{-1} \right)_{\chi\phi}
\]
with
\[m_{\chi\phi} =
\begin{cases}
0 & \text{if $\chi = \chi_0$} \\ 2 & \text{if $\chi \ne \chi_0$}
\end{cases}\]
is a representative of $\partial_{\ZpG, \BdR[G]}^1(\epsilon_D(N/K, V))$. We note that by the proof of \cite[Prop.~5.2.1]{BC} we obtain  $\tau_\Qp(\Ind_{K/\Qp}\chi\phi) = \tau_K(\chi\phi)$. Since we intend  to apply \cite[Prop.~5.1.5]{BC} we will use the integral normal basis generator $p^2 \alpha_M\theta_2$ of $\calL = \p_N^{p+1}$ which we constructed in \cite[Sec.~5.1]{BC}. Then Proposition~5.2.1 of loc.cit. implies that $\partial_{\ZpG, \BdR[G]}(\theta) +  \partial_{\ZpG, \BdR[G]}(\epsilon_D(N/K, V))$ is represented by
\[\frac{\theta_\chi u^{-mm_{\chi\phi}}}{\tau_K(\chi\phi)}=
\begin{cases}
\frd_K^{1/2} \calN_{K/\Qp}(\theta_2|\phi) p^{2m} & \text{if $\chi = \chi_0$} \\
\frd_K^{1/2} \calN_{K/\Qp}(\theta_2|\phi) p^{m} \chi(4)\phi(b^{-2}) u^{-2m} & \text{if $\chi \ne \chi_0$.}
\end{cases}\]

We want to point out that this is a crucial step in the proof and relies on one of the main results of \cite{PickettVinatier} which was basic for the proof of \cite[Prop.~5.1.5]{BC}. Note that this is the inverse of the result in \cite[Prop.~5.2.1]{BC} after substituting $\phi(b)$ with $u\phi(b)$ everywhere.

If $\chinr|_{G_N}=1$ then the calculation of $\chi_{\ZpG,\BdR[G]}(M^\bullet(\calL), \lambda_2^{-1})$ corresponds to the calculation of $E(\exp(\calL))_p$ in \cite{BC}: the representative of the local fundamental class has to be twisted by $\chinr$, then starting from \cite[Lemma~4.1.2]{BC} one has to substitute $b$ with $u^mb$ everywhere. Since also $\Ucris$ corresponds to $M_{N/K}$, just substituting $b$ with $u^mb$, the same calculations of \cite[Sec.~6]{BC} can be used to prove Theorem \ref{main theorem} in the case $\chinr|_{G_N}=1$. So the interesting case is when $\chinr|_{G_N}\neq 1$, which we will assume from now on.

By Proposition \ref{Ucris prop} the element $\Ucris$ is represented by $\ucris \in \QpG^\times$ where
\[\ucris_{,\chi\phi} =
\begin{cases}
\frac{1 - p^{-m}u^{-m} \phi(b)^{-1}}{1 - u^m \phi(b)} & \text{if $\chi = \chi_0$} \\
1& \text{if $\chi \ne \chi_0$.} 
\end{cases}\]

We first deal with the case $\omega = 0$. Inserting the result of Proposition \ref{prop 117} we obtain that $\tilde R_{N/K}$ is 
represented by $\tilde r \in \tilde \Omega$ where
\[\begin{split}
\tilde r_{\chi\phi}&=
\begin{cases}\frac{\frd_K^{1/2}\norm_{K/\Qp}(\theta_2|\phi)p^{2m}(1-p^{-m}u^{-m}\phi(b^{-1}))}{p^m(1-u^m\phi(b))}&\text{if $\chi=\chi_0$}\\
\frac{\frd_K^{1/2}\norm_{K/\Qp}(\theta_2|\phi)p^m\chi(4)\phi(b)^{-2} u^{-2m}}{(-1)^{m-1}(u^{-m}\phi(b)^{-1}-1)(\chi(a)-1)^{m(p-1)}}&\text{if $\chi\neq\chi_0$}\end{cases}\\
&=\begin{cases}\frd_K^{1/2}\norm_{K/\Qp}(\theta_2|\phi)u^{-m}\phi(b^{-1})\frac{(p^m u^m\phi(b)-1)}{(1-u^m\phi(b))}&\text{if $\chi=\chi_0$}\\
-\frd_K^{1/2}\norm_{K/\Qp}(\theta_2|\phi)\left(\frac{-p}{(\chi(a)-1)^{p-1}}\right)^m\frac{\chi(4)\phi(b)^{-2}u^{-2m}}{1-u^m\phi(b)}u^m\phi(b)&\text{if $\chi\neq\chi_0$}\end{cases}\\
&=\begin{cases}-\frac{\frd_K^{1/2}\norm_{K/\Qp}(\theta_2|\phi)}{1-u^m\phi(b)}u^{-m}\phi(b)^{-1}(1-p^mu^{m}\phi(b))&\text{if $\chi=\chi_0$}\\
-\frac{\frd_K^{1/2}\norm_{K/\Qp}(\theta_2|\phi)}{1-u^m\phi(b)}u^{-m}\phi(b)^{-1}\left(\frac{-p}{(\chi(a)-1)^{p-1}}\right)^m\chi(4)&\text{if $\chi\neq\chi_0$.}\end{cases}
\end{split}\]

Similarly to \cite[Proof of Theorem~1]{BC}, we define $W_{\theta_2}\in\calO_p^t[G]\in\overline{\Zpnr}[G]^\times$ by
\[\chi\phi(W_{\theta_2})=\frd_K^{1/2}\norm_{K/\Qp}(\theta_2|\phi).\]

By a straightforward computation (see \cite[Lemma~3.2.3]{BC}) there exists a unit $\tilde u\in\Z_p[a]$ such that $p(1-e_a)=(a-1)^{p-1}\tilde u(1-e_a)$ and the augmentation of $\tilde u$ is $(p-1)!\equiv -1\pmod{p}$. Evaluating at $\chi\neq\chi_0$ we get $p=(\chi(a)-1)^{p-1}\chi(\tilde u)$.

Then by the above calculations we have
\[\tilde r=-\frac{W_{\theta_2}u^{-m}b^{-1}}{1-u^mb}\left((1-p^mu^{m}b)e_a+(-\tilde u)^m \sigma_4 (1-e_a)\right).\]
By \cite[Sec.I, Prop.~4.3]{Froehlich83} the element $W_{\theta_2}$ is a unit. Furthermore, recalling that $v_p(1-u^{md_{N/K}})=\omega=0$,
\[\frac{1}{1-u^mb}=\frac{1+u^mb+\dots+u^{m(d_{N/K}-1)}b^{d_{N/K}-1}}{1-u^{md_{N/K}}}\in\overline{\Zpnr}[G].\]
Since its inverse $1-u^mb$ is also in $\overline{\Zpnr}[G]$, it is a unit. So we just need to consider
\[\tilde{\tilde r}=(1-p^m u^m b)e_a+(-\tilde u)^m \sigma_4 (1-e_a).\]
We first show that $\tilde{\tilde r}$ is contained in $\overline{\Zpnr}[G]$, which is true if and only if 
\[e_a+(-\tilde u)^m \sigma_4 (1-e_a)\]
is in $\overline{\Zpnr}[G]$. This is equivalent to
\[1\equiv(-\chi(\tilde u))^m \chi(\sigma_4) \pmod{1-\zeta_p}\]
for any non-trivial character $\chi$, which is straightforward since $\chi(\tilde u)\equiv -1\pmod{1-\zeta_p}$ and $\chi(\sigma_4)\equiv 1\pmod{1-\zeta_p}$. Note that in the present situation this is not enough to conclude the proof of our theorem as it was in \cite{BC}, since the conjecture we are considering is not known to be true over maximal orders. Therefore we need to prove also that $\tilde{\tilde r}^{-1}$ has integral coefficients. Note that
\[\frac{1}{1-p^m u^m\phi(b)}\\
=\frac{1+p^m u^m\phi(b)+\dots+(p^m u^m\phi(b))^{d_{N/K}-1}}{1-(p^mu^m)^{d_{N/K}}}.\]
Therefore we have
\[\tilde{\tilde r}^{-1}=\frac{1+p^m u^m b+\dots+(p^m u^m b)^{d_{N/K}-1}}{1-(p^m u^m)^{d_{N/K}}}e_a+(-\tilde u)^{-m}\sigma_4^{-1}(1-e_a).\]
The integrality of this term is obviously equivalent to the integrality of
\[\frac{1}{1-(p^m u^m)^{d_{N/K}}}e_a+(-\tilde u)^{-m}\sigma_4^{-1}(1-e_a),\]
which, in turn, is equivalent to the congruence
\[\frac{1}{1-(p^m u^m)^{d_{N/K}}}\equiv (-\chi(\tilde u))^{-m}\chi(\sigma_4)^{-1}\pmod{1-\zeta_p}.\]
This is true since both sides are congruent to $1$. This concludes our proof in the case $\omega=0$.

In the case $\omega > 0$ we insert the result of Proposition \ref{prop 118} in the formula for $\tilde r$ and, with the same notation as in the case $\omega=0$, we obtain
\[\begin{split}
\tilde r_{\chi\phi}&=
\begin{cases}\frac{\frd_K^{1/2}\norm_{K/\Qp}(\theta_2|\phi)p^{2m}(u^m\phi(b)-1)(1-p^{-m}u^{-m}\phi(b^{-1}))}{-u^{m\tilde m}\phi(b)^{\tilde m}p^m(1-u^m\phi(b))}&\text{if $\chi=\chi_0$}\\
\frac{\frd_K^{1/2}\norm_{K/\Qp}(\theta_2|\phi)p^m\chi(4)\phi(b)^{-2} u^{-2m}}{(-1)^{m-1}(u^m\phi(b))^{-\tilde m(m-1)}(\chi(a)-1)^{m(p-1)}}&\text{if $\chi\neq\chi_0$}\end{cases}\\
&=\begin{cases}\frd_K^{1/2}\norm_{K/\Qp}(\theta_2|\phi)u^{-m}\phi(b^{-1})\frac{(p^mu^m\phi(b)-1)}{u^{m\tilde m}\phi(b)^{\tilde m}}&\text{if $\chi=\chi_0$}\\
-\frd_K^{1/2}\norm_{K/\Qp}(\theta_2|\phi)\left(\frac{-p}{(\chi(a)-1)^{p-1}}\right)^m\frac{\chi(4)\phi(b)^{-2}u^{-2m}}{(u^m\phi(b))^{-\tilde m(m-1)}}&\text{if $\chi\neq\chi_0$}\end{cases}\\
&=\begin{cases}-\chi\phi(W_{\theta_2})\phi(b)^{-1-\tilde m}u^{-m-m\tilde m}(1-p^m u^{m}\phi(b))&\text{if $\chi=\chi_0$}\\
-\chi\phi(W_{\theta_2})\phi(b)^{-1-\tilde m}u^{-m-m\tilde m}\left(\frac{-p}{(\chi(a)-1)^{p-1}}\right)^m\chi(4)u^{m^2\tilde m-m}&\text{if $\chi\neq\chi_0$.}
\end{cases}
\end{split}\]

Hence we have
\[\tilde r=W_{\theta_2}b^{1-\tilde m}u^{-m-m\tilde m}\left((1-p^m u^{m}b)e_a+(-\tilde u)^m\sigma_4u^{m^2\tilde m-m}(1-e_a)\right).\]
As above we have to prove that
\[\tilde{\tilde r}=(1-p^m u^{m}b)e_a+(-\tilde u)^m\sigma_4u^{m^2\tilde m-m}(1-e_a)\]
is a unit. Note that this is the same as the element we got in the case $\omega>0$, up to the factor $u^{m^2\tilde m-m}$ in the $1-e_a$-component. Since we are in the case $\omega>0$ we have $u^{dm} \equiv 1 \pmod{1 - \zeta_p}$. Furthermore, since $m\tilde m \equiv 1 \pmod{d}$ we derive $u^{m^2\tilde m} \equiv u^m \pmod{1-\zeta_p}$, i.e. $u^{m^2\tilde m-m} \equiv 1\pmod{1-\zeta_p}$. Using this fact, the proof for $\omega=0$ is given by the same calculations as for $\omega>0$.
\qed
\end{section}

\appendix
\begin{section}{\texorpdfstring{A more detailed computation of $\Ucris$}{A more detailed computation of Ucris}}
The aim of this appendix is to give some more details for the computation of $\Ucris$. By (\ref{Ucriskappabeta}) we need to compare $\beta'$ and $\kappa$. In the definition of both trivialisations one starts with the same isomorphism
\[\begin{split}\DEPV&=[R\Gamma(N,V)]\otimes [\Ind_{N/\Qp}(V)]\\
&\qquad\lra [H^1(N,V)]^{-1}\otimes [H^2(N,V)]\otimes [\Ind_{N/\Qp}(V)].\end{split}\]
We set
\[\begin{split}&A=[\Dcris^N(V)]\otimes [\Dcris^N(V)]^{-1},\\
&B=[H^1(N,V)]^{-1}\otimes [t_V(N)]\otimes [H^1(N/V)/H^1_f(N,V)],\\
&C=[H^1(N/V)/H^1_f(N,V)]^{-1}\otimes [\Dcris^N(V^*(1))]\otimes[\Dcris^N(V^*(1))]^{-1}\otimes [H^2(N,V)],\\
&D=[t_V(N)^{-1}]\otimes [\DdR^N(V)],\\
&E=[\DdR^N(V)]^{-1}\otimes [\Ind_{N/\Qp}(V)].
\end{split}\]
Then there is a natural isomorphism
\[[H^1(N,V)]^{-1}\otimes [H^2(N,V)]\otimes [\Ind_{N/\Qp}(V)]\to A\otimes B\otimes C\otimes D\otimes E,\]
where we used the defining property of right inverses and the commutativity constraint to identify $[t_V(N)]^{-1}\otimes [H^2(N,V)]$ with $[H^2(N,V)]\otimes [t_V(N)]^{-1}$. Up to now all the calculations are in common for $\beta'$ and $\kappa$. Now we can compare the two trivialisations term by term:
\begin{enumerate}
\item[$A\to 1$:] This is determined by $1-\phi$ in $\beta'$ and by $1$ in $\kappa$.
\item[$B\to 1$:] In both cases this isomorphism comes from the exponential map $\exp_V$, see for example the exact sequence (\ref{degenerate 4}).
\item[$C\to 1$:] For $\beta'$ we use (\ref{degenerate 3}), while for $\kappa$ we first identify $C$ with $[H^1(N/V)/H^1_f(N,V)]^{-1}\otimes [H^2(N,V)]$ and then consider the valuation map $\nu_N$.
\item[$D\to 1$:] This comes from (\ref{tangent ses}), recalling Lemma \ref{degenerate}.
\item[$E\to 1$:] This comes from the $\comp_V$ map, for both $\beta'$ and $\kappa$, but in the definition of $\beta'$ there also appears a multiplication by $t^{t_H(V)}=t$.
\end{enumerate}
So the difference between $\beta'$ and $\kappa$ lies in $A\to 1$ and $C\to 1$ and in the multiplication by $t$ in $E\to 1$ and leads to (\ref{yar 4}) when $\chinr|_{G_N}\neq 1$ and (\ref{yar 5}) when $\chinr|_{G_N}=1$.

\end{section}

\end{document}